\documentclass{amsart}

\usepackage{amssymb} \usepackage{amsfonts} \usepackage{amsmath}
\usepackage{amsthm} \usepackage{epsfig} 
\usepackage{color}
\usepackage{amscd}
\usepackage[all]{xy}

\newtheorem{lemma}{Lemma}[section]
\newtheorem{teo}[lemma]{Theorem}
\newtheorem{prop}[lemma]{Proposition}
\newtheorem{cor}[lemma]{Corollary} 
\newtheorem{conj}[lemma]{Conjecture}

\theoremstyle{definition}
\newtheorem{defn}[lemma]{Definition}

\newtheorem{quest}[lemma]{Question}

\newtheorem{example}[lemma]{Example}

\theoremstyle{remark}
\newtheorem{rem}[lemma]{Remark}

\newcommand{\matZ} {\ensuremath {\mathbb{Z}}}

\newcommand{\matCP} {\ensuremath {\mathbb{CP}}}

\newcommand{\calS} {\ensuremath {\mathcal{S}}}
\newcommand{\calM} {\ensuremath {\mathcal{M}}}
\newcommand{\calC} {\ensuremath {\mathcal{C}}}

\newcommand{\calD} {\ensuremath {\mathcal{D}}}

\newcommand{\calP} {\ensuremath {\mathcal{P}}}

\newcommand{\calG}{\ensuremath {\mathcal{G}}}

\newcommand{\nota} [1] {\caption{\footnotesize{#1}}}

\newcommand{\interior}[1]{{\rm int}(#1)}
\newfont{\Got}{eufm10 scaled 1200}

\newcommand{\dimo}[1]{\vspace{2pt}\noindent\textit{Proof of \ref{#1}}.\ }
\newcommand{\finedimo}{{\hfill\hbox{$\square$}\vspace{2pt}}}
\newcommand{\cref}{c^{\rm ref}}
\newcommand{\isom}{\cong}

\newcommand{\timtil}{\begin{picture}(12,12)
\put(2,0){$\times$}\put(2,4.5){$\sim$}\end{picture}}

\newcommand{\piu}{
\begin{picture}(8,8)
\put(4,0){\line(0,1){8}}
\put(0,4){\line(1,0){8}}
\end{picture}
}
\newcommand{\hacca}{
\begin{picture}(8,8)
\put(3,0){\line(0,1){4}}
\put(5,8){\line(0,-1){4}}
\put(0,4){\line(1,0){8}}
\end{picture}
}

\author{Bruno Martelli}
\address{Dipartimento di Matematica ``Tonelli'', Largo Pontecorvo 5, 56127 Pisa, Italy}
\email{martelli at dm dot unipi dot it}
\title[Complexity of PL-manifolds]{Complexity of PL-manifolds}

\title{Four-manifolds with shadow-complexity zero}

\begin{document}

\begin{abstract}
We prove that a closed 4-manifold has shadow-complexity zero if and only if it is a kind of 4-dimensional graph manifold, which decomposes into some particular blocks along embedded copies of $S^2\times S^1$, plus some complex projective spaces. We deduce a classification of all 4-manifolds with finite fundamental group and shadow-complexity zero.
\end{abstract}

\maketitle

\section{Introduction}

Piecewise-linear (equivalently, smooth) closed four-manifolds form an enormous set which is still poorly understood. In contrast with dimensions 2 and 3, even a conjectural picture which aims to describe this set globally is missing. Restricting to simply connected manifolds does not help much: Donaldson and Seiberg-Witten invariants have revealed the existence of infinitely many distinct simply-connected manifolds sharing the same topological structure; these \emph{exotic} 4-manifolds have been constructed using various techniques, but a general procedure for constructing (and classifying) \emph{all} simply connected 4-manifolds sharing the same topological structure is still not available. For an overview on this topic, see for instance \cite{Ste}. For an introduction to 4-manifolds see the books \cite{GoSti, Sco}.

We would like to study the set of all closed oriented 4-manifolds globally, by means of a suitable \emph{complexity}. A complexity is a function which assigns to every compact manifold a non-negative integer that measures in some sense how ``complicate'' the manifold is. A complexity induces a \emph{filtration} of the set of all 4-manifolds into subsets $\calM_0 \subset \calM_1 \subset \calM_2 \subset \ldots$ where $\calM_c$ is the set of all manifolds having complexity at most $c$. In such a setting, we would like to construct (and hopefully classify) all 4-manifolds lying in $\calM_c$, starting from $c= 0,1, \ldots $

There are of course various types of reasonable complexities, and different choices may lead to completely different filtrations. However, the problem of constructing and listing all the manifolds in $\calM_0$, $\calM_1, \ldots$ is hard for most of these choices. For instance, a natural complexity might be the minimum number of 4-simplexes in a simplicial (or semisimplicial?) triangulation: with this choice, it may be encouraging to know that $\calM_c$ is finite for all $c$. However, as far as we know, noone has attempted to classify 4-manifolds that can be triangulated with $2, 4, \ldots$ simplexes. 

In fact, triangulations seem too rigid and complicate for our purposes. In dimension 3, Matveev \cite{Mat} has used the somewhat dual notion of \emph{simple spine} to define a complexity for all compact 3-manifolds which satisfies various nice properties: for instance, it is additive on connected sums. A two-dimensional polyhedron is \emph{simple} when it has generic singularities, as in Fig.~\ref{models:fig}. Matveev defines the complexity $c(M)$ of a 3-manifold $M$ as the minimum number of vertices in a simple spine. The price to pay for using spines instead of triangulations is that we get infinitely many manifolds in each $\calM_c$. However, each set $\calM_c$ contains only finitely many ``interesting'' 3-manifolds (say, closed irreducible or bounded hyperbolic), which have been listed for low values of $c=0,1,2,3,\ldots $ by various authors, see \cite{survey, Mat:book, Mat11} and the references therein.

\begin{figure}
 \begin{center}
  \includegraphics[width = 9cm]{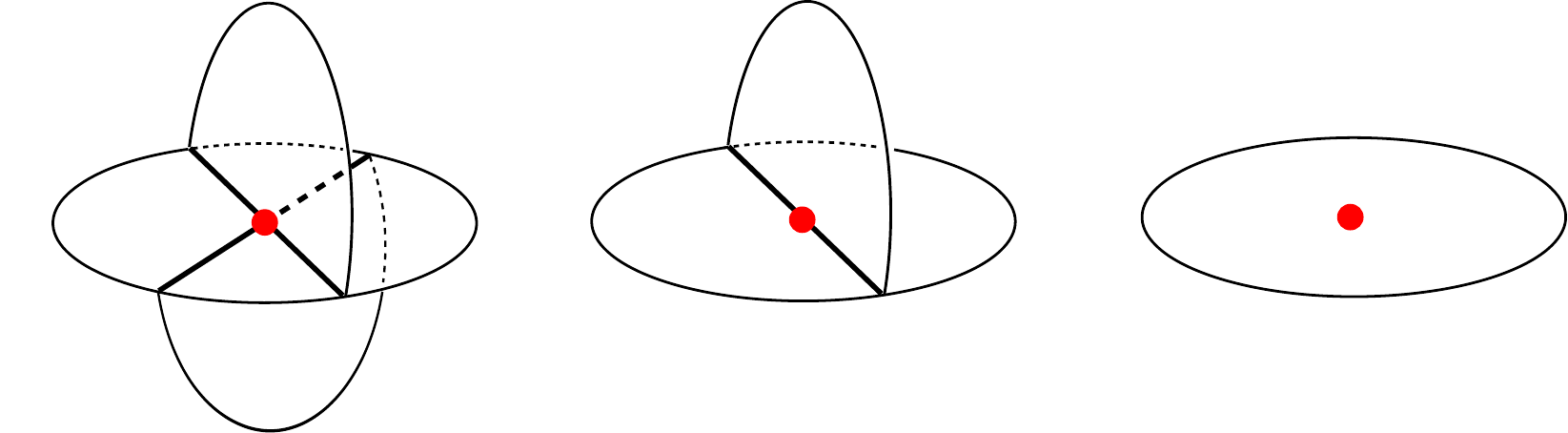}
 \end{center}
 \nota{Neighborhoods of points in a simple polyhedron. A point as in the left picture is a \emph{vertex}.}
 \label{models:fig}
\end{figure}

Most 4-manifolds do not have two-dimensional spines, so Matveev's definition cannot be extended \emph{as is} to dimension 4. There are however two natural variations, which lead to two distinct complexities for compact (piecewise-linear) 4-manifolds. 

One natural variation is obtained by taking three-dimensional simple spines. This extension works in fact for piecewise-linear manifolds of arbitrary dimension $n$ (by taking simple spines of dimension $n-1$): the resulting complexity is introduced and studied in \cite{Ma:PL}. Let $\calM_0 \subset \calM_1 \subset \ldots$ be the induced filtration in dimension 4: as shown in \cite{Ma:PL}, the set $\calM_0$ contains closed 4-manifolds with arbitrary fundamental group, and thus cannot be classified completely. Moreover, many (possibly all) simply-connected 4-manifolds lie in $\calM_0$, so even restricting to simply connected manifolds does not help much. The set $\calM_0$ is interesting, but is too big to be classified.

Another variation consists of using 2-dimensional simple polyhedra not as spines but as more general objects, called \emph{shadows}: following Turaev \cite{Tu0, Tu}, a shadow is a (locally flat) simple polyhedron $X$ in the interior of a compact 4-manifold $M$ such that $M$ is obtained from a regular neighborhood of $X$ by adding 3- and 4-handles. Every compact 4-manifold has a shadow, so it makes sense to define the complexity of a compact 4-manifold as the minimum number of vertices of a shadow. This notion has been recently introduced and studied by Costantino \cite{Co}. 

To avoid confusion, the two notions just introduced in dimension 4 may be called respectively \emph{spine-complexity} and \emph{shadow-complexity}. Spine-complexity was studied in \cite{Ma:PL}. We study here the shadow-complexity (which we call complexity for short) and its induced filtration, which we still denote by $\calM_0 \subset \calM_1 \subset \ldots$

In this paper we give a characterization of the set $\calM_0$ of all closed 4-manifolds having shadow-complexity zero. As we will see, such a set is considerably smaller than the one we obtain from spine-complexity. In particular, we can classify completely the manifolds in $\calM_0$ having finite fundamental group. The set $\calM_0$ is big enough to contain various interesting manifolds, and small enough to allow classifications. Shadow-complexity thus seems to be particularly well-behaved and it seems both feasable and interesting to pursue our program with $\calM_1, \calM_2, \ldots $

The most important discovery is that the set $\calM_0$ looks very much like the set of Waldhausen's 3-dimensional graph manifolds \cite{Wa}. Recall that a Waldhausen graph manifold is any 3-manifold which decomposes into blocks homeomorphic to $D^2\times S^1$ or $P^2\times S^1$, where $P^2$ is the pair-of-pants. The manifold can indeed be described via a graph, with vertices of valence 1 and 3 encoding the blocks, and some data on the edges telling us how they are glued. 

There are many ways to extend this notion to higher dimensions. To preserve generality, we may take a fixed set of oriented $n$-manifolds $\calS = \{M_1,\ldots, M_k,\ldots\}$ and say that an oriented $n$-manifold $M$ is a \emph{graph manifold generated by $\calS$} if $M$ decomposes (along codimension-1 submanifolds) into pieces (orientation-preservingly) PL-homeomorphic to these manifolds. The manifold $M$ can thus be described appropriately by a graph, with vertices of different types corresponding to the elements of $\calS$, and some information on the edges encoding the way they are glued. 

In dimension 4 there are various interesting choices for $\calS$, which lead to quite different notions of graph manifolds. For instance, Mozgova defined in \cite{Moz} a 4-dimensional graph manifold as a manifold generated by torus bundles over compact surfaces of negative Euler characteristic. The blocks are glued along torus bundles over $S^1$, such as the 3-torus.

The generalization we propose here of Waldhausen's graph manifolds is of different kind. 
Each block has some boundary components, all homeomorphic to $S^2\times S^1$. The pieces are thus glued along copies of $S^2\times S^1$. A simple way to get 4-manifolds with such boundary consists of drilling a closed manifold along closed curves (thus removing a $D^3\times S^1)$ or along spheres with Euler number zero (thus removing a $D^2\times S^2$). Consider the following blocks.
\begin{itemize}
\item $M_{i_1\cdots i_k}$ is obtained from $S^3 \times S^1$ by drilling a closed braid as in Fig~\ref{blocchi:fig}.
\item $N_i$ is obtained from $S^2\times S^2$ by drilling $i$ parallel spheres of type $\{pt\}\times S^2$.
\end{itemize}

The graph manifolds we consider here are generated by the following set: 
$$\calS_0 = \big\{M_1, M_{11}, M_2, M_{111}, M_{12}, M_3, N_1, N_2, N_3\big\}.$$

\begin{figure}
\begin{center}
\includegraphics[width = 12.5 cm]{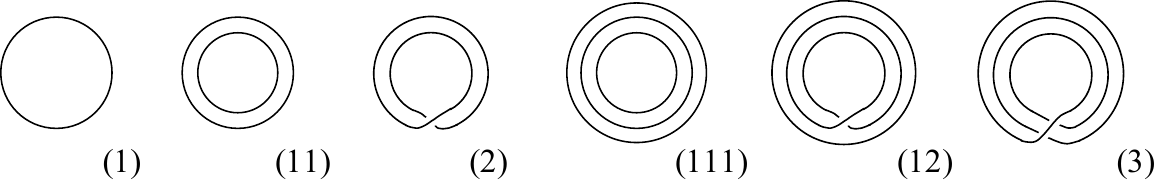}
\nota{Let $L= K_1\cup \ldots \cup K_k$ be a link in $S^3\times S^1$. In dimension 4 homotopy of links implies isotopy, and thus $L$ is determined by the natural numbers $i_j = |[K_j]|$ from $[K_j] \in \pi_1(S^3\times S^1) = \matZ$. Let $M_{i_1\cdots i_k}$ be the manifold obtained by drilling $S^3\times S^1$ along $L$. We show here the links that are important to define $\calS_0$. }
\label{blocchi:fig}
\end{center}
\end{figure}

We can now state the main result proved in this paper. For any integer $h > 0$ and any oriented $n$-manifold $N$, we denote by $\#_h N$ the connected sum of $h$ copies of $N$. When $h=0$ we set $\#_0 N = S^n$ and when $h<0$ we set $\#_h N = \#_{-h}\overline N$. 

\begin{teo} \label{main:teo}
A closed oriented 4-manifold $M$ has complexity zero if and only if $M= M'\#_h \matCP^2$ for some
integer $h$ and some graph manifold $M'$ generated by $\calS_0$.
\end{teo}

We now investigate these graph manifolds: we would like to show that they indeed lie among ``the simplest 4-manifolds'' also from other viewpoints.

A simple method for constructing non-trivial closed 4-manifolds consists of taking the double of a 4-dimensional 2-handlebody, \emph{i.e.}~a compact 4-manifold made of 0-, 1-, and 2-handles. The resulting manifolds may have arbitrary (finitely presented) fundamental group. 

The graph manifolds generated by $\calS_0$ belong to this set. Actually, they are doubles of the ``simplest'' types of 2-handlebodies: those which collapse to simple polyhedra without vertices, as the following shows.
\begin{prop} \label{graph:prop}
Let $M$ be a closed oriented 4-manifold different from $\#_h (S^3\times S^1)$. The following conditions are equivalent.
\begin{enumerate}
\item $M$ is a graph manifold generated by $\calS_0$. 
\item $M$ is the boundary of a compact oriented 5-manifold which collapses onto a simple polyhedron without vertices.
\item $M$ is the double of a compact oriented 4-manifold which collapses onto a simple polyhedron without vertices.
\end{enumerate}
\end{prop}

Graph manifolds generated by $\calS_0$ bound 5-manifolds and have thus signature zero. Therefore the integer $h$ in the statement of Theorem \ref{main:teo} equals the signature of $M$.

We mention that most doubles of 2-handlebodies are \emph{not} graph manifolds: the hypothesis that the collapsed 2-polyhedron has no vertices is quite strong. In some sense, graph manifolds are the ``simplest'' such doubles. In particular, graph manifolds generated by $\calS_0$ do not realize every possible fundamental group, see Proposition \ref{finite:teo} below.

There are various analogies between Waldhausen's graph manifolds and those generated by $\calS_0$. Compare Proposition \ref{graph:prop} to the following.
\begin{teo}[Costantino-Thurston \cite{CoThu}] Let $M$ be a closed oriented 3-manifold. The following conditions are equivalent.
\begin{enumerate}
\item $M$ is a graph manifold.
\item $M$ is the boundary of a compact oriented 4-manifold which collapses onto a locally flat simple polyhedron without vertices.
\end{enumerate}
\end{teo}

Note also that Waldhausen's graph manifolds are generated by the set
$$\calS^{\rm Wald} = \big\{L_1, L_2, L_3 \big\}$$
where $L_i$ is obtained from $S^2\times S^1$ by drilling along $i$ parallel curves of type $\{pt\}\times S^1$. This set has some resemblances with $\calS_0$. The following proposition holds also for Waldhausen's manifolds.
\begin{prop} \label{G:prop}
The set $\calG_0$ of all 4-dimensional graph manifolds generated by $\calS_0$ is closed under connected sum and finite coverings. That is, 
\begin{enumerate}
\item if $M, M' \in \calG_0$ then $M\# M' \in \calG_0$;
\item if $M\in \calG_0$ and $\widetilde M \to M$ is a finite covering, then $\widetilde M \in \calG_0$.
\end{enumerate}
\end{prop}

In a weak sense, complexity in dimension 4 is similar to Gromov norm in dimension 3: Waldhausen's graph manifolds are precisely the closed 3-manifolds having Gromov norm zero (thanks to geometrization!), while the graph manifolds generated by $\calS_0$ plus projective planes are precisely the closed 4-manifolds having complexity zero.

Waldhausen introduced and also classified his graph manifolds in \cite{Wa}. We classify here the graph manifolds generated by $\calS_0$ having finite fundamental group. These manifolds are easily described as boundaries of some 5-manifolds, as follows. 

A finite presentation $\calP$ of a group defines a 2-dimensional polyhedron $X^2$ with one vertex, one edge for each generator, one disc for each relator. Let $\calS(\calP)$ denote the set of all closed oriented 4-manifolds that are boundaries of some oriented 5-manifold that collapses onto $X^2$. The following is easily proved. Recall that an oriented 4-manifold is \emph{spin} when its second Stiefel-Whitney calss $w_2$ vanishes.

\begin{prop} \label{presentation:prop}
The following holds.
\begin{enumerate}
\item The set $\calS(\calP)$ contains finitely many 4-manifolds, precisely one of which is spin. 
\item The manifolds in $\calS(\calP)$ share the same cellular 3-skeleton: therefore all their homology groups and the homotopy groups $\pi_1$ and $\pi_2$ depend only on $\calP$. 
\item If $\calP$ and $\calP'$ are related by Andrew-Curtis moves \cite{AnCu}, then $\calS(\calP) = \calS(\calP')$.
\end{enumerate}
\end{prop}

For instance, the trivial (empty) presentation $\calP = \langle\, |\, \rangle$ yields $\calS(\calP) = \{S^4\}$. A balanced presentation (\emph{i.e.}~having the same number of generators and relators) of the trivial group always yields a unique homotopy 4-sphere. The Andrew-Curtis conjecture states that every such presentation is related to the trivial one by AC-moves \cite{AnCu}. If this holds, then such a homotopy 4-sphere is always $S^4$. However, such a conjecture is commonly believed to be false: one way to disprove it could be to constuct a fake $S^4$ in this way.

Consider the standard presentations 
$$\calC_n = \langle a | a^n \rangle, \quad \calD_{2n} = \langle a,b | a^2, b^2, (ab)^n \rangle$$ 
of the cyclic and dihedral groups. We classify the manifolds in $\calS(\calC_n)$ and $\calS(\calD_{2n})$ and assign them some names.

\begin{prop} \label{finite:prop}
We have the following.
\begin{eqnarray*}
\calS(\calC_n) & = & \left\{\begin{array}{ll} \left\{C_n^0, C_n^1 \right\} & {\rm \ if\ } n {\rm \ is \ even,}\\
\left\{C_n^0\right\} & {\rm \ if\ } n {\rm \ is \ odd.} \end{array}\right. \\
\calS(\calD_{2n}) & = & \left\{\begin{array}{ll} \left\{D_n^0, D_n^1, D_n^2, D_n^3\right\}& {\rm \ if\ } n=2 \\
\left\{D_n^{00}, D_n^{10}, D_n^{20}, D_n^{01}, D_n^{11}, D_n^{21} \right\}  & {\rm \ if\ } n>2 {\rm \ is \ even.} \\ 
\left\{D_n^{0}, D_n^{1}, D_n^{2}\right\}  & {\rm \ if\ } n>2 {\rm \ is \ odd.} \\ 
\end{array}\right.
\end{eqnarray*}
The manifolds $C^0_n, D^0_n, D^{00}_n$ are spin, the others are not. The manifolds $C_n^0$, $C_n^1$, $D^0_n$, $D^2_n$, $D^{00}_n$, $D^{10}_n$, $D^{20}_n$ are even, the others are odd. The universal covering of every manifold in the list is $\#_k (S^2\times S^2)$, for some $k$.
\end{prop}

Recall that a spin 4-manifold is always even, while the converse is true for simply connected manifolds, but not in general. Some non-spin manifolds in the list, like  $D_2^1$ and $D_2^3$, have the same homotopy and homology groups, and intersection forms. We have distinguished them by counting the number of spin coverings.

We may now deduce from Theorem \ref{main:teo} a classification of all 4-manifolds with complexity zero and finite fundamental group.

\begin{teo} \label{finite:teo}
A closed 4-manifold $M$ with finite fundamental group has complexity zero if and only if 
$$M = N \#_h(S^2\times S^2) \#_k \matCP^2 \#_l \overline\matCP^2$$
for some
$$N \in \calS(\calC_{2^n}) \cup \calS(\calC_{3\cdot 2^n}) \cup \calS(\calD_{2\cdot 2^n})$$ 
and $h,k,l,n \geqslant 0$.
\end{teo}
\begin{cor} \label{finite:cor}
A simply connected closed 4-manifold $M$ has complexity zero if and only if $M$ is a connected sum of copies of $S^4$, $S^2\times S^2$, and $\matCP^2$ (with both orientations). That is,
$$ M = \#_h (S^2\times S^2) \qquad {\rm or } \qquad M = \#_h \matCP^2 \#_k \overline\matCP^2$$
for some $h,k\geqslant 0$.
\end{cor}

It is worth emphasizing that Corollary \ref{finite:cor} needs the whole proof of Theorem \ref{main:teo}, which is the core result of this paper. As far as we know, restricting to simply connected 4-manifolds (and thus shadows) does not help much: the whole machinery described in this paper is needed. 

We can easily calculate the classical topological invariants of the manifolds found.
\begin{cor}
For every pair of integers $(\chi, \sigma)$ with $\chi+\sigma$ even there is a closed 4-manifold having complexity zero, signature $\sigma$, and Euler number $\chi$. 
\end{cor}
\begin{proof}
The Euler characteristic of a graph manifold generated by $\calS_0$ is the sum of the characteristics of the blocks. All blocks have $\chi = 0$, except $\chi (N_1) = 2$ and $\chi (N_3) = -2$. Therefore the Euler characteristic of a graph manifold may be any even integer. Its signature is zero since it bounds a 5-manifold. Via connected sums with $\matCP^2$ we get manifolds with arbitrary $\sigma$.
\end{proof}
Note that $\chi + \sigma$ is even for every closed oriented 4-manifold. Concerning intersection forms, we get the following. Let $H$ denote the form $\begin{pmatrix} 0 & 1 \\ 1 & 0 \end{pmatrix}$.
\begin{cor}
The intersection form of a closed 4-manifold having complexity zero is either $n[-1]\oplus m[+1]$ or $kH$.
\end{cor}
\begin{proof}
Graph manifolds have zero signature and thus an indefinite form which is either $n[-1]\oplus n[+1]$ or $kH$. By summing projective planes we get the result.
\end{proof}
The only intersection form admitted for 4-manifolds which has not yet been encountered is $2mH\oplus nE_8$. We thus ask the following.

\begin{quest}
What are the manifolds of lowest complexity having intersection form $2mE_8\oplus nH$? Is the $K3$ among them? Which pairs $(m,n)$ do we get?
\end{quest}

As we said above, Matveev's complexity induces a filtration $\calM_0\subset \calM_1\subset \ldots$ where each $\calM_c$ contains infinitely many 3-manifolds, but only finitely many interesting ones. This also holds for our $\calM_0$, if we decide that doubles of 2-handlebodies and non-irreducible 4-manifolds are not interesting. We conjecture that this holds for all values of $c$.

\begin{conj} \label{finite:conj}
For every natural number $c$ there are only finitely many irreducible 4-manifolds of complexity $c$ that are not doubles of 2-handlebodies.
\end{conj}

In fact, constructing shadows with few vertices of doubles of 2-handlebodies is pretty easy and we expect that there are infinitely many of them for all $c$. 

One may reasonably argue that doubles of 2-handlebodies are interesting, since they might contain for instance fake copies of $S^4$, see \cite{AnCu}. We thus propose an alternative conjecture.

\begin{conj} \label{simply:conj}
For every natural number $c$ there are finitely many irreducible simply connected 4-manifolds of complexity $c$.
\end{conj}

Inside $\calM_0$ we found only $\matCP^2$ and $S^2\times S^2$. Note that, by a result of Auckly, the number of such manifolds (if finite) grows faster than polinomially.

\begin{teo}[Auckly \cite{Au}] The number $n_c$ of distinct manifolds lying in $\calM_c$ that are topologically homeomorphic to $K3$ is bigger than $c^{k\sqrt[3] c}$ for some constant $k$ and sufficiently big $c$.
\end{teo}

Note that a fixed simple polyhedron may give rise only to finitely many closed manifolds \cite{Ma:link}, but there are infinitely many simple polyhedra with a given number of vertices. One may try to attack the conjectures by proving that only finitely many simple polyhedra may yield ``interesting'' 4-manifolds. 

Finiteness may also be obtained \emph{a priori} by defining a complexity which uses a much more restricted class of simple polyhedra, \emph{i.e.}~the \emph{special} ones: see \cite{Au, Co, Ma:link}. This \emph{special complexity} $c^{\rm spec}$ is only related to the complexity $c$ we use here via the obvious inequality $c(M^4)\leqslant c^{\rm spec}(M^4)$. The closed 4-manifolds $M^4$ having $c^{\rm spec}(M^4)\leqslant 1$ are $S^4$, $\matCP^2$, $\matCP^2\#\overline \matCP^2$, $\matCP^2\#\matCP^2$, and $S^2\times S^2$, see \cite{Co}.

Finally, we show how complexity allows to state three well-known conjectures in a similar form. We denote by P, AC, P4 respectively the (now proven) Poincar\'e conjecture, the Andrew-Curtis conjecture \cite{AnCu}, and the (piecewise-linear) 4-dimensional Poincar\'e conjecture. We denote by $\sim$ the homotopy equivalence between manifolds.

\begin{teo} \label{analogies:teo}
The following holds.
\begin{enumerate}
\item P holds $\Longleftrightarrow$ $c(M^3)=0$ for every 3-manifold $M^3\sim S^3$;
\item P4 holds $\Longleftrightarrow$ $c(M^4)=0$ for every 4-manifold $M^4\sim S^4$;
\item AC holds $\Longleftrightarrow$ $c(\calP)=0$ for every presentation $\calP$ of the trivial group.
\end{enumerate}
\end{teo}

Complexity of presentations is defined in Section \ref{presentations:subsection}. The three types of complexities mentioned in Theorem~\ref{analogies:teo} are all defined as the minimum number of vertices of some simple polyhedron. The equivalence (1) follows from Matveev's seminal paper \cite{Mat}, 
(2) follows from Corollary \ref{finite:cor} and
(3) is easily proved in Section \ref{presentations:subsection}.

\subsection*{Structure of the paper}
All the results stated in the introduction except Theorem \ref{main:teo} are proved in Section \ref{simple:section}. The rest of the paper is devoted to proving Theorem \ref{main:teo}. An outline of the proof is present in Section \ref{outline:subsection}.

In Section \ref{shadows:section} we recall (a version of) the definition of Turaev's shadows. We construct shadows (with boundary) without vertices of all the blocks in $\calS_0$ and of $\matCP^2$. In Section \ref{operations:section} we prove that blocks can be assembled along their $(S^2\times S^1)$-boundaries and can be summed (via an internal connected sum) without increasing the complexity. This approach is very similar to the \emph{bricks} construction used in \cite{MaPe} for 3-manifolds.

Section \ref{moves:section} collects some moves that relate two shadows of the same 4-manifold. We introduce there various new moves that are particularly useful when there are no vertices. In Section \ref{without:section} we study simple shadows without vertices, their 4-dimensional thickening, and their 3-dimensional boundary. Sections \ref{reduction:section} to \ref{proof:section} contain the core of the proof of Theorem \ref{main:teo}.

We will always work in the piecewise-linear category. Every manifold and map is tacitly assumed to be PL. 

\subsection*{Acknowledgements}
The author would like to thank Francois Costantino for the many discussions on this topic, and the Maths Department of Austin for its hospitality.

\section{Simple polyhedra} \label{simple:section}
We prove here all the assertions made in the introduction except Theorem \ref{main:teo}. 
We introduce a graph notation to encode simple polyhedra without vertices which will also be used in the subsequent sections.

\subsection{Simple polyhedra with boundary}
\begin{figure}
\begin{center}
\includegraphics{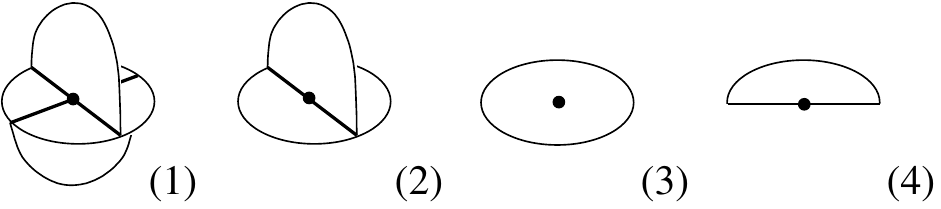}
\nota{Neighborhoods of points in a simple polyhedron with boundary.}
\label{shadow:fig}
\end{center}
\end{figure}

A \emph{simple polyhedron with boundary} is a compact polyhedron $X$ where every point has a link homeomorphic to a circle with three radii, a circle with a diameter, a circle, or a segment. Star neighborhoods are shown in Fig.~\ref{shadow:fig}. 

The \emph{boundary} $\partial X$ is the union of all points of type (4). Points of type (1) are called \emph{vertices}. The points of type (2) and (3) form respectively some manifolds of dimension 1 and 2: their connected components are called respectively \emph{edges} and \emph{regions}. The \emph{singular part} $SX$ of $X$ is the union of all points of type (1), (2), and (4). For simplicity, we will often employ the term \emph{simple polyhedron} to denote a simple polyhedron with boundary.

\subsection{Simple polyhedra without vertices} \label{simple:subsection}
In this paper we are concerned only with simple polyhedra $X$ without vertices. Consider one such polyhedron $X$. Each component of $SX$ is a circle. Its regular neighborhood $N$ has the structure of a $Y$-bundle over $S^1$, where $Y$ denotes the cone over 3 points. There are three topological types for $N$: its  boundary may have 3, 2, or 1 components, and look like respectively as $(111)$, $(12)$, and $(3)$ from Fig.~\ref{sum_no_gleam:fig}. We use the names $Y_{111}$, $Y_{12}$, and $Y_3$ to denote these three objects. Of course we have $Y_{111} \isom Y\times S^1$.

After removing regular neighborhoods of the circles in $SX$ we are left with regions. These in turn decompose, as every surface, into discs, M\"obius strips, and pair-of-pants. We denote such objects by $D^2$, $Y_2$, and $P^2$. The name $Y_2$ follows from analogy with Fig.~\ref{sum_no_gleam:fig}. We have proved the following.

\begin{prop} \label{simple:prop}
Every simple polyhedron without vertices decomposes along simple closed curves into pieces homeomorphic to $D^2$, $P^2$, $Y_2$, $Y_{111}$, $Y_{12}$, and $Y_3$.
\end{prop}

\begin{figure}
\begin{center}
\includegraphics[width = 11 cm]{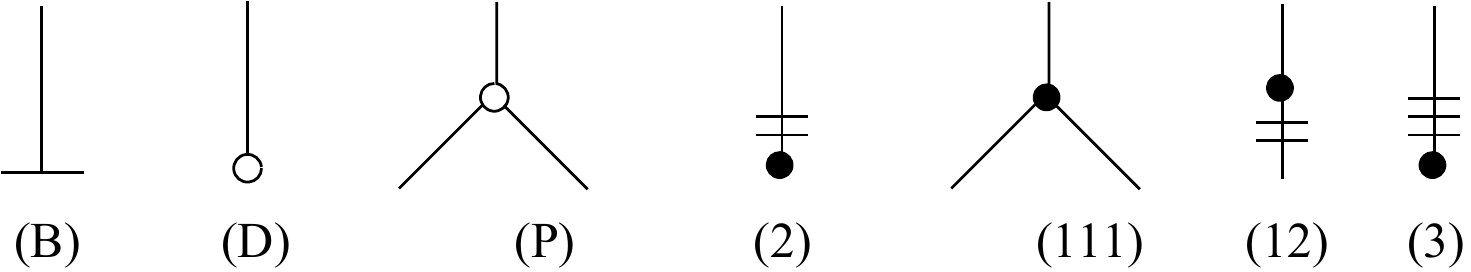}
\nota{A graph with these types of vertices encodes a simple polyhedron without vertices.}
\label{vertices:fig}
\end{center}
\end{figure}
A simple polyhedron $X$ without vertices $X$ is easily encoded by a graph $G$ with vertices as in Fig.~\ref{vertices:fig}. Vertices of type (D), (P), (2), (111), (12), (3) denote respectively pieces homeomorphic to $D^2$, $P^2$, $Y_2$ $Y_{111}$, $Y_{12}$, and $Y_3$. A vertex of type (B) encodes a boundary component of $X$. Note that the vertex of type (12) is not symmetric: the edge marked with two lines should correspond to the region winding twice over the singular circle in $SX$. 

Every edge of $G$ denotes a gluing of two such pieces. There are two possible gluings, since there are two self-homeomorphisms of $S^1$ up to isotopy, one orientation-preserving and one reversing. This gives a map $\beta:H_1(G,\matZ_2)\to\matZ_2$. Each piece admits a self-homeomorphism that reverses the orientation of the boundary circles. Therefore the graph $G$ and $\beta$ together encode the simple polyhedron $X$.

Since a surface can split along pants, discs, and M\"obius strips in multiple ways, there are some moves that modify the graph while leaving the associated polyhedron unchanged. Some of these are shown in Fig.~\ref{mosse_innocue_ungleamed:fig}.

\begin{figure}
\begin{center}
\includegraphics[width = 10 cm]{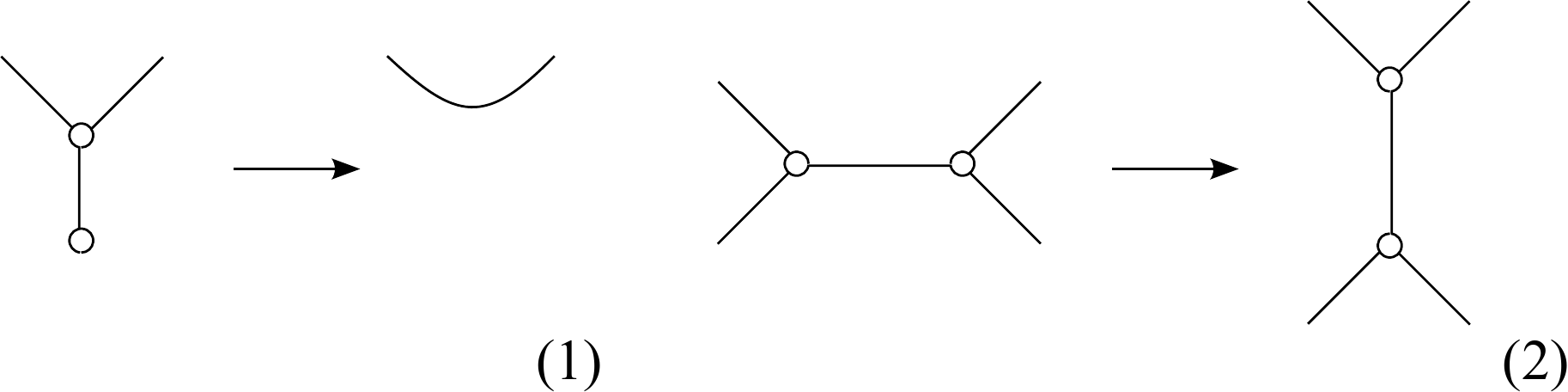}
\nota{These moves do not modify the polyhedron $X$.}
\label{mosse_innocue_ungleamed:fig}
\end{center}
\end{figure}

\subsection{Simple homotopy and presentations} \label{presentations:subsection}

A \emph{simple homotopy} between two polyhedra $X, X'$ of dimension 2 is a composition of simplicial collapses and expansions that transform $X$ into $X'$. Two polyhedra $X$ and $X'$ are \emph{3-deformation equivalent} if there is a simple homotopy between them which involves only collapses and expansions of simplexes of dimension $\leqslant 3$. Recall from the introduction that every presentation $\calP$ defines a 2-dimensional polyhedron $X_\calP$.

\begin{teo} \label{bijection:teo}
The map $\calP \mapsto X_\calP$ defines a bijection between Andrew-Curtis classes of presentations and 3-deformation classes of 2-dimensional polyhedra.
\end{teo}

See \cite{HoMe} for a careful proof of this theorem and a nice introduction to the subject. We introduce the following definition.

\begin{defn} The \emph{complexity} $c(\calP)$ of a presentation $\calP$ is the minimum number of vertices of a simple polyhedron $X$ with boundary which is 3-deformation equivalent to $X_\calP$. 
\end{defn}

This number is always finite, since every 2-dimensional polyhedron is easily seen to be 3-deformation equivalent to a simple one. By Theorem \ref{bijection:teo}, the number $c(\calP)$ depends only on the Andrew-Curtis class of $\calP$ and may also be interpreted as a complexity on 3-deformation classes of polyhedra. 

Thanks to Theorem \ref{bijection:teo} we can safely shift from presentations (up to AC-equivalence) to 2-dimensional polyhedra (up to 3-deformation). Free products of presentations correspond to wedge products of polyhedra, and we denote both these operations by $\vee$. For the sake of clearness, we denote by $S^2$ the presentation $\langle a| a,a\rangle$ which indeed corresponds to $S^2$. Here we will need the following.

\begin{prop} \label{AC:prop}
The presentations (up to AC-equivalence) of finite groups having complexity zero are precisely those of the form $\calP \vee_h S^2$ for some $h\geqslant 0$ and some $\calP = \calC_{2^n}$, $\calC_{3\cdot 2^n}$, or $\calD_{2\cdot 2^n}$ with $n\geqslant 0$.
\end{prop}

\begin{proof}
We will use at various points the following trick. Let $X$ be a simple polyhedron without vertices. It is described by a graph $G$ with vertices as in Fig.~\ref{vertices:fig}. Consider the move in Fig.~\ref{wedge:fig}. An edge of the graph determines a circle in a region of $X$. If we shrink the circle to a point (and $G$ is a tree), the resulting polyhedron is a wedge $X_1\vee X_2$ of two simple polyhedra, as described by the move. 

\begin{figure}
\begin{center}
\includegraphics[width = 5 cm]{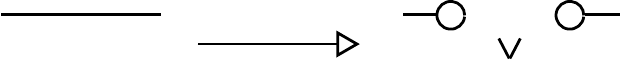}
\nota{This move shrinks a circle contained in a region of the polyhedron (determined by an edge of the graph) to a point. If the original graph is a tree, the move disconnects the graph and the resulting polyhedron is the wedge product $\vee$ of two simple polyhedra.}
\label{wedge:fig}
\end{center}
\end{figure}

There is an obvious map $X\to X_1\vee X_2$ which induces a surjective map
$$\pi_1(X) \to \pi_1(X_1) * \pi_1(X_2).$$
If $X$ is simply connected then both $X_1$ and $X_2$ also are, and if $\pi_1(X)$ is finite then either $\pi_1(X_1)$ or $\pi_1(X_2)$ is trivial (and the other is finite).

\begin{figure}
\begin{center}
\includegraphics[width = 11 cm]{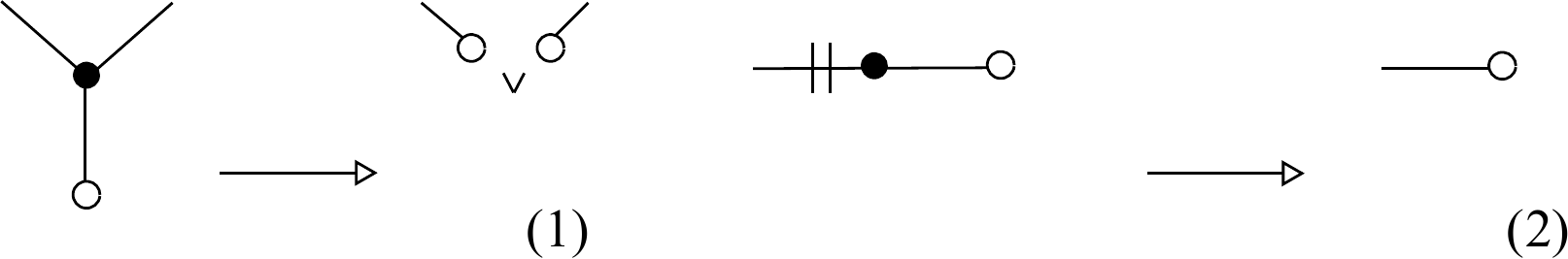}
\nota{Each of these moves may be realized as a 3-deformation. We apply (1) only when the graph is a tree: it transforms $X$ into two polyhedra $X_1$ and $X_2$, and we have $X \sim X_1\vee X_2$.}
\label{AC_mosse:fig}
\end{center}
\end{figure}

Another fact that we will use is that both moves in Fig.~\ref{AC_mosse:fig} can be realized via 3-deformations (this can be seen easily). We will denote 3-deformation equivalence via the symbol $\sim$.

We will now prove a general claim. Let $Y_i$ be the simple polyhedron drawn in Fig.~\ref{gira:fig}-(1). 
Let $X$ be any simple polyhedron without vertices and with one boundary component (\emph{i.e.}, we have $\partial X\isom S^1$). Let $\hat X$ be obtained from $X$ by capping the boundary with a disc. 

\begin{figure}
\begin{center}
\includegraphics[width = 9 cm]{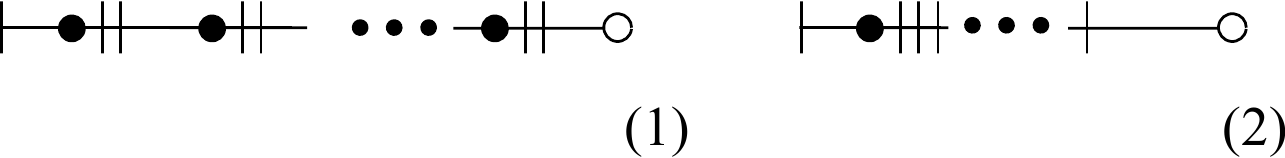}
\nota{The simple polyhedron with boundary $Y_i$, with $i\geqslant 0$ black vertices and one white one (1). Such a polyhedron is 3-deformation equivalent (relative to $\partial Y_i$) to a polyhedron made of an annulus $A$ and a disc winding $2^i$ times around one component of $\partial A$. We might denote this polyhedron as in (2), with $2^i$ small vertical lines.}
\label{gira:fig}
\end{center}
\end{figure}

\emph{Claim. If $\pi_1(\hat X) = \{e\}$ then $X$ is 3-deformation equivalent (relative to $\partial X$) to $X' = Y_i\vee_h S^2$ for some $i,h\geqslant 0$.}

By a 3-deformation equivalence relative to $\partial X$ we mean that collapses and expansions take place away from $\partial X$. Note that the claim easily implies the following.

\emph{Corollary. A simply connected simple polyhedron without boundary and without vertices is 3-deformation equivalent to $\vee_h S^2$.}

We prove the claim. The polyhedron $X$ is described by a graph $G$ with vertices as in Fig.~\ref{vertices:fig}. There is precisely one vertex of type \includegraphics[width = 0.6 cm]{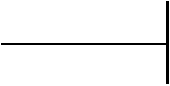}, corresponding to $\partial X$. A graph $\hat G$ for $\hat X$ is obtained simply by substituting this vertex with a \includegraphics[width = 0.6 cm]{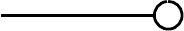}. Both graphs are trees since $H_1(\hat X,\matZ)$ is trivial.

We prove the claim by induction on the number of vertices of $G$. The vertex
\includegraphics[width = 0.6 cm]{0.pdf} cannot be incident to one vertex of type 
\includegraphics[width= 0.6 cm]{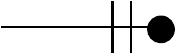} or \includegraphics[width = 0.6 cm]{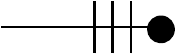} because
\includegraphics[width= 0.6 cm]{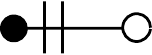} and \includegraphics[width = 0.6 cm]{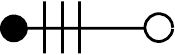} are not simply connected. Therefore the vertex \includegraphics[width = 0.6 cm]{0.pdf} is incident to one vertex of type 
\includegraphics[width = 0.6 cm]{1.pdf}, \includegraphics[width = 0.6 cm]{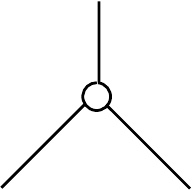}, \includegraphics[width = 0.6 cm]{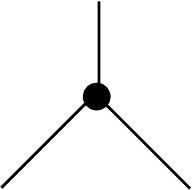}, or \includegraphics[width = 0.6 cm]{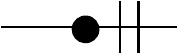}. In the first case $X$ is a disc, \emph{i.e.} $X = Y_0$ and we are done. In all other cases we conclude by induction, as follows.

\begin{figure}
\begin{center}
\includegraphics[width = 10 cm]{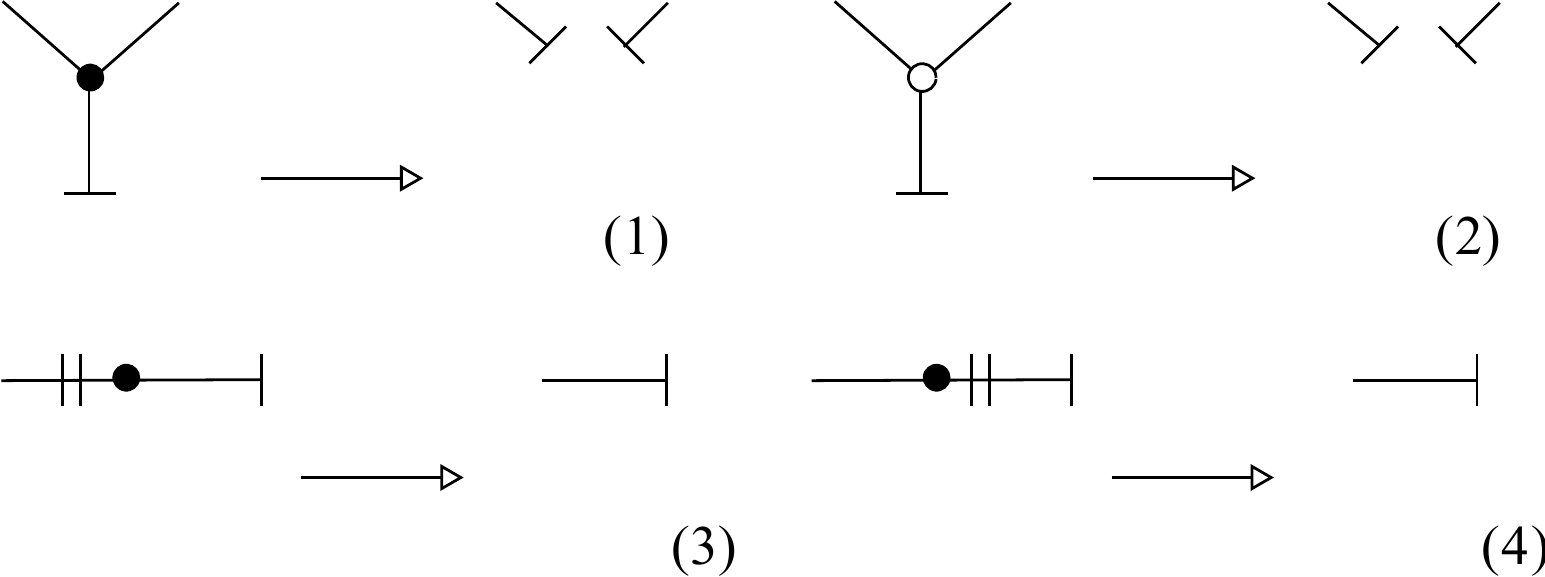}
\nota{Proof of the claim in Proposition~\ref{AC:prop}.}
\label{AC_mosse2:fig}
\end{center}
\end{figure}

Each of the moves in Fig.~\ref{AC_mosse2:fig} transforms $X$ into one or two polyhedra which satisfy our induction hypothesis: we can easily conclude in each case. More precisely, move (1) transforms $X$ into two polyhedra $X_1$ and $X_2$. Consider the capped polyhedra $\hat X$, $\hat X_1$, and $\hat X_2$: we have $\hat X \sim \hat X_1 \vee \hat X_2$. Therefore $\{e\} = \pi_1(\hat X) = \pi_1(\hat X_1)* \pi_1(\hat X_2)$. Thus $\pi(\hat X_1) = \pi(\hat X_2) = \{e\}$ and our induction hypothesis apply to both $X_1$ and $X_2$. Therefore
\begin{eqnarray*}
X_1 & \sim & Y_i\vee_h S^2 \\
X_2 & \sim & Y_j \vee_k S^2 
\end{eqnarray*}
and we easily deduce that
\begin{eqnarray*}
X & \sim & Y_{\min\{i,j\}} \vee_{h+k+1} S^2.
\end{eqnarray*}
Note that all the 3-deformations are performed away from $\partial X_1$ and $\partial X_2$ and therefore survive in $X$.

We turn to move (2). The first trick described above gives a map $\hat X \to \hat X_1\vee\hat X_2$ which is surjective on fundamental groups, thus we conclude again that $X_1$ and $X_2$ fulfill the induction hypothesis. Again we get
\begin{eqnarray*}
X_1 & \sim & Y_i \vee_h S^2 \\
X_2 & \sim & Y_j\vee_k S^2 
\end{eqnarray*}
which implies that
\begin{eqnarray*}
\hat X & \sim & X_{\langle a | a^{2^i}, a^{2^j} \rangle} \vee_{h+k} S^2. 
\end{eqnarray*}
Since $\hat X$ is simply connected, either $i=0$ or $j=0$. Suppose $i=0$: we then get 
\begin{eqnarray*}
X & \sim & Y_j\vee_{h+k} S^2. 
\end{eqnarray*}
In move (3) the polyhedron $X$ is transformed into a polyhedron $X'$ such that $\hat X \sim \hat X'$, see Fig.~\ref{AC_mosse:fig}-(2). Therefore $X'$ fulfills the hypothesis and we get
\begin{eqnarray*}
X' & \sim & Y_i\vee_h S^2
\end{eqnarray*}
which implies that
\begin{eqnarray*}
X & \sim & Y_{i+1}\vee_h S^2.
\end{eqnarray*}
Finally, in move (4) we have a map $\hat X \to \hat X'$ which is surjective on fundamental groups. Therefore $X'$ fulfills the hypothesis. We get 
\begin{eqnarray*}
X'  & \sim & Y_i\vee_h S^2
\end{eqnarray*}
which implies that
\begin{eqnarray*}
\hat X & \sim & \langle a | a^{2^i}, a^2 \rangle \vee_h S^2.
\end{eqnarray*}
Since $\pi_1(\hat X)=\{e\}$, we deduce that $i=0$. This implies that $X \sim Y_0\vee_h S^2$.

We have proved the claim. It is now easy to deduce the proposition. Let $X$ be a simple polyhedron without vertices. It always collapses onto the union of a simple polyhedron without boundary and some 1-dimensional polyhedron. Since $\pi_1(X)$ is finite, the 1-dimensional polyhedron also collapses and we are left either with a simple polyhedron without boundary, which we still call $X$, or with a point. In the latter case we are done.

Represent $X$ via a graph $G$. Take an edge of $G$. It determines a loop $\gamma$ in a region of $X$, which separates $X$ into two polyhedra $X_1$, $X_2$ with $\partial X_1 = \partial X_2 = \gamma$. We apply the usual trick by shrinking $\gamma$ to a point. We get a surjective map from $\pi_1(X)$ to $\pi_1(\hat X_1)*\pi_1(\hat X_2)$. Since $\pi_1(X)$ is finite, either $\pi_1(\hat X_1)$ or $\pi_1(\hat X_2)$ is trivial. Suppose that $\pi_1(\hat X_1)$ is trivial. Then we apply the claim to $X_1$.
We get $X_1 \sim Y_i\vee S^2$ relative to $\gamma$. 

We can apply this to every edge of $G$. It is easy to conclude that $X$ is 3-deformation equivalent to a polyhedron which may be represented via one single vertex $v$ from Fig.~\ref{vertices:fig} and a polyhedron of type $Y_i\vee_h S^2$ attached to each of the incident edges. We conclude as follows:
\begin{itemize}
\item if $v$ is of type (D) then $X \sim X_{\calC_{2^i}}\vee_h S^2$;
\item if $v$ is of type (P) then $X \sim X_{\langle a,b | a^{2^i}, b^{2^j}, (ab)^{2^k} \rangle}\vee_h S^2$;
\item if $v$ is of type (2) then $X \sim X_{\calC_{2\cdot 2^i}}\vee_h S^2$;
\item if $v$ is of type (111) then $X \sim X_{\langle a | a^{2^i}, a^{2^j}, a^{2^k} \rangle}\vee_h S^2 \sim X_{\calC_{2^{\min \{i,j,k\}}}} \vee_{h+2} S^2$;
\item if $v$ is of type (12) then $X \sim X_{\langle a | a^{2\cdot 2^i}, a^{2^j} \rangle}\vee_h S^2 \sim X_{\calC_{2^{\min \{i+1,j\}}}} \vee_{h+1} S^2$;
\item if $v$ is of type (3) then $X \sim X_{\calC_{3\cdot 2^i}}\vee_h S^2$.
\end{itemize}
In all cases we are done except when $v$ is of type (P). Recall that a group presented as
$$\langle a,b \ | \ a^p, b^q, (ab)^r \rangle$$
is finite precisely when $1/p+1/q+1/r<1$. Thus when $v$ is of type (P) and we take $i\leqslant j \leqslant k$ we get:
\begin{itemize}
\item $(i,j,k) = (0,j,k)$, and $X \sim X_{\calC_{2^{\min\{j,k\}}}} \vee_{h+1} S^2$,
\item $(i,j,k) = (1,1,k)$, and $X \sim X_{\calD_{2\cdot 2^k}} \vee_h S^2$
\end{itemize}
as required.
\end{proof}

\subsection{Five-dimensional thickenings}
We now study 5-dimensional thickenings of simple polyhedra, and their 4-dimensional boundaries.
Five-dimensional thickenings are easier to study than four-dimensional ones: this may explain why Proposition \ref{graph:prop} is much easier to prove than Theorem \ref{main:teo}. To prove the proposition we start with a general lemma (which is well-known to experts).

\begin{lemma} \label{23:lemma}
Let $X$ be a compact 2-dimensional polyhedron. Let $M$ be a closed oriented 4-manifold. The following conditions are equivalent.
\begin{enumerate}
\item $M$ is the boundary of a compact oriented 5-manifold which collapses on $X$;
\item $M$ is the double of a compact 4-manifold which collapses on $X$.
\end{enumerate}
\end{lemma}
\begin{proof}
(2) $\Rightarrow$ (1). We have $M = DN$ for some 4-dimensional compact $N$ which collapses to $X$. Clearly the 5-dimensional $N\times [0,1]$ also collapses to $X$ and $\partial (N\times [0,1]) \isom M$.

(1) $\Rightarrow$ (2). We have $M= \partial W$ for some oriented 5-manifold $W$ which collapses to $X$. Choose a triangulation of $X$ and thicken it to a handle decomposition for $W$. Thicken arbitrarily the triangulation of $X$ to a handle decomposition of a 4-manifold $N$, and thicken it again to a handle decomposition of $N\times [0,1]$. The manifolds $W$ and $N\times [0,1]$ have the same 0- and 1-handles. Concerning 2-handles, their attaching circles are homotopic, and since they lie in some 4-dimensional manifold they are actually isotopic. 

The only thing that might differ between the handle decompositions of $W$ and $N\times [0,1]$ is the way each 2-handle is attached: there are two possibilities since $\pi_1(SO(3))=\matZ_2$. In dimension 4, there are infinitely many possibilities since $\pi_1(SO(2)) = \matZ$. A 2-handle for $N$ induces a 2-handle for $N \times [0,1]$ according to the surjective homomorphism $\pi_1(SO(2))\to \pi_1(SO(3))$ induced by a standard injective map $SO(2) \to SO(3)$. When constructing $N$, it suffices to choose on each 2-handle a framing with the right parity, coherent with the corresponding 2-handle of $W$. With this choice, we get $W \isom N\times [0,1]$, and we are done.
\end{proof}

In practice, to deal with graph manifolds we may use a smaller generating set $\calS_0' \subset \calS_0$, as the following shows.
\begin{prop} \label{smaller:prop}
Every graph manifold generated by $\calS_0$ is a connected sum of $h\geqslant 0$ copies of $S^3\times S^1$ and $k\geqslant 0$ graph manifolds generated by the set
$$\calS_0' = \big\{M_2, M_{111}, M_{12}, M_3, N_1, N_3\big\}.$$
\end{prop}
\begin{proof}
A graph manifold generated by $\calS_0$ decomposes into blocks homeomorphic to those of $\calS_0'$ and $M_1, M_{11}, N_2$. Each block homeomorphic to $M_{11} = N_2 = S^2\times S^1 \times [0,1]$ may be simply removed or substituted with a pair of $N_1 = D^2\times S^2$ and $N_3=P^2\times S^2$. It remains to prove that we can also rule out the block $M_1 = D^3\times S^1$. Every self-diffeomorphism of $S^2\times S^1$ extends to $D^3\times S^1$, see \cite{LaPo}. Therefore there is only one way to glue this block to the adjacent block. 

Gluing $M_1$ consists of \emph{filling}, the opposite of drilling along a curve. It is thus clear that by gluing $M_1$ to some piece $M_{i_1\cdots i_h}$ we get a simpler piece $M_{i_1\cdots \hat i_j \cdots i_h}$, or $M_\emptyset = S^3\times S^1$. So after finitely many simplifications we may suppose that each $M_1$ is glued only along a copy of $N_1$ or $N_3$. In the first case we get $S^4$. In the second case, it is easy to see that 
$$M_1 \cup N_3 \isom M_1 \# M_1 $$
and we proceed by iteration.
\end{proof}

\begin{rem}
Every manifold in $\calS_0$ is easily seen to be a double and thus admits an orientation-reversing self-homeomorphism. For that reason the chosen orientation is not important. The same holds for every graph manifold generated by $\calS_0$.
\end{rem}

We may now prove Proposition \ref{graph:prop}. A simple polyhedron without boundary in a 4-manifold is locally flat if it is locally contained in a 3-dimensional slice, see Definition \ref{properly:defn}.

\begin{prop} \label{graph2:prop}
Let $M$ be a closed oriented 4-manifold different from $\#_h (S^3\times S^1)$. The following conditions are equivalent.
\begin{enumerate}
\item $M$ is a graph manifold generated by $\calS_0$. 
\item $M$ is the boundary of a compact oriented 5-manifold which collapses onto a simple polyhedron without vertices (and without boundary).
\item $M$ is the double of a compact 4-manifold which collapses onto a simple polyhedron without vertices (and without boundary).
\item $M$ is the double of a compact 4-manifold which collapses onto a locally flat simple polyhedron without vertices (and without boundary).
\end{enumerate}
\end{prop}
\begin{proof}
The equivalence between (2) and (3) is settled by Lemma \ref{23:lemma}

(2) $\Rightarrow$ (1). Let $X$ be a simple polyhedron without vertices and $W^5$ a compact oriented 5-manifold collapsing to it. The polyhedron $X$ decomposes into pieces as stated by Proposition \ref{simple:prop}. The pieces are homeomorphic to $D^2$, $P^2$, $Y_2$, $Y_{111}$, $Y_{12}$, or $Y_3$. 

The regular neighborhood $N(X)$ of $X$ in $W^5$ decompose similarly into pieces obtained by thickening the pieces above. These pieces are homeomorphic respectively to $D^2\times D^3$, $P^2\times D^3$, $S^1\times D^4$, $S^1\times D^4$, $S^1\times D^4$, and $S^1\times D^4$ again.

Each piece $P$ of $N(X)$ fibers over the corresponding piece $\pi(P)$ of $X$. The 4-dimensional boundary $\partial P$ decomposes into a ``horizontal'' part, which is contained in $\partial N(X)$, and a ``vertical'' part, consisting of $\pi^{-1}(\partial (\pi(P)))$. The vertical part is made of copies of $D^3\times S^1$ that are glued together to form properly embedded submanifolds of $N(X)$.

It is easy to check that the horizontal part is homeomorphic respectively to $N_1$, $N_3$, $M_2$, $M_{111}$, $M_{12}$, or $M_3$. Therefore $\partial N(X)$ is a graph manifold. Since $W^5$ collapses onto $X$, we have $W^5 \isom N(X)$ and we are done.

(1) $\Rightarrow$ (2). 
By Proposition \ref{smaller:prop}, every graph manifold $M\neq \#_k(S^3\times S^1)$ is a connected sum of some graph manifolds $Q_1,\ldots, Q_h$ generated by $\calS_0'$ and $h'$ copies of $S^3\times S^1$. (We have $h\geqslant 1$ and $h'\geqslant 0$ since $M\neq \#_k(S^3\times S^1)$.)

Consider one $Q_i$. It decomposes into pieces homeomorphic to $M_2$, $M_{111}$, $M_{12}$, $M_3$, $N_1$, and $N_3$. As we have seen, every such piece is the horizontal boundary of a 5-dimensional block which fibers over some simple polyhedron with boundary without vertices. 

Every self-homeomorphism of $S^2\times S^1$ is isotopic to one which preserves the foliation in spheres and thus extends to $D^3\times S^1$. We can therefore glue correspondingly the 5-dimensional blocks. The resulting 5-manifold $W^5_i$ fibers (and collapses) to a simple polyhedron $X_i$ without boundary and without vertices. Its boundary $\partial W^5_i$ is homeomorphic to $Q_i$.

By using $h-1$ times the move in Fig.~\ref{sum_no_gleam:fig} we construct from $X_1,\ldots, X_h$ a connected simple polyhedron $X$ such that the boundary-sum $W^5 = W_1^5\sharp \ldots \sharp W_h^5$ collapses onto $X$. Of course, we have $\partial W^5 = Q_1\#\ldots \#Q_h$. We then use $h'$ times Fig.~\ref{sum_no_gleam:fig} again to realize $h'$ self-connected sums and get the $\#_{h'}(S^3\times S^1)$ factors. 

(4) $\Rightarrow$ (3). Obvious.

(2) $\Rightarrow$ (4). In the proof of Lemma \ref{23:lemma}, we have the freedom to construct a locally flat $X$. 
\end{proof}

We can easily prove Proposition \ref{G:prop}.

\begin{prop} 
The set $\calG_0$ of all 4-dimensional graph manifolds generated by $\calS_0$ is closed under connected sum and finite coverings. That is, 
\begin{enumerate}
\item if $M, M' \in \calG_0$ then $M\# M' \in \calG_0$;
\item if $M\in \calG_0$ and $\widetilde M \to M$ is a finite covering, then $\widetilde M \in \calG_0$.
\end{enumerate}
\end{prop}
\begin{proof}
If $W^5$ collapses onto a simple polyhedron $P^2$ and $W'^5$ collapses onto $P'^2$, then the $\partial$-connected sum $W\sharp W'$ collapses onto the simple polyhedron $R^2$ constructed in Fig.~\ref{sum_no_gleam:fig}. Since $\partial (W \sharp W') = \partial W \# \partial W'$, we get (1). We turn to (2). Since $W^5$ collapses onto a 2-dimensional polyhedron, it admits a decomposition with 0-, 1-, and 2-handles. Therefore the inclusion $\partial W^5 \to W^5$ induces an isomorphism on fundamental groups. Every covering of $\partial W^5$ is thus induced by a covering of $P^2$. The covering of a simple polyhedron without vertices is a simple polyhedron without vertices, hence we are done.
\begin{figure}
\begin{center}
\includegraphics[width = 9 cm]{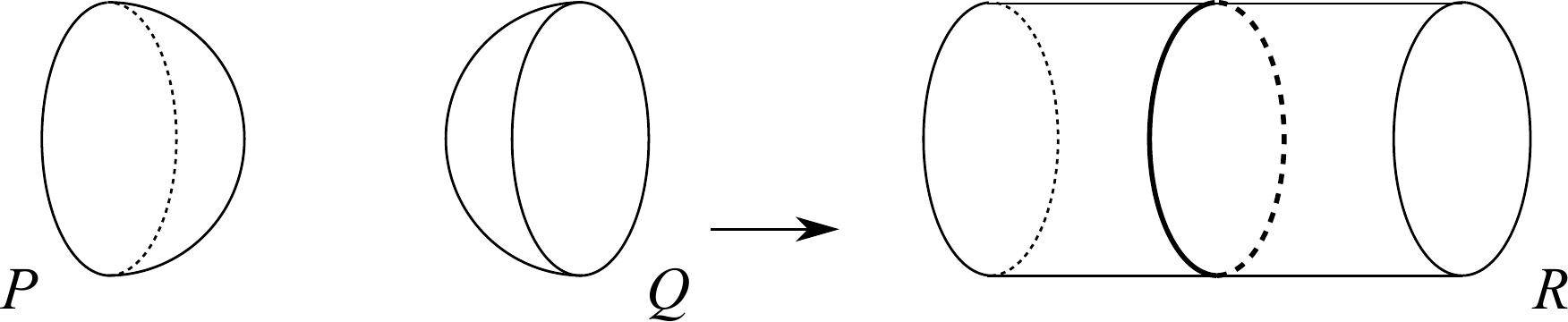}
\nota{Given two simple polyhedra $P$ and $Q$ without vertices, we easily construct a simple polyhedron $R$ which is 3-deformation equivalent to the wedge $P\vee Q$ and has still no vertices. In $R$, a disc is attached along the black circle.}
\label{sum_no_gleam:fig}
\end{center}
\end{figure}
\end{proof}

\subsection{Finite fundamental groups}
We prove here Propositions \ref{presentation:prop} and \ref{finite:prop}. Let $\calS(X^2)$ denote the set of all closed 4-manifolds that are boundaries of some orientable 5-manifold that collapses onto $X^2$.
\begin{teo}[Andrews and Curtis, and others \cite{AnCu, HoMe}] \label{AC:teo}
If $X$ and $X'$ are 3-deformation equivalent then $\calS(X) = \calS(X')$.
\end{teo}
\begin{proof}
The set of all 5-dimensional thickenings of $X$ and $X'$ coincide, see \cite{AnCu, HoMe}. Therefore the set of their boundaries also coincide.
\end{proof}

The following shows that $\calS(X)$ is finite.
\begin{prop} \label{w2:prop}
Let $X^2$ be a compact 2-dimensional polyhedron. For every class $\alpha\in H_2(X^2,\matZ_2)$ there is precisely one 5-dimensional manifold $W^5$ collapsing onto $X^2$ with $w_2(W^5) = \alpha$.
\end{prop}

See \cite{HaKrTe} for a proof. The 5-dimensional thickenings of a 2-dimensional polyhedron $X^2$ are thus in natural correspondence with the elements in $H_2(X^2,\matZ_2)$. We can now prove Propositions \ref{presentation:prop} and \ref{finite:prop}.

\begin{prop}
The following holds.
\begin{enumerate}
\item The set $\calS(\calP)$ contains finitely many 4-manifolds, precisely one of which is spin. 
\item The manifolds in $\calS(\calP)$ share the same cellular 3-skeleton: therefore all their homology groups and the homotopy groups $\pi_1$ and $\pi_2$ depend only on $\calP$. 
\item If $\calP$ and $\calP'$ are related by Andrew-Curtis moves \cite{AnCu}, then $\calS(\calP) = \calS(\calP')$.
\end{enumerate}
\end{prop}

\begin{proof}
Let $X^2$ be the polyhedron determined by $\calP$.
Proposition \ref{w2:prop} implies that $X^2$ thickens to finitely many 5-manifolds $W^5$, precisely one of which has vanishing $w_2(W^5)$. The map $i^*:H^2(W^5,\matZ_2)\to H^2(\partial W^5,\matZ_2)$ induced by inclusion is injective since $H^1(W^5,\partial W^5) \isom H_4(W^5)= 0$. By the naturality of the Stiefel-Whitney class we have $i^*(w_2(W^5)) = w_2(\partial W^5)$. Hence $W^5$ is spin if and only if $\partial W^5$ is spin, and (1) is proved.

We turn to (2). The 1-skeleton of $X^2$ can be thickened in a unique way to a 5-manifold, whose boundary is $\#_k (S^3\times S^1)$. Such a boundary intersects the 2-cells of $X^2$ into a link. The set $\calS(\calP)$ consists of all the 4-manifolds that can be obtained by surgery along that link. Therefore these manifolds share the same 3-skeleton. (A surgery consists of removing $S^1\times D^3$ and then adding a 2-handle and a 4-handle. The 2-handle depends on a framing, but its core disc does not. By adding only the core discs we thus get a common 3-skeleton for all the manifolds in $\calS(\calP)$.)

Finally, (3) follows from Theorem \ref{AC:teo}.
\end{proof}

\begin{prop}
We have the following.
\begin{eqnarray*}
\calS(\calC_n) & = & \left\{\begin{array}{ll} \left\{C_n^0, C_n^1 \right\} & {\rm \ if\ } n {\rm \ is \ even,}\\
\left\{C_n^0\right\} & {\rm \ if\ } n {\rm \ is \ odd.} \end{array}\right. \\
\calS(\calD_{2n}) & = & \left\{\begin{array}{ll} \left\{D_n^0, D_n^1, D_n^2, D_n^3\right\}& {\rm \ if\ } n=2 \\
\left\{D_n^{00}, D_n^{10}, D_n^{20}, D_n^{01}, D_n^{11}, D_n^{21} \right\}  & {\rm \ if\ } n>2 {\rm \ is \ even.} \\ 
\left\{D_n^{0}, D_n^{1}, D_n^{2}\right\}  & {\rm \ if\ } n>2 {\rm \ is \ odd.} \\ 
\end{array}\right.
\end{eqnarray*}
The manifolds $C^0_n, D^0_n, D^{00}_n$ are spin, the others are not. The manifolds $C_n^0$, $C_n^1$, $D^0_n$, $D^2_n$, $D^{00}_n$, $D^{10}_n$, $D^{20}_n$ are even, the others are odd. The universal covering of every manifold in the list is $\#_k (S^2\times S^2)$, for some $k$.
\end{prop}

\begin{proof}
Let $X_\calP$ be the 2-dimensional polyhedron associated to some presentation $\calP$. Let $W^5$ be the 5-dimensional thickening of $X_\calP$, determined by its Steifel-Whitney class $w_2 \in H^2(W^5,\matZ_2)\isom H^2(X_\calP,\matZ_2)$. 

By naturality, the Stiefel-Whitney class of $\partial W^5$ is the image $i^*(w_2)$ along the injective map $i^*:H^2(W^5,\matZ_2)\to H^2(\partial W^5,\matZ_2)$. The following holds:

\begin{enumerate}
\item the 4-manifold $\partial W^5$ is spin if and only if $i^*(w_2)(\alpha)=0$ for all $\alpha\in H_2(\partial W^5,\matZ_2)$;
\item the 4-manifold $\partial W^5$ is even if and only if $i^*(w_2)(\alpha)=0$ for all $\alpha\in H_2(\partial W^5,\matZ)$.
\end{enumerate}
Note that $i_*:H_2(\partial W^5) \to H_2(W^5)$ is surjective (because $H_2(W^5,\partial W^5) \isom H^3(W^5) \isom H^3(X_\calP) = 0$). Of course we have $i^*(w_2) (\alpha) = w_2(i_*(\alpha))$ for all $\alpha$. We can thus modify the two assertions above as follows.

\begin{enumerate}
\item the 4-manifold $\partial W^5$ is spin if and only if $w_2(\alpha)=0$ for all $\alpha\in H_2(W^5,\matZ_2)$;
\item the 4-manifold $\partial W^5$ is even if and only if $w_2(\alpha)=0$ for all $\alpha\in H_2(W^5,\matZ)$.
\end{enumerate}
We identify the homologies of $W^5$ and $X_\calP$. 
Let us now consider the case $\calP = \calC_n = \langle a |a^n \rangle$. 
We have the following.
$$H_2(X_{\calC_n}, \matZ_2) = \left\{\begin{array}{ll} \matZ_2 & {\rm \ if\ } n {\rm \ is \ even,} \\
0 & {\rm \ if\ } n {\rm \ is \ odd.} \end{array}\right. 
$$
If $n$ is odd, there is only one spin 5-dimensional thickening and $\partial W^5$ is a spin manifold, which we denote by $C_n^0$. If $n$ is even, we have two possibilities: one spin manifold $C_n^0$ and one non-spin manifold $C_n^1$. We have $H_2(X_{\calC_n},\matZ)=0$: by what just said, the manifold $C_n^1$ is even.

Let us turn to dihedral manifolds, \emph{i.e.}~to $\calP = \calD_{2n} = \langle a, b| a^2, b^2, (ab)^n\rangle$. We first consider the very symmetric case $n=2$. We may picture $X=X_{\calD_4}$ after a small 3-deformation as a pair-of-pants with 3 projective planes attached. We have 
\begin{eqnarray*}
H_2(X,\matZ_2) & = & \matZ_2 + \matZ_2 + \matZ_2, \\
H_2(X,\matZ) & = & \matZ.
\end{eqnarray*}
A basis for $H_2(X,\matZ_2)$ is given by the three projective planes. We then get a dual basis for $H^2(X,\matZ_2$). The modulo-2 map $H_2(X,\matZ) \to H_2(X,\matZ_2)$ sends 1 to $(1,1,1)$. 
Up to symmetries of $X$, there are four choices for $w_2 \in H^2(X,\matZ_2)$:
\begin{enumerate}
\item $(0,0,0)$ leads to a spin manifold $D_n^0$;
\item $(1,0,0)$ leads to a non-spin odd manifold $D_n^1$;
\item $(1,1,0)$ leads to a non-spin even manifold $D_n^2$;
\item $(1,1,1)$ leads to a non-spin odd manifold $D_n^3$.
\end{enumerate}
We need to distinguish $D_n^1$ from $D_n^3$. We do this by looking at their index-two coverings. Each $D_n^i$ has three such coverings, and it turns out that the number of spin manifolds among them is $3-i$. This is easily seen as follows: each covering $\pi:\tilde X \to X$ is determined by the choice of one projective plane $P$ in $X$. The polyhedron $\tilde X$ contains two projective planes fibering over $P$ and two spheres fibering over the two other projective planes in $X$. These four surfaces generate $H_2(\tilde X, \matZ_2)$. 

Let $p:\tilde W \to W$ be the covering of thickenings. We have $p^*(w_2)(\alpha) = w_2(p(\alpha))$. If $\alpha$ is a sphere, it double-covers a projective plane $P'\subset X$ and we have $w_2(p(\alpha)) = w_2(2P') = 0$. If $\alpha$ is a projective plane over $P$ we get $p^*(w_2)(\alpha) = w_2(P)$. Thus $\tilde X$ is spin iff $w_2(P)=0$. Therefore $D^i_n$ has $3-i$ spin coverings of index two.

The other dihedral manifolds are treated similarly. We always take $X$ to be a pair-of-pants with three discs attached along its boundary, winding 2, 2, and $n$ times. If $n>2$ is odd, we get
\begin{eqnarray*}
H_2(X,\matZ_2) & = & \matZ_2 + \matZ_2, \\
H_2(X,\matZ) & = & \matZ.
\end{eqnarray*}
The modulo-2 map $H_2(X,\matZ) \to H_2(X,\matZ_2)$ sends 1 to $(1,1)$. Up to symmetries we have three choices for $w_2$: 
\begin{enumerate}
\item $(0,0)$ leads to a spin manifold $D_n^0$;
\item $(1,0)$ leads to a non-spin odd manifold $D_n^1$;
\item $(1,1)$ leads to a non-spin even manifold $D_n^2$.
\end{enumerate}
If $n>2$ is even, we get
\begin{eqnarray*}
H_2(X,\matZ_2) & = & \matZ_2 + \matZ_2 + \matZ_2, \\
H_2(X,\matZ) & = & \matZ.
\end{eqnarray*}
A basis for $H_2(X,\matZ_2)$ is given by two projective planes and one 2-cell winding $n$ times. The modulo-2 map $H_2(X,\matZ) \to H_2(X,\matZ_2)$ sends 1 to $(0,0,1)$ or $(1,1,1)$, depending on whether $n/2$ is even or odd. Suppose $n/2$ is even. Up to symmetries we have six choices for $w_2$: 
\begin{enumerate}
\item $(0,0,0)$ leads to a spin manifold $D_n^{00}$;
\item $(1,0,0)$ leads to a non-spin even manifold $D_n^{10}$;
\item $(1,1,0)$ leads to a non-spin even manifold $D_n^{20}$;
\item $(0,0,1)$ leads to a non-spin odd manifold $D_n^{01}$;
\item $(1,0,1)$ leads to a non-spin odd manifold $D_n^{11}$;
\item $(1,1,1)$ leads to a non-spin odd manifold $D_n^{21}$.
\end{enumerate}
To distinguish them, we look at coverings determined by non-normal subgroups $H$ of order two. Up to conjugacy, there are only two such groups, generated by $a$ and $b$. Thus we get two coverings. As above, we see that the number of spin coverings of $D_n^{ij}$ is $2-i$, and we are done. When $n/2$ is odd the discussion is the same, except for $(1,0,1)$ and $(1,0,0)$ that are swapped:
\begin{enumerate}
\item $(1,0,1)$ leads to a non-spin even manifold $D_n^{10}$;
\item $(1,0,0)$ leads to a non-spin odd manifold $D_n^{11}$.
\end{enumerate}

Finally, the same arguments show that the universal covering of each such manifold is spin, since $H_2(\tilde X,\matZ_2)$ has a basis generated by spheres which cover an even number of times the elements in $H_2(X,\matZ_2)$. Such a manifold is still a graph manifold generated by $\calS_0$, and thus it must be $\#_k(S^2\times S^2)$.
\end{proof}

\subsection{Outline of the proof of Theorem \ref{main:teo}} \label{outline:subsection}
Theorem \ref{main:teo} says that $c(M)=0$ if and only if $M=M'\#_h\matCP^2$ for some graph manifold $M'$ generated by $\calS_0$ and some integer $h$. It is easy to see that every manifold of type $M'\#_h\matCP^2$ has indeed complexity zero using the following result. (A more detailed proof will be given in Section \ref{operations:section}, see Theorem \ref{easy:teo}.)

\begin{figure}
\begin{center}
\includegraphics[width = 9 cm] {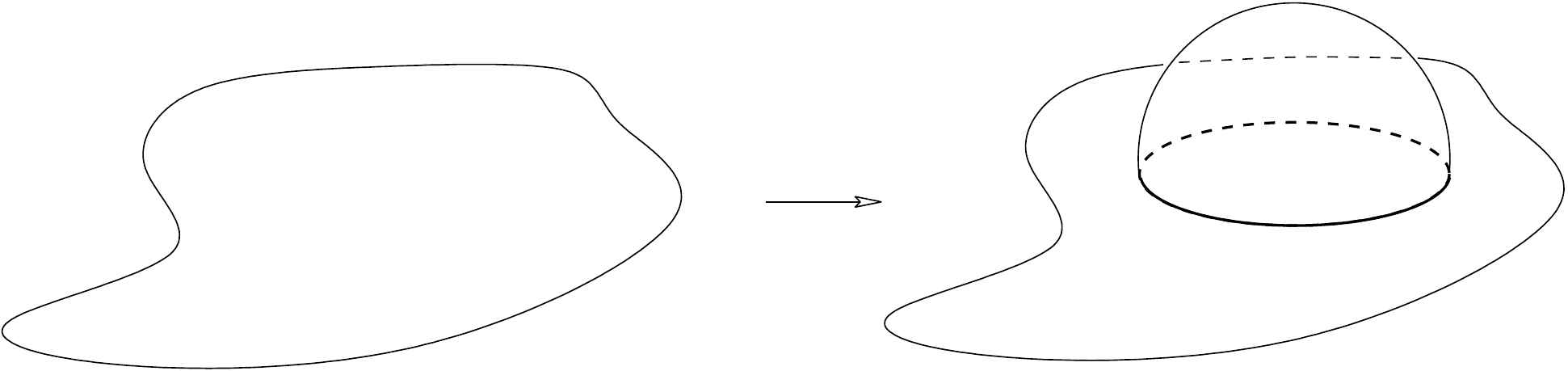}
\nota{By adding a bubble on each region we construct a shadow for the double $DM$.}
\label{bubble:fig}
\end{center}
\end{figure}

\begin{prop} \label{bubble:prop}
Let a compact orientable 4-manifold $M$ collapse onto a simple polyhedron $X\subset \interior{M}$ without boundary. A shadow $DX$ for the double $DM$ of $M$ is constructed from $X$ by adding a bubble on each region as in Fig.~\ref{bubble:fig}.
\end{prop}
\begin{proof}
By Proposition \ref{graph2:prop} we may suppose that $X$ is locally flat. We have two mirror copies $X_1$ and $X_2$ of $X$ inside $DM$. The complement of a regular neighborhood of $X_1$ in $DM$ collapses onto $X_2$. 

Take one point $x$ inside each region of $X_1$. Since $M$ collapses onto $X$, for each $x$ there is a natural properly embedded 2-disc $D\subset M$ intersecting $X_1$ in $x$. Its double gives a 2-sphere $S_x\subset DM$. Let $X_1'$ be $X_1$ plus the union of all these spheres $S_x$, one for each region of $X_1$. The polyhedron $X_1'$ intersects $X_2$ transversely in one point in each region of $X_2$. Therefore the complement of a regular neighborhood of $X_1'$ in $DM$ collapses onto a 1-dimensional subpolyhedron of $X_2$. Thus this complement is made of 3- and 4-handles.

To get a shadow it remains to perturb the double points $x$. This can be done as in Fig.~\ref{perturb:fig} below. The resulting polyhedron $DX$ is simple and is thus a shadow of $DM$. The result of the perturbation is that $DX$ is $X$ plus one bubble on each region.
\end{proof}

\begin{cor}
Let a compact 4-manifold $M$ collapse onto a simple polyhedron $X$ with $n$ vertices. We have $c(DM)\leqslant n$.
\end{cor}
\begin{proof}
Bubbles do not add vertices to a simple polyhedron. Therefore
the shadow $DX$ for $DM$ has $n$ vertices.
\end{proof}

Proposition \ref{graph:prop} implies that every graph manifold generated by $\calS_0$ has complexity zero. A shadow for $\matCP^2$ is also easily described (a projective line, which is homeomorphic to $S^2$). Finally, complexity is subadditive on connected sums, that is
$$c(M\#N) \leqslant c(M) + c(N)$$
and it is hence clear that every manifold $M'\#_h\matCP^2$ in Theorem \ref{main:teo} has complexity zero.

Proving that these are the only manifolds we can get is considerably harder. In some sense, this result is quite surprising, because there are many complicate shadows without vertices of closed manifolds that are not of the type prescribed by Proposition \ref{bubble:prop}. Many of them do not contain bubbles at all. For instance, let $X$ be the union of two (real) projective planes with an annulus connecting two non-trivial loops as in Fig.~\ref{example:fig}. It is easy to see that such a polyhedron is a shadow of the manifold $C_2^1$ introduced in Proposition \ref{finite:prop}. However, it does not contain bubbles. 

\begin{figure}
\begin{center}
\includegraphics[width = 3 cm] {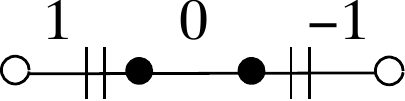}
\nota{Two projective planes connected with an annulus. This simple polyhedron without vertices is encoded by this graph. The integers encode the \emph{gleams} that are necessary to determine a 4-dimensional thickening, see Section \ref{shadows:section}. This is a shadow of $C_2^1$.  However, it is not of the type prescribed by Proposition \ref{bubble:prop}.}
\label{example:fig}
\end{center}
\end{figure}

The point is that there are various non-trivial moves that relate shadows of the same manifolds. The ones that we use here are collected in Fig.~\ref{move_all:fig} below (or equivalently Fig.~\ref{thickening:fig}). For instance, using move (5) we transform the polyhedron $X$ from Fig.~\ref{example:fig} into a projective plane with a bubble, which is indeed a shadow of the type prescribed by Proposition \ref{bubble:prop}.

Note that the graphs in the moves have (half-)integers decorating the edges. A shadow has a half-integer 
decorating each region called \emph{gleam}. Gleams make 4-dimensional thickenings much more complicate than 5-dimensional ones. Each of the listed moves can be applied only in presence of appropriate gleams.

The core proof of Theorem \ref{main:teo} consists of showing that every shadow $X$ without vertices of a closed 4-manifold can be transformed into a nice shadow with bubbles as in Proposition \ref{bubble:prop} by mean of the moves listed in the pictures. When we find a shadow with a bubble on each region we can conclude that $M$ is a graph manifold generated by $\calS_0$. (Bubbles of course have appropriate gleams.) In the transformations, we sometimes need to remove some $\matCP^2$-summands. 

To find the appropriate moves that transform a given complicated shadow $X$ into a nice shadow with bubbles we adapt to this setting a technique of Neumann and Weintraub \cite{Neu}. Neumann and Weintraub proved that a plumbing of spheres plus a 4-handle can only give rise to connected sums of $S^2\times S^2$ and $\matCP^2$. The point was that the boundary of such a plumbing is forced to be $S^3$ (in order for a 4-handle to be attached). The plumbing describes $S^3$ as a graph manifold (two solid tori connected by a chain of products $T\times [0,1]$). Since $S^3$ is a ``simple'' 3-manifold, a ``complicate'' description of $S^3$ as a graph manifold must simplify somewhere. 
Luckily, the simplification of the boundary graph manifold translates into a semplification of the plumbing, and they may proceed by induction.

We apply the same procedure here. Let $X$ be a complicate shadow without vertices of a closed 4-manifold $M$. The boundary $\partial N(X)$ of the thickening of $X$ must be homeomorphic to $\#_h(S^2\times S^1)$, in order for the 3- and 4-handles to be attached. This is a very restrictive condition. As noted by Costantino and Thurston \cite{CoThu}, the subdivision of $X$ into fundamental pieces described by Proposition \ref{simple:prop} induces a decomposition of $\partial N(X)$ as a graph manifold. Since $\#_h(S^2\times S^1)$ is relatively ``simple'', the description as a graph manifold must simplify somewhere. Hopefully, this simplification translates into a move that transforms $X$ into a simpler shadow for $M$, and we proceed by induction. Unfortunately, not all simplifications translate from $\partial N(X)$ to $X$, and more work has to be done.

During all the proof we use an approach similar to the one introduced in \cite{MaPe}. Namely, we extend the notion of shadows from closed manifolds to manifolds bounded by copies of $S^2\times S^1$: we call such a manifold a \emph{block}. When simplifying $X$, we sometimes discard some blocks that belong to $\calS_0$.

\section{Shadows} \label{shadows:section}
In this section we recall Turaev's definition of shadow \cite{Tu0, Tu}. We then focus on manifolds whose boundary is a (possibly empty) union of copies of $S^2\times S^1$, which we call \emph{blocks}. We then construct shadows for all the blocks contained in $\calS_0$ and $\matCP^2$.

\subsection{Shadows}
Let $M$ be a compact oriented 4-manifold (possibly with boundary) and $L\subset\partial M$ a (possibly empty) framed link. 
\begin{defn} \label{properly:defn}
A \emph{properly embedded} simple polyhedron $X$ in $(M,L)$ is a simple polyhedron $X\subset M$ such that $\partial X = X\cap \partial M = L$ and $X$ is locally flat in $M$, \emph{i.e.}~it is locally embedded as $Q\times \{0\} \subset D^3\times D^1$ where $Q\subset D^3$ is one of the models of Fig.~\ref{shadow:fig}. 
\end{defn}
\begin{rem} \label{horizontal:rem}
Let $X$ be a properly embedded simple polyhedron in a pair $(M,L)$. 
The boundary $\partial N(X)$ of a regular neighborhood $N(X)$ of $X$ has a
\emph{vertical} part $\partial_{\rm vert}N(X) = N(X)\cap\partial M$, 
consisting in some solid tori, and a \emph{horizontal} part $\partial_{\rm hor}N(X) = \overline {\partial N(X)\setminus \partial M}$. 
\end{rem}
We will often use the following terminology.
\begin{defn}
A \emph{1-handlebody} is a (possibly disconnected) oriented 4-manifold made of 0- and 1-handles. 
\end{defn}
Every connected component of a 1-handlebody is  homeomorphic to either $D^4$ or the boundary-connected sum of some copies of $D^3\times S^1$. 

\subsection{Gleams}
Let $X$ be a simple polyhedron properly embedded in some pair $(M,L)$. Every region of $X$ is naturally equipped with a half-integer called \emph{gleam}, defined by Turaev in \cite{Tu}. We recall its definition here.

\begin{figure}
\begin{center}
\includegraphics[width = 12.5 cm] {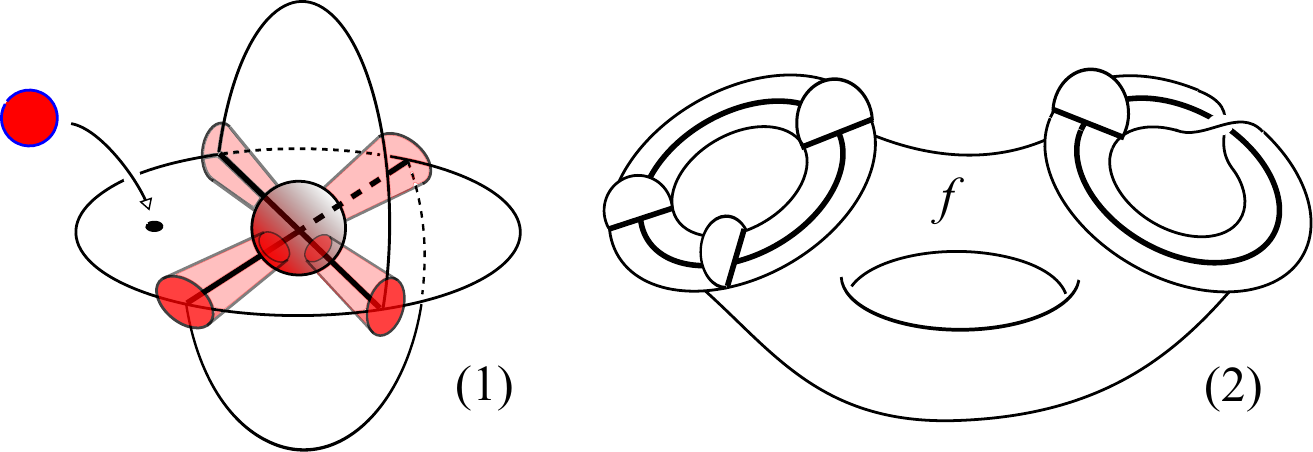}
\nota{The regular neighborhood $N(X)$ of $X$ decomposes into a 1-handlebody intersecting $X$ into a regular neighborhood of $SX$, and a disc bundle over the rest of $X$ (1). The polyhedron $X$ induces a $D^1$-fibering on each component of $\partial f$, which may be trivial or twisted (2). }
\label{two-component:fig}
\end{center}
\end{figure}

The singular part of $X$ thickens to a 1-handlebody. The rest of $X$ consists of some regions $f_1, \ldots, f_k$: each $f_i$ thickens to a $D^2$-bundle over $f_i$, see Fig.~\ref{two-component:fig}-(1). Take one $f=f_i$. The gleam of $f$ is defined by comparing this disc bundle with the interval bundle over $\partial f$ induced by $X$, see Fig.~\ref{two-component:fig}-(2). This is done as follows. 

The boundary of the $D^2$-bundle $B$ over $f$ consists of a horizontal part $\partial_{\rm hor} B$, a $S^1$-bundle over $f$, and a vertical part $\partial_{\rm vert}B$, the $D^2$-bundle over $\partial f$. The 3-manifold $\partial_{\rm hor}B$ is oriented as the boundary of $B$, which is in turn oriented since $M$ is.

Fix a section $s$ of the $S^1$-bundle $\partial_{\rm hor}B$ over $f$ and an orientation on the $S^1$-fiber. The section $s$ induces on each boundary torus $T_i$ of $\partial_{\rm hor}B$ a homology basis $(\mu_i,\lambda_i)$ such that $\lambda_i$ is the oriented fiber and $\mu_i$ is contained in $\partial s$ and oriented so that $(\mu_i,\lambda_i)$ is a positive basis (with respect to the orientation on $T_i$ induced by the one of $\partial_{\rm hor}B$).

Let $\gamma_i$ be one component of $\partial f$. If $\gamma_i$ is a component of $L$, the framing of $L$ induces a trivial $D^1$-subbundle of the $D^2$-bundle over $\gamma$. If $\gamma_i$ is not in $L$, there is a $D^1$-subbundle on $\gamma_i$ induced by $X$, which might be twisted: see Fig~\ref{two-component:fig}-(2). In both cases we get a $S^0$-subbundle of the $S^1$-bundle $\partial_{\rm hor}B$ over $\partial f$. If the $S^0$-bundle is trivial, it consists of two parallel curves which are homologically described as $\mu_i+e_i\lambda_i$ for some integer $e_i$. If the bundle is twisted, it consists of one curve, homologically described as $2\mu_i + \bar e_i\lambda_i$ for some odd integer $\bar e_i$. In this case we set $e_i = \bar e_i/2$. 

If $f$ has at least one boundary component, the \emph{gleam} of $f$ is defined as $\sum e_i$. (It does not depend on the chosen section and orientation on the $S^1$-fiber.) When $X=f$ is a closed surface, the gleam is defined as the Euler number $e$ of the $S^1$-fibration over $X$. If $X$ is orientable, this equals the self-intersection $[X]\cdot [X]$. 

Let a region $f$ of $X$ be \emph{odd} or \emph{even} if the number of twisted $D^1$-bundles on $\partial f_0$ is respectively odd or even. (This notion depends only on $X$ and not on its embedding.)
Note that the gleam of $f$ is an integer or a half-odd, depending on whether $f$ is even or odd. 

\begin{rem} \label{orientation:rem}
If the orientation of $M$ is switched, all gleams change by a sign.
\end{rem}

\begin{rem} \label{clockwise:rem}
The frame of $L$ determines the gleams of the adjacent faces. 
If we change the frame of a component of $L$ by a clockwise twist, the gleam of the adjacent face of $X$ changes by $+1$. 
\end{rem}

\subsection{Shadows}
The following definition is due to Turaev.
\begin{defn}
A \emph{shadow} is a simple polyhedron with boundary equipped with an integer (resp.~half-odd) decorating each even (resp.~odd) region.
\end{defn}
The discussion above shows that a simple polyhedron $X$ properly embedded in a pair $(M,L)$ is naturally a shadow. A converse holds. We say that the pair $(M,L)$ is a \emph{thickening} of $X$ if $M$ collapses onto $X$.

\begin{prop}[Turaev \cite{Tu}] \label{Turaev:prop}
Every shadow has a unique thickening up to homeomorphism.
\end{prop}
Recall that every homeomorphism is implicitely assumed piecewise-linear. The boundary $\partial M$ of a thickening decomposes into a horizontal and vertical part, see Remark \ref{horizontal:rem}.

\subsection{Blocks} \label{blocks:subsection}
The only pairs $(M,L)$ we consider in this paper are the following.

\begin{defn} \label{block:defn}
A \emph{block} is a compact 4-manifold $M$ with (possibly empty) boundary made of some copies of $S^2\times S^1$. A \emph{framed block} is a pair $(M,L)$ where $M$ is a block and $L$ consists of one fiber $\{pt\}\times S^1$ on each boundary component, with some framing.
\end{defn}
The link $L$ of a famed block $(M,L)$ is in fact determined up to isotopy by the block $M$, but its framing is not. The notion of shadow of a closed manifold was introduced by Turaev in \cite{Tu}. We extend it to blocks, in the spirit of \cite{MaPe}. 

\begin{defn} \label{shadow:defn}
A properly embedded simple polyhedron $X$ in a block $(M,L)$ is a \emph{shadow of $(M,L)$} if $M$ is obtained from a regular neighborhood of $X\cup\partial M$ by adding $3$- and $4$-handles.
\end{defn}

When $M$ is closed, the link $L$ is empty and we get Turaev's definition. 

\begin{rem} \label{1:handlebody:rem}
A properly embedded simple polyhedron $X$ in $(M,L)$ is a shadow of $(M,L)$ if and only if $M\setminus \interior{N(X)}$ is a 1-handlebody.
\end{rem}

A well-known result of Laudenbach and Poenaru together with Proposition \ref{Turaev:prop} show that a shadow of a closed 4-manifold determines the manifold. This result can be extended to blocks. 
\begin{prop} \label{unique:prop}
Let $X$ be a shadow of some famed block $(M,L)$. The framed block is determined by the thickening $(N(X),L)$ of $X$, and hence by $X$ itself.
\end{prop}
\begin{proof}
The shadow $X$ determines its thickening $(N(X),L)$ by Proposition \ref{Turaev:prop}. The vertical boundary $\partial_{\rm vert}N(X)$ consists of one solid torus $V_i$ fibering on each component $\gamma_i$ of $\partial X$. We can reconstruct the full boundary $\partial M$ by attaching a mirror copy $V_i'$ of $V_i$ along $\partial V_i$, so that $V_i\cup V_i'\isom S^2\times S^1$, see Fig.~\ref{determina:fig}. 

\begin{figure}
\begin{center}
\includegraphics[width = 12.5cm]{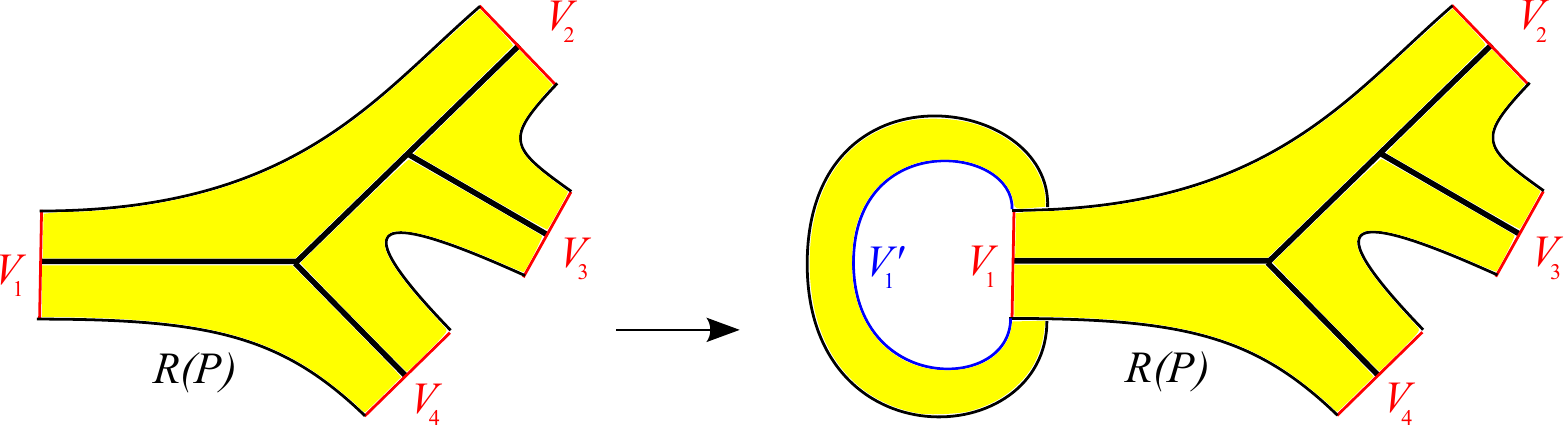}
\nota{How to reconstruct $(M,L)$ from $X$. Each vertical solid torus $V_i\subset\partial N(X)$ is doubled, so that $V_i\cup V_i'\isom S^2\times S^1$. (Here, this is shown for $i=1$.) 
}
\label{determina:fig}
\end{center}
\end{figure}

The regular neighborhood $R= N(X\cup V_1'\cup \ldots \cup V_k') = N(X\cup\partial M)$ in $M$ is uniquely determined by collaring each $V_i'$. The complement of $R$ in $M$ consists of 3- and 4-handles: by Laudenbach-Poenaru's theorem \cite{LaPo} the manifold $M$ does not depend on the way these handles are attached. Finally, the link $L$ is $\partial X$ and its framing is determined by the gleams of the incident faces, see Remark \ref{clockwise:rem}.
\end{proof}

Proposition \ref{unique:prop} talks about uniqueness. Actually, its proof also shows the following existence result. Recall that the boundary of a connected 1-handlebody is homeomorphic  to $\#_h (S^2\times S^1)$, for some $h$.

\begin{prop} \label{bordo:prop}
Let $X$ be a shadow. It is the shadow of some block $(M,L)$ if and only if the boundary $\partial N(X)$ of its thickening is homeomorphic to $\#_h (S^2\times S^1)$ for some $h\geqslant 0$.
\end{prop}

\begin{rem} \label{abuse:rem}
Let $X$ be a shadow of some framed block $(M,L)$. By modifying the gleams on the regions incident to $L$ we get a shadow of the same block $M$, with a possibly different framing $L'$, see Remark \ref{clockwise:rem}. With a little abuse we therefore sometimes omit the gleams on these regions, and call the resulting partially decorated polyhedron a \emph{shadow} of the (unframed) block $M$. (The unframed link $L$ is determined by $M$, so we also omit it.) 
\end{rem}

\subsection{Examples} \label{examples:subection}
The 4-sphere has a shadow without vertices.
\begin{prop} \label{sphere:prop}
The 2-sphere with gleam 0 is a shadow for $S^4$.
\end{prop}
\begin{proof}
Its thickening is $S^2\times D^2$. By adding a 3- and a 4-handle we get $S^4$.
\end{proof}

Complex projective space and the blocks in $\calS_0$ have shadows without vertices.

\begin{prop} \label{plane:prop}
Any complex line is a shadow for $\matCP^2$. It is a 2-sphere with gleam 1.
\end{prop}
\begin{proof}
The complement of an open regular neighborhood is a disc. The gleam equals its self-intersection number.
\end{proof}

We turn to the blocks in $\calS_0$.

\begin{prop} \label{blocks:prop}
The (unframed) blocks
$$M_{11}, M_2, M_{111}, M_{12}, M_3, N_1, N_2, N_3$$
have shadows homeomorphic to (respectively)
$$Y_{11}, Y_2, Y_{111}, Y_{12}, Y_3, D^2, A^2, P^2.$$
\end{prop}
\begin{proof}
It is easy to find a natural proper embedding of each polyhedron in the corresponding block. The complement (of an open regular neighborhood) of each polyhedron is then easily seen to collapse onto a 1-dimensional polyhedron: this implies that it is a 1-handlebody; we are hence done by Remark \ref{1:handlebody:rem}. 
\end{proof}

\begin{rem} As an example, let us denote by $P^3$ the 3-dimensional pair-of-pants, \emph{i.e.}~the 3-sphere $S^3$ minus three open balls. We have $M_{111} = P^3\times S^1$. Let $Y$ be the cone over 3 points. The polyhedron $Y_{111}$ is homeomorphic to $Y\times S^1$. It is easy to visualize $Y_{111}$ as a shadow of $M_{111}$. Embed $Y$ inside $P^3$ as in Fig.~\ref{spider:fig}. Note that $P^3\setminus \interior{N(Y)} \isom D^3$. Therefore $M_{111}\setminus \interior{N(Y_{111})} \isom D^3\times S^1$, a 1-handlebody.
\end{rem}

\begin{figure}
\begin{center}
\includegraphics[width = 6cm]{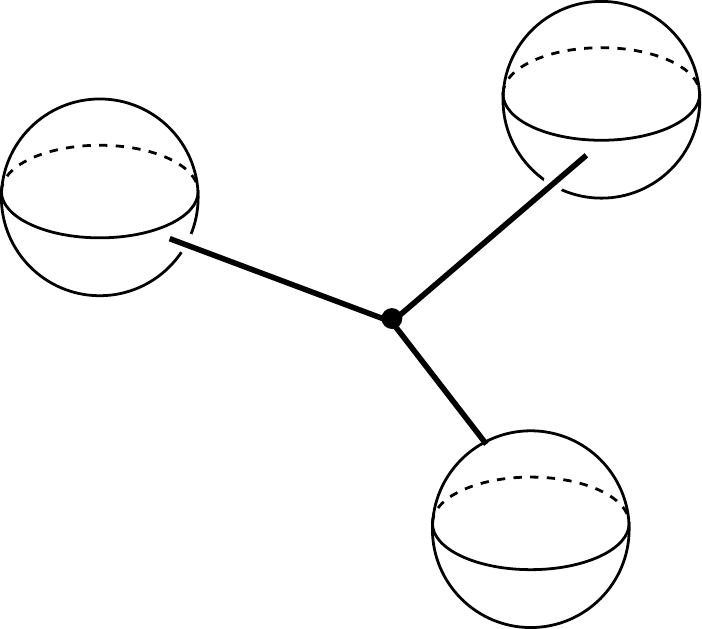}
\nota{Embed $Y$ in the 3-dimensional pair-of-pants $P^3$. The complement is an (open) disc.}
\label{spider:fig}
\end{center}
\end{figure}

\section{Operations with shadows} \label{operations:section}
Two blocks can be combined to produce a new block in two ways: by an internal connected sum, or by glueing two boundary components (the latter operation is called an \emph{assembling}, following the terminology of \cite{MaPe}). We show here how both these operations can be easily translated into some moves on shadows. An important feature of these moves is that they do not produce any new vertex. 

We recover another proof of the easy part of Theorem \ref{main:teo}, namely that every manifold of type $M'\#_h\matCP^2$ (with $M'$ graph manifold generated by $\calS_0$) has complexity zero. (Another proof was given in Subsection \ref{outline:subsection}.)

\subsection{Connected sum}
\begin{figure}
\begin{center}
\includegraphics[width = 9 cm]{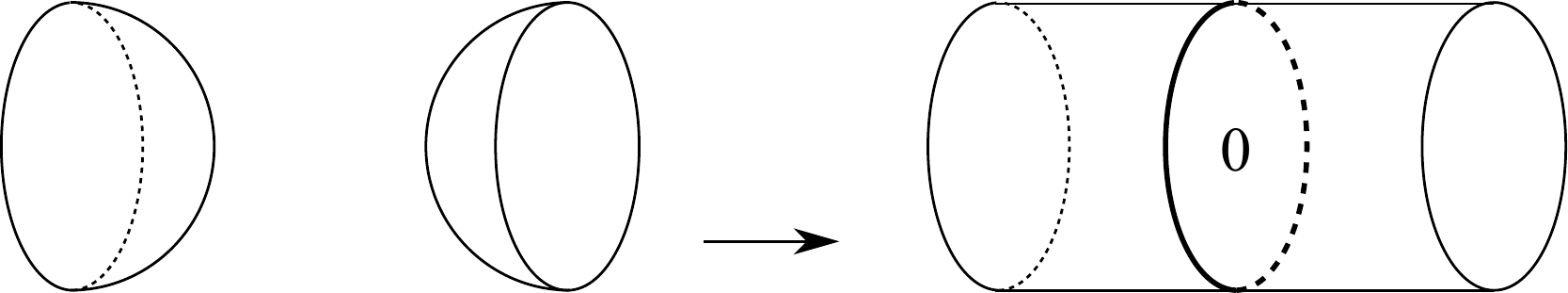}
\nota{This move on shadows corresponds to a connected sum of manifolds.}
\label{sum:fig}
\end{center}
\end{figure}

A \emph{connected sum} in a (possibly disconnected) framed block $(M,L)$ consists of removing the interiors of two $n$-discs and identifying the new boundary spheres via an orientation-reversing map. (We use this slightly more general definition instead of the usual one, where $M$ has two connected components each containing one ball.)

\begin{prop} \label{sum:prop}
The move in Fig~\ref{sum:fig} transforms a shadow $X_1$ of some framed block $(M_1,L_1)$ into a shadow $X_2$ of some other framed block $(M_2,L_2)$, and viceversa. The pair $(M_2,L_2)$ is a connected sum of $(M_1,L_1)$.
\end{prop}

\begin{proof}
Consider the 4-dimensional thickenings $N(X_1)$, $N(X_2)$ of $X_1$, $X_2$. Since the gleam of the disc is zero, the portion on the right embeds in a three-dimensional slice, \emph{i.e.}~in a 3-disc $D^3\subset N(X_2)$. The move in Fig.~\ref{sum2:fig} does not change the thickening of $X_2$. Therefore $N(X_2)$ is obtained from $N(X_1)$ by adding a 1-handle. This easily implies the assertion.
\end{proof}

\begin{figure}
\begin{center}
\includegraphics[width = 9 cm]{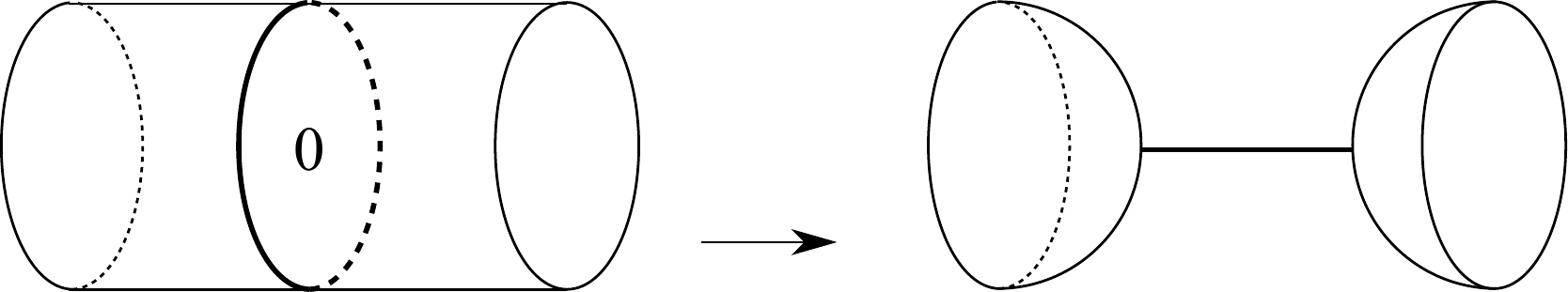}
\nota{This move does not modify the regular neighborhood of the polyhedron.}
\label{sum2:fig}
\end{center}
\end{figure}

\subsection{Immersed shadows}
An \emph{immersed shadow} is a properly embedded polyhedron $X$ in $(M,L)$ which is everywhere simple, except at finitely many double points. More precisely, the link of every point $x$ of $X$ is either a circle with three radii, a circle with a diameter, a circle, a segment, or two circles. We require implicitly as above that $X$ be locally flat, \emph{i.e.}~the star of each point is standardly embedded. The first 4 types must be embedded in a 3-dimensional slice as in Fig~\ref{shadow:fig}, and the new type is embedded as two transverse discs intersecting in $x$.

An immersed shadow $X$ is also equipped with gleams. It is naturally the image of a shadow $\widetilde X \to X$ along a map which is everywhere injective except at the double points. The regular neighborhood of $X$ in $M$ can be naturally pulled back to an abstract regular neighborhood $N(\widetilde X)$ of $\widetilde X$, which induces some gleams on $\widetilde X$. These gleams can then be projected to $X$. 

\begin{lemma} \label{perturb:lemma}
Every double point of $X$ can be locally perturbed as in Fig.~\ref{perturb:fig}, with the gleams changed as shown (there are two possible moves). The move does not change the regular neighborhood of the polyhedron.
\end{lemma}
\begin{proof}
Locally at the double point, the polyhedron $X$ consists of two transverse discs in $D^4$. Then $X$ intersects $S^3 = \partial D^4$ into a Hopf link. 

The move substitutes the two transverse discs with $A\cup D$, where $A\subset S^3$ is an annulus spanning the Hopf link and $D\subset D^4$ is a properly embedded 2-disc intersecting the core of $A$ in $\partial D$. Since the core of $A$ is an unknot in $S^3$, the disc $D$ is obtained simply by pushing inside $D^4$ a spanning disc in $S^3$.

The regular neighborhood does not change, because the removed piece (two transverse discs) and the new one $D\cup A$ both thicken to a 4-disc. 

There are two non-isotopic spanning annuli in the Hopf link, and they give rise to non-isotopic constructions. The gleam of $D$ is $\pm 1$ depending on the choice of $A$. The gleams of the incident faces are changed correspondingly as $\mp 1$. The gleams were calculated in \cite{CoThu}.
\end{proof}

\begin{figure}
\begin{center}
\includegraphics[width = 12.5 cm]{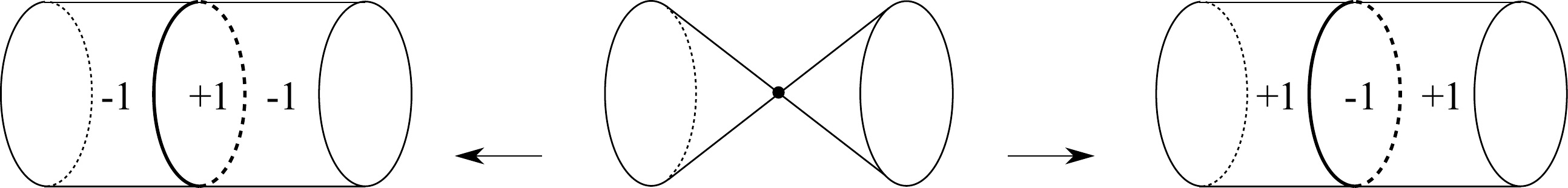}
\nota{A double point can be locally perturbed to a simple
  polyhedron. There are two ways to do this, and the resulting gleams depend on that choice.}
\label{perturb:fig}
\end{center}
\end{figure}

The perturbation is the analogue of \piu $\to$ \hacca in half dimensions (perturb a 4-valent vertex inside a surface: note that there are two possible moves also here). 

\subsection{Assembling} \label{assembling:subsection}
Let $(M,L)$ be a (possibly disconnected) framed block. Let $N_1$ and $N_2$ be two boundary components of $M$. Each component contains a framed knot. 
\begin{defn} An \emph{assembling} of $(M,L)$ is the operation of identifying $N_1$ and $N_2$ via a map $\psi$ which preserves the framed knots. The result of this operation is a new framed block $(M',L')$.
\end{defn}
We now investigate the effect of this operation on shadows. We will need the following result, proved in \cite{LaPo}.
\begin{lemma}\label{1:handlebody:lemma}
Every 2-sphere $\Sigma\subset \partial H$ in the boundary of
a $1$-handlebody $H$ bounds a properly embedded 3-disc $D^3\subset H$ such that $H\setminus
\interior {N(D^3)}$ is a 1-handlebody.
\end{lemma}
\begin{proof}
This is an easy consequence of Laudenbach-Poenaru's theorem \cite{LaPo} which states that every self-homeomorphism of $\partial H$ extends to $H$. Recall that the 1-handlebody need not to be connected.\end{proof}

\begin{figure}
\begin{center}
\includegraphics[width = 10 cm]{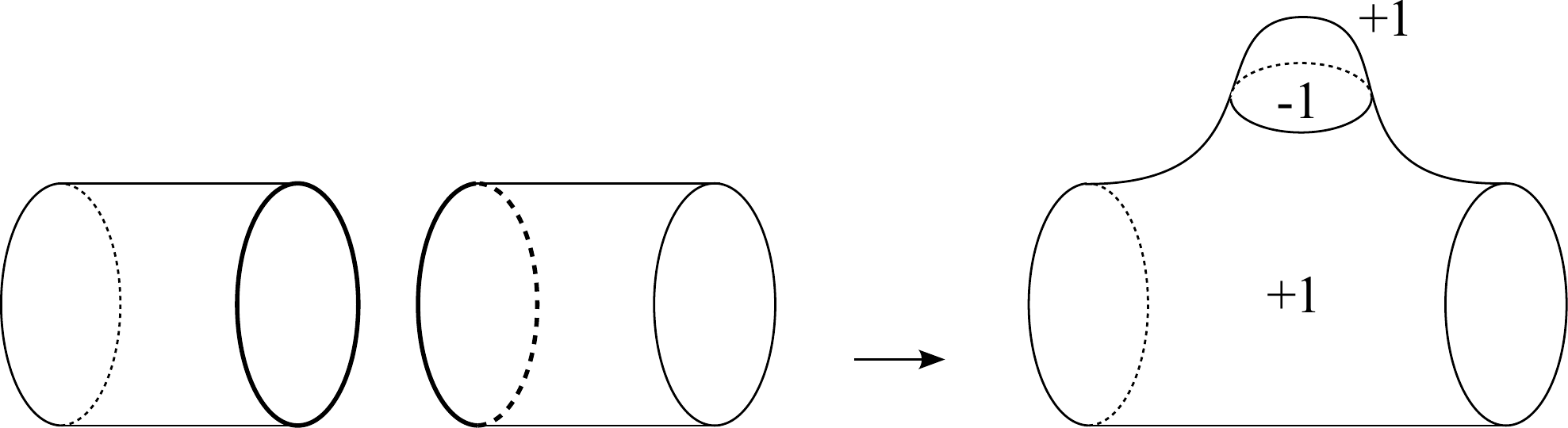}
\nota{This move on shadows represents and assembling of blocks. Two components of $\partial X$ are glued, and a bubble is added (with appropriate gleams).}
\label{assembling:fig}
\end{center}
\end{figure}

\begin{prop} \label{assembling:prop}
The move in Fig.~\ref{assembling:fig} transforms a shadow $X_1$ of some framed block $(M_1,L_1)$ into a shadow $X_2$ of some other framed block $(M_2,L_2)$, and viceversa. The pair $(M_2,L_2)$ is an assembling of $(M_1,L_1)$.
\end{prop}
\begin{proof}
The move in Fig.~\ref{perturb2:fig} transforms $X_2$ into an immersed simple polyhedron (with gleams) $Q\cup\Sigma$. Here $\Sigma$ is a 2-sphere with gleam zero. The regular neighborhoods of $X_2$ and $Q\cup\Sigma$ are the same by Lemma \ref{perturb:lemma}, so we may work with $Q\cup\Sigma$ instead of $X_2$.

\begin{figure}
\begin{center}
\includegraphics[width = 9 cm]{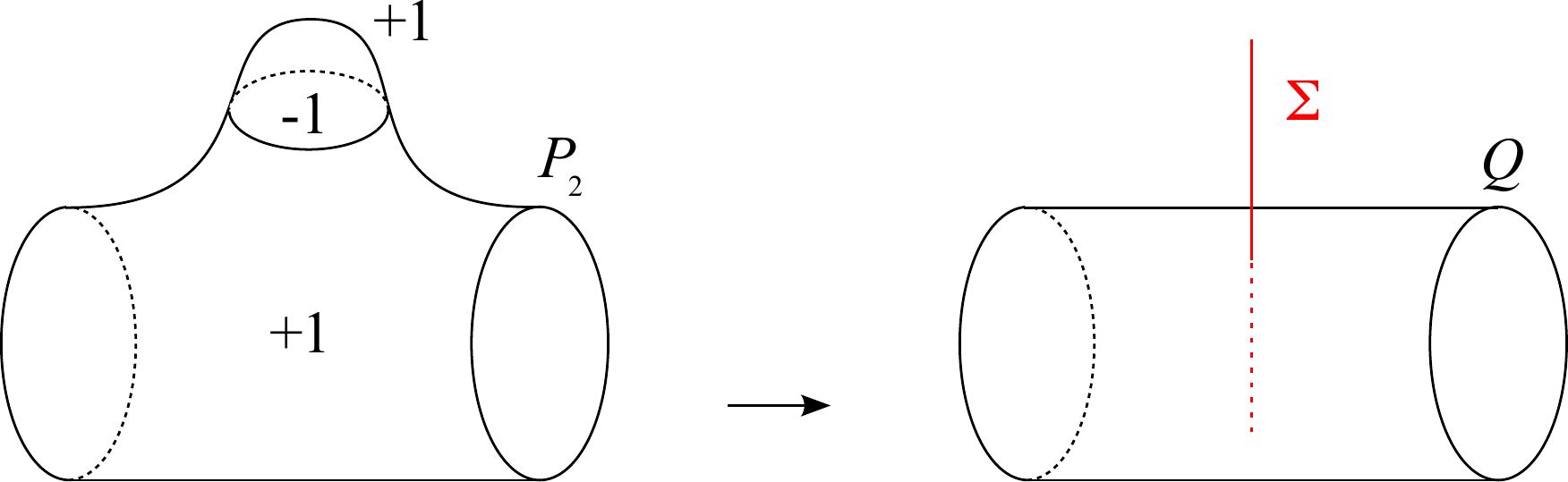}
\nota{This move does not change the regular neighborhood of the (immersed) simple polyhedron. Here $\Sigma$ is a 2-sphere with gleam zero.}
\label{perturb2:fig}
\end{center}
\end{figure}

Suppose $X_1$ is a shadow of some framed block $(M_1,L_1)$. The polyhedron $Q$ is obtained by gluing two components of $\partial X_1$ contained in two components $N'$, $N''$ of $\partial M_1$. This map can be extended to a unique homeomorphism between $N'$ and $N''$ which preserves the framing. Let $(M_2,L_2)$ be the result of such an assembling. 

We have a natural embedding $Q\subset M_2$. 
The components $N'$ and $N''$ glue to form a submanifold $N\subset M_2$ homeomorphic to $S^2\times S^1$ and intersecting $Q$ into $\{pt\}\times S^1$. Embed also $\Sigma$ as $S^2\times \{pt\}$, see Fig.~\ref{converse:fig}-(1).

Note that $N\setminus \interior{N(Q\cup \Sigma)}$ is homeomorphic to $D^3$. Therefore $M_2\setminus \interior{N(Q\cup\Sigma)}$ is obtained by adding a 1-handle to $M_1\setminus \interior{N(X_1)}$. Since the latter is a 1-handlebody, the former also is. By Lemma \ref{perturb:lemma} the regular neighborhood $N(Q\cup\Sigma)$ is isotopic to $N(X_2)$. Therefore $X_2$ is a shadow of $(M_2,L_2)$.

\begin{figure}
\begin{center}
\includegraphics[width = 12.5 cm]{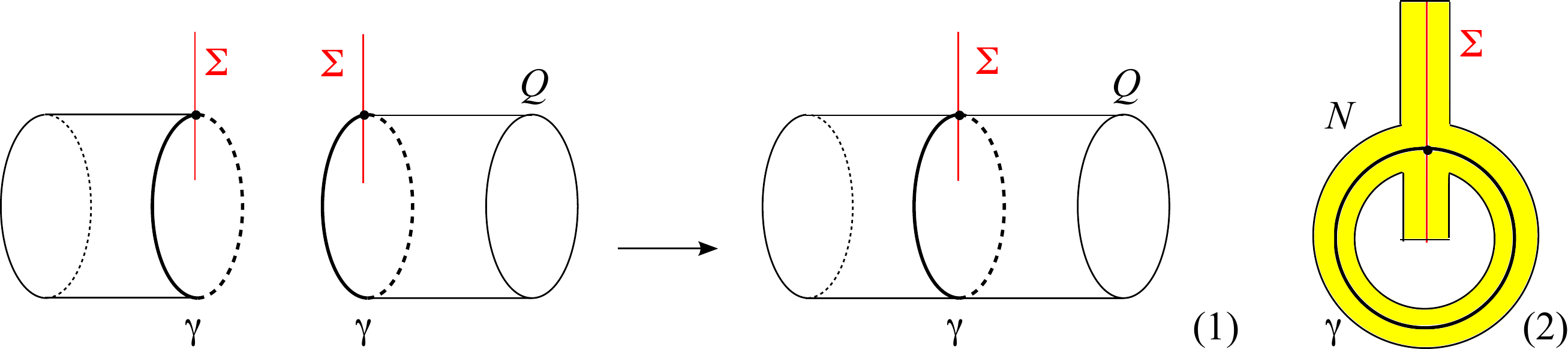}
\nota{Assembling with (immersed) shadows (1). We can take a normal regular neighborhood $N$ of $\gamma\cup\Sigma$. Its boundary is homeomorphic to $S^2$ (one annulus over $\gamma$ glued to two discs over $\Sigma$) (2)}
\label{converse:fig}
\end{center}
\end{figure}

The converse is proved similarly. Given $X_2$ shadow of $(M_2,L_2)$, we transform it into $Q\cup \Sigma$. The regular neighborhood $N(Q\cup\Sigma)$ has a 3-dimensional slice $N$ as in Fig.~\ref{converse:fig}-(2) homeomorphic to $S^2\times S^1$ minus an open ball. The boundary $\partial N$ is a 2-sphere in $\partial N(X)$. 

Since $M_2\setminus\interior{N(Q\cup\Sigma)}$ is a 1-handlebody $H$, it contains a properly embedded 3-disc $D^3$ with $\partial D^3 = \partial N$, such that $H\setminus \interior{N(D^3)}$ is again a 1-handlebody by Lemma \ref{1:handlebody:lemma}. Therefore $N\cup D^3\isom S^2\times S^1$ and by cutting $(M_2,L_2,Q)$ along $N\cup D^3$ we get a $(M_1,L_1,X_1)$, with $X_1$ a shadow for $(M_1,L_1)$ as required.
\end{proof}

\subsection{Filling}
The block $D^3\times S^1$ plays a particular role here.
We call the assembling of a framed block $(M,L)$ and a framed $D^3\times S^1$ along some component $N$ of $\partial M$ a \emph{filling} of $(M,L)$. This operation consists of attaching a 3-handle and a 4-handle to $N$, so by Laudenbach-Poenaru theorem \cite{LaPo}, the filled block depends only on $(M,L)$ and $N$.

In Section \ref{blocks:subsection} we have described some shadows of all the blocks involved in Theorem \ref{main:teo}, except $D^3\times S^1$. In some sense, the natural shadow for this block is the \emph{empty} shadow, whose complement in $D^3\times S^1$ is indeed made of 3- and 4-handles!  We adapt Proposition \ref{assembling:prop} to this particular situation.

\begin{figure}
\begin{center}
\includegraphics[width = 9 cm]{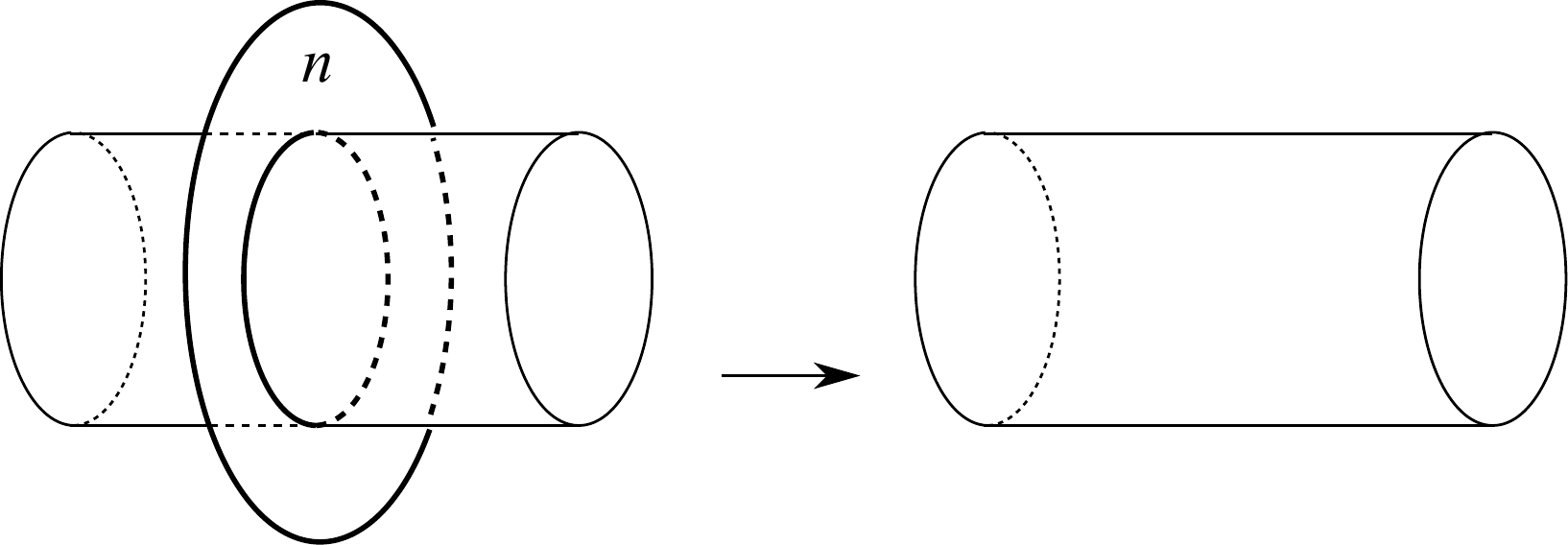}
\nota{This move on shadows represents the filling of a block. The removed annulus is adjacent to a component of $L$, whose framing is determined by its gleam $n$.}
\label{filling:fig}
\end{center}
\end{figure}

\begin{prop} \label{filling:prop}
The move in Fig.~\ref{filling:fig} transforms a shadow $X_1$ of some framed block $(M_1,L_1)$ into a shadow $X_2$ of some framed block $(M_2,L_2)$, and viceversa. The block $(M_2,L_2)$ is a filling of $(M_1,L_1)$.
\end{prop}
\begin{proof}
As suggested by Fig.~\ref{homeomorphic:fig}, there is a homeomorphism between $M_1\setminus N(X_1\cup\partial M_1)$ and $M_2\setminus N(X_2\cup\partial M_2)$. Therefore $X_1$ is a shadow if and only if $X_2$ is, and it follows easily that $M_2$ is obtained by filling $M_1$.
\end{proof}

\begin{figure}
\begin{center}
\includegraphics[width = 8 cm]{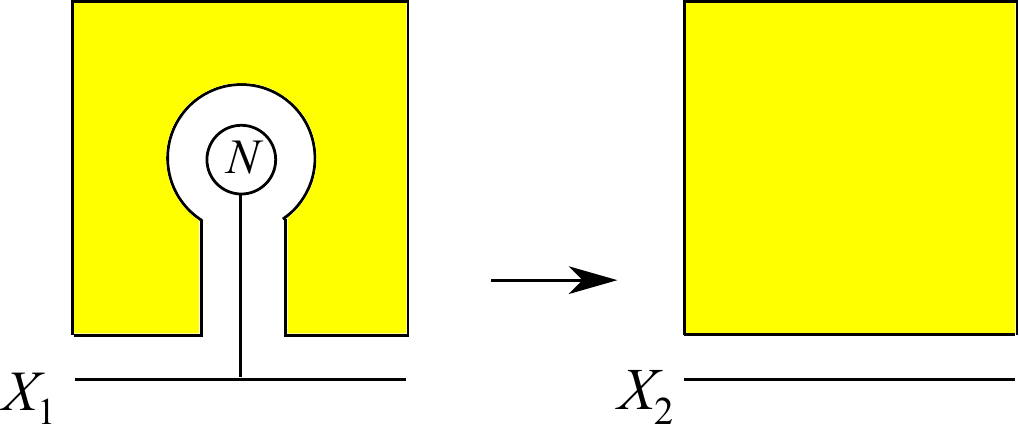}
\nota{How to pass from $X_1$ to $X_2$. Here $N$ denotes the component of $\partial M_1$ which is filled. The complements (of regular neighborhoods) are homeomorphic (painted in yellow). }
\label{homeomorphic:fig}
\end{center}
\end{figure}

\subsection{Complexity zero}
We can now prove again the easy half of Theorem \ref{main:teo} (another proof was given in Section \ref{outline:subsection}).
\begin{teo} \label{easy:teo}
Let $M$ be a graph manifold generated by $\calS_0$ and $h$ an integer. The manifold $M\#_h \matCP^2$ has complexity zero.
\end{teo}
\begin{proof}
By Proposition \ref{smaller:prop}, the manifold $M$ is a connected sum of $h\geqslant 0$ copies of $S^3\times S^1$ and $k\geqslant 0$ graph manifolds generated by 
$$\calS_0' = \big\{M_2, M_{111}, M_{12}, M_3, N_1, N_3\big\}.$$
If $h=k=0$ then $M=S^4$ which has a shadow without vertices, see Proposition \ref{sphere:prop}. The blocks in $\calS_0'$ and $\matCP^2$ also have shadows without vertices, see Propositions \ref{plane:prop} and \ref{blocks:prop}. Assemblings and connected sums translate into moves for shadow that do not produce vertices by Propositions \ref{sum:prop} and \ref{assembling:prop}.
\end{proof}

It remains to show that every closed oriented 4-manifold having complexity zero is of this type. The rest of the paper is devoted to the proof of this non-trivial fact.

\section{Moves} \label{moves:section}
We describe here some moves that relate two shadows of the same block. Some basic moves are well-known: these were discovered by Turaev and are shown in Fig.~\ref{move_all:fig}. 
The moves shown in Fig.~\ref{mossa_all:fig} are new and more useful in our vertex-free context: they are proved in this section. They are more efficiently encoded in Fig.~\ref{thickening:fig}.

\begin{figure}
\begin{center}
\includegraphics[width = 11 cm]{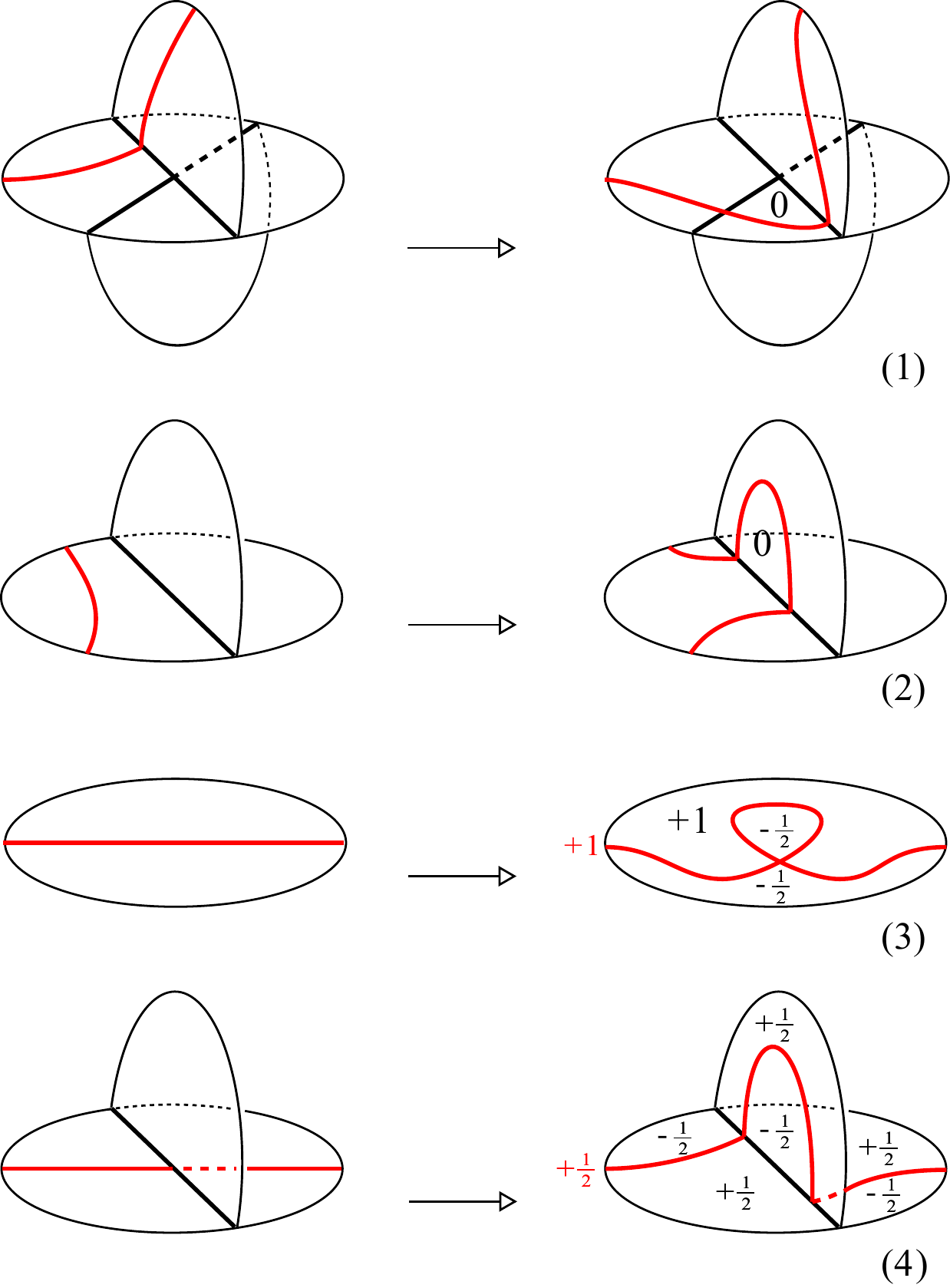}
\nota{These moves relate two shadows of the same block. A disc is attached along each red arc. Moves (1) and (2) can be embedded in a 3-dimensional slice, while the moves (3) and (4) cannot. Move (1) consists of the Matveev-Piergallini move, which is fundamental in the theory of spines of 3-manifolds. In moves (3) and (4), the gleam of the red region is modified after the move respectively by $+1$ and $+1/2$ (the number is pictured in red).}
\label{move_all:fig}
\end{center}
\end{figure}

\begin{prop}[Turaev \cite{Tu}] \label{moves:prop}
The moves in Fig.~\ref{move_all:fig} relate two shadows $X_1, X_2$ of the same block $(M,L)$.
\end{prop}
\begin{proof}
As shown by Turaev, the shadows $X_1$ and $X_2$ have homeomorphic thickenings. Therefore the blocks are also homeomorphic by Proposition \ref{unique:prop}.
\end{proof}

\begin{figure}
\begin{center}
\includegraphics[width = 12.5 cm]{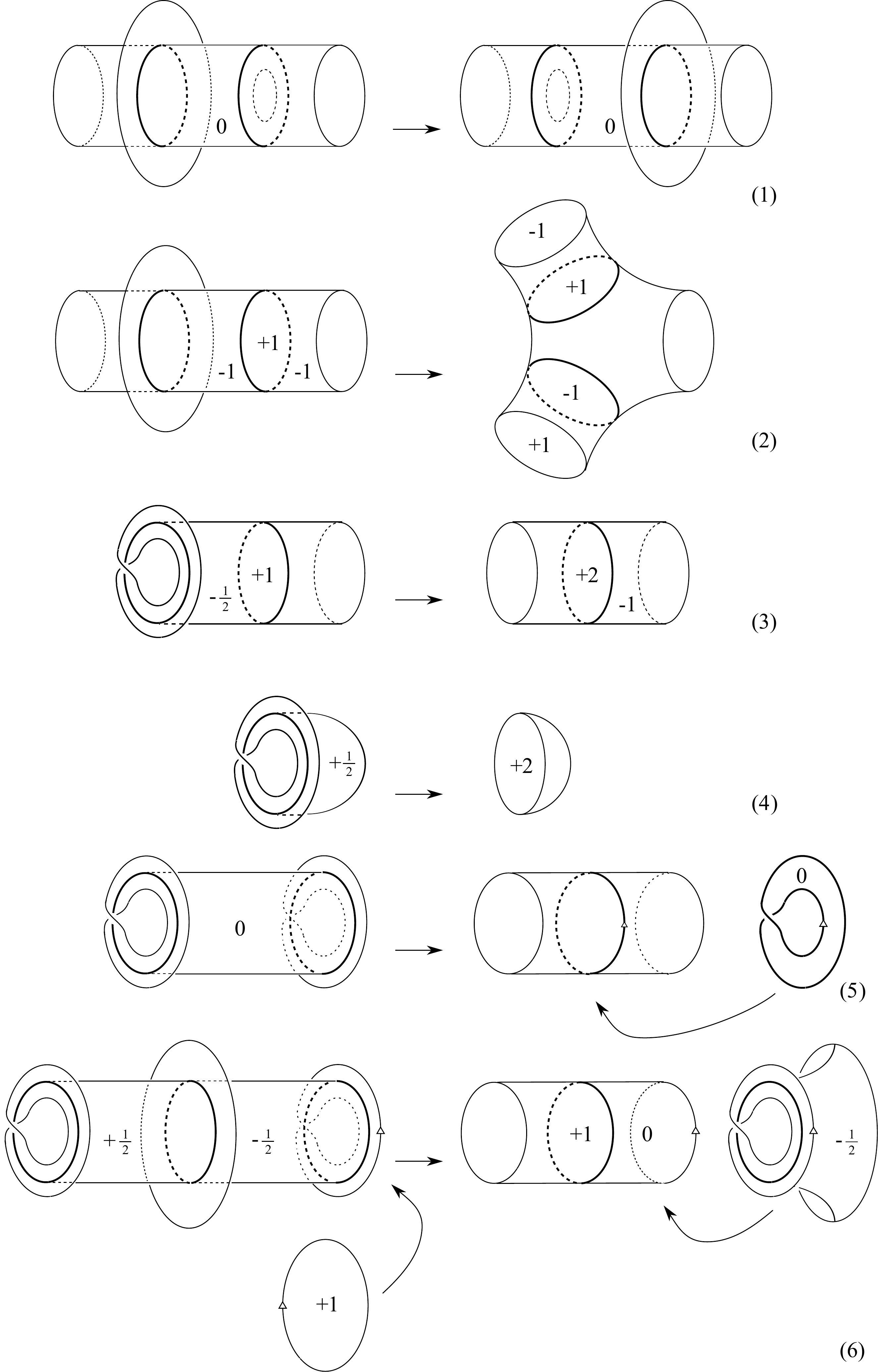}
\nota{These moves relate two shadows of the same block. In move (5) the polyhedron on the right is a M\"obius strip (with gleam zero) attached to the core of an annulus. Analogously, in move (6)-left a $+1$-gleamed disc is attached to the rightmost M\"obius strip producing a projective plane, and in (6)-right a M\"obius strip is attached as prescribed by the matching arrows. These moves are more efficiently encoded in Fig.~\ref{thickening:fig}.}
\label{mossa_all:fig}
\end{center}
\end{figure}

\begin{prop} \label{mosse:prop}
The moves in Fig.~\ref{mossa_all:fig} relate two shadows $X_1, X_2$ of the same block $(M,L)$.
\end{prop}
\begin{proof}
The annular region of both portions in Fig.~\ref{mossa_all:fig}-(1) have gleam zero. Therefore both portions may be embedded in a 3-dimensional slice $D^3$ as in the figure. Their regular neighborhoods are the same since they are so in $D^3$. (Alternatively, use Fig.~\ref{move_all:fig}-(2) a couple of times.)

\begin{figure}
\begin{center}
\includegraphics[width = 10 cm]{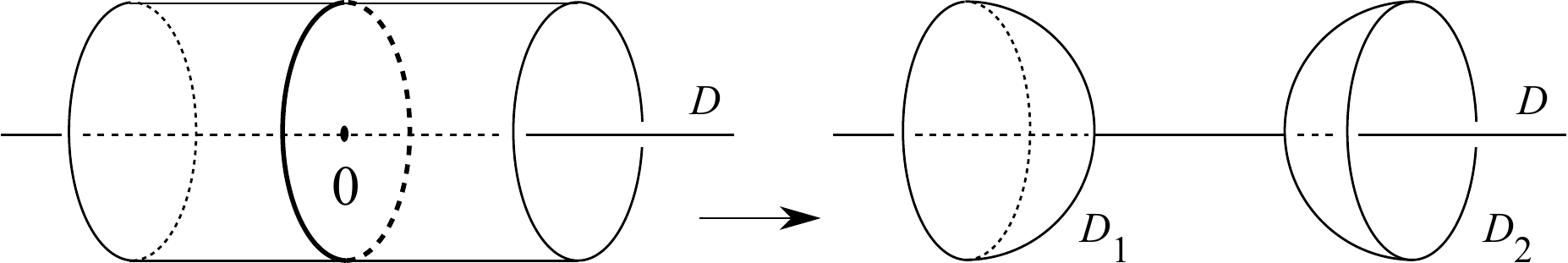}
\nota{Two intermediate steps of the move shown in Fig.~\ref{mossa_all:fig}-(2).}
\label{mossa2_proof:fig}
\end{center}
\end{figure}

The left portion in Fig.~\ref{mossa_all:fig}-(2) is the perturbation of the left portion $Q\cup D$ in Fig.~\ref{mossa2_proof:fig}, see Fig.~\ref{perturb:fig}. The portion $Q$ can be embedded in a 3-dimensional slice $D^3$ because the disc has gleam zero, and the disc $D$ intersects the slice in an arc, as in the figure. Apply the move in Fig.~\ref{sum2:fig} as in Fig.~\ref{mossa2_proof:fig}-right. The result is the union $D_1\cup D_2\cup D$ of three transverse discs. By perturbing the two intersection points $D_1\cap D$ and $D_2\cap D$ we get Fig.~\ref{mossa_all:fig}-(2)-right. (Alternatively, the move may also be obtained as a combination of the basic moves in Fig.~\ref{move_all:fig}.)

The portion of shadow in Fig.~\ref{mossa_all:fig}-(3)-left can also be drawn as in Fig.~\ref{triplet:fig}-left, with a $+1$-gleamed disc attached along the red circle. We can apply the moves shown in Fig.~\ref{triplet:fig}. In the resulting portion the disc delimited by the red circle has gleam $+1-1/2 = 1/2$. The new portion can be described as in Fig.~\ref{triplet2:fig}-left, with an annulus attached to the red circle. A final step is then shown in Fig.~\ref{triplet2:fig}.

\begin{figure}
\begin{center}
\includegraphics[width = 12.5 cm]{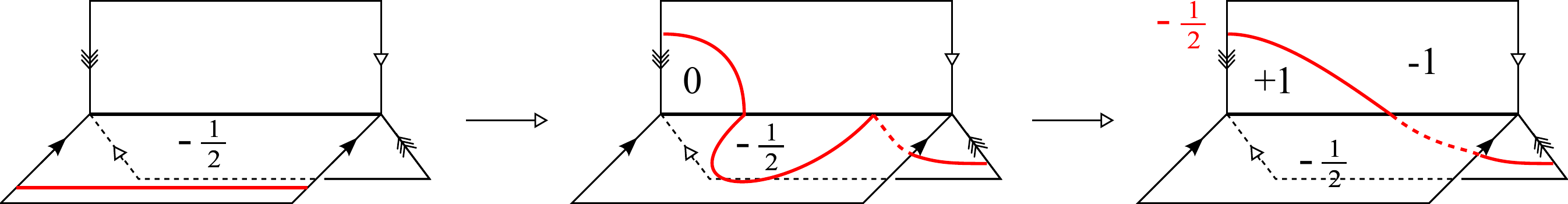}
\nota{Segments with matching arrows should be identified. We apply here Fig.~\ref{move_all:fig}-(2) and \ref{move_all:fig}-(4).}
\label{triplet:fig}
\end{center}
\end{figure}
\begin{figure}
\begin{center}
\includegraphics[width = 10 cm]{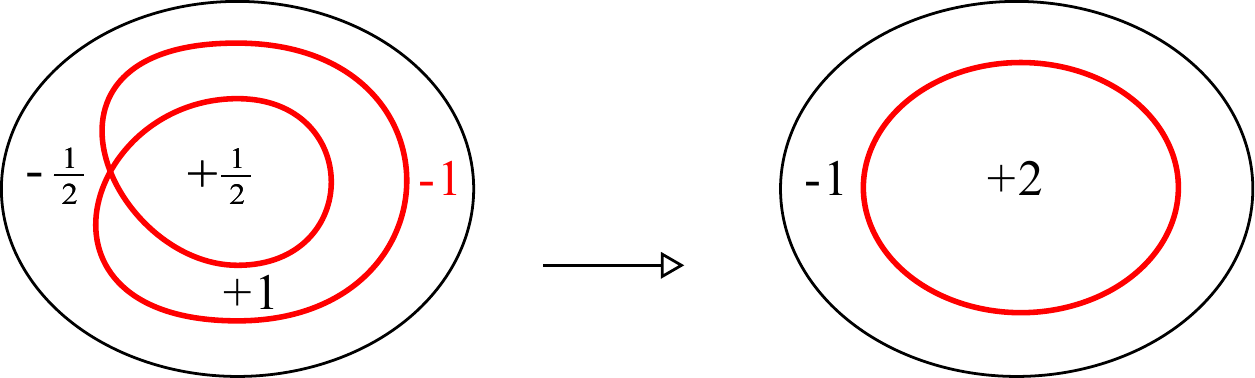}
\nota{Final step for Fig.~\ref{mossa_all:fig}-(3). Here we apply Fig.~\ref{move_all:fig}-(3).}
\label{triplet2:fig}
\end{center}
\end{figure}

The move in Fig.~\ref{mossa_all:fig}-(4) follows from the one in Fig.~\ref{mossa_all:fig}-(3): it suffices to add temporarily an auxiliary annulus in order to transform the portion in Fig.~\ref{mossa_all:fig}-(4)-left as in Fig.~\ref{mossa_all:fig}-(3)-left. 

\begin{figure}
\begin{center}
\includegraphics[width = 12.5 cm]{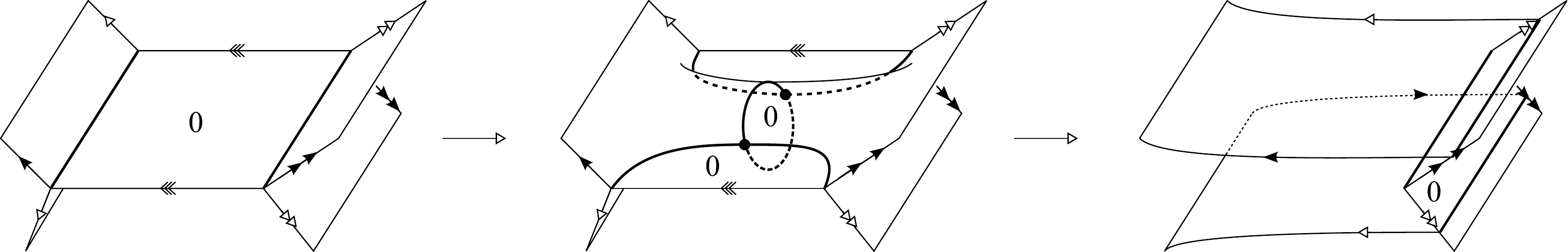}
\nota{Intermediate steps for Fig.~\ref{mossa_all:fig}-(5). Here we apply Fig.~\ref{move_all:fig}-(2) and its inverse.}
\label{mossa5_proof:fig}
\end{center}
\end{figure}

\begin{figure}
\begin{center}
\includegraphics[width = 11 cm]{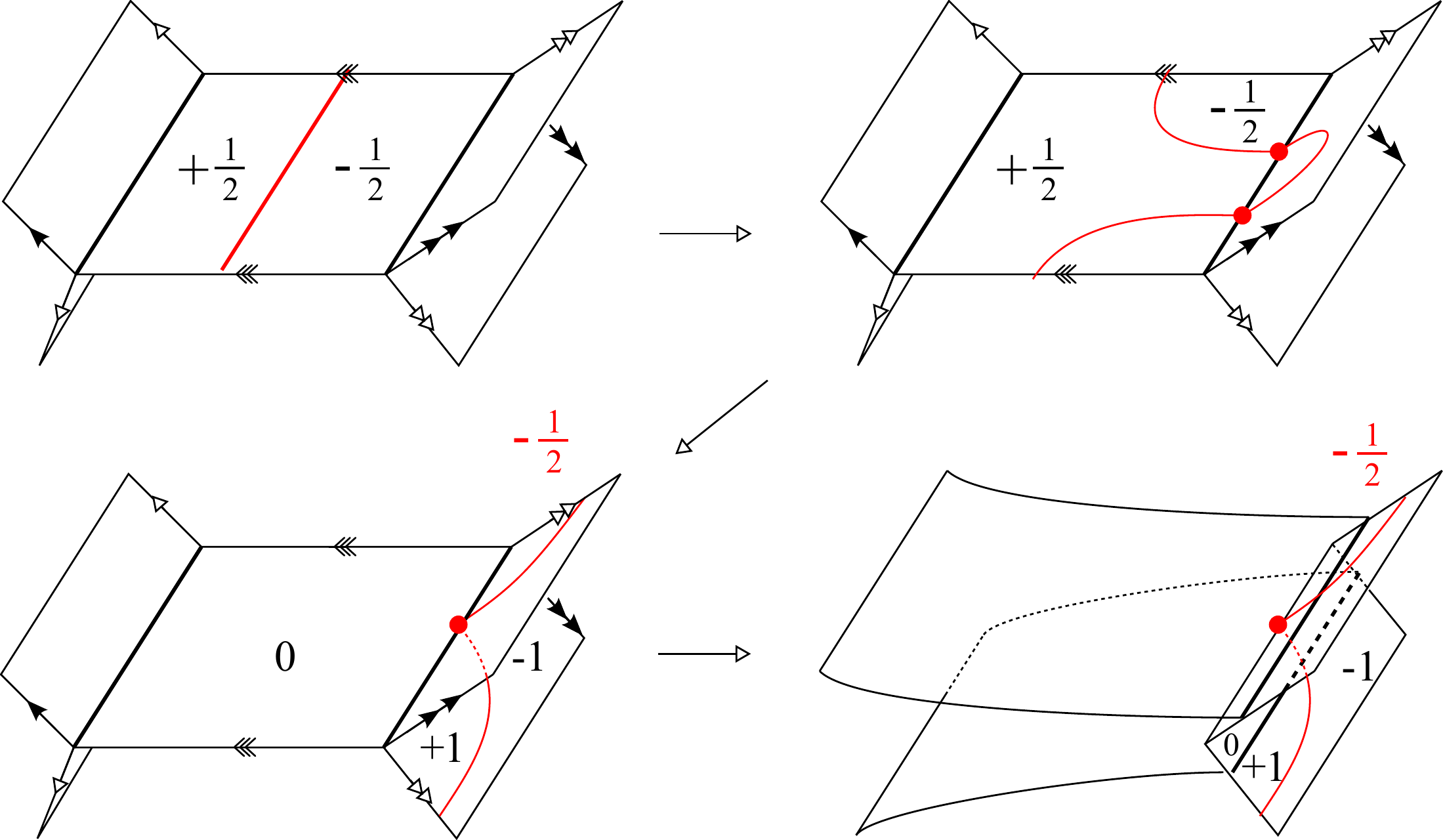}
\nota{Intermediate steps for Fig.~\ref{mossa_all:fig}-(6). We apply Fig.~\ref{move_all:fig}-(2), the opposite of Fig.~\ref{move_all:fig}-(4), and Fig.~\ref{mossa5_proof:fig} (the new vertex here can be ignored thanks to Fig.~\ref{move_all:fig}-(1)).}
\label{mossa6_proof:fig}
\end{center}
\end{figure}

\begin{figure}
\begin{center}
\includegraphics[width = 7 cm]{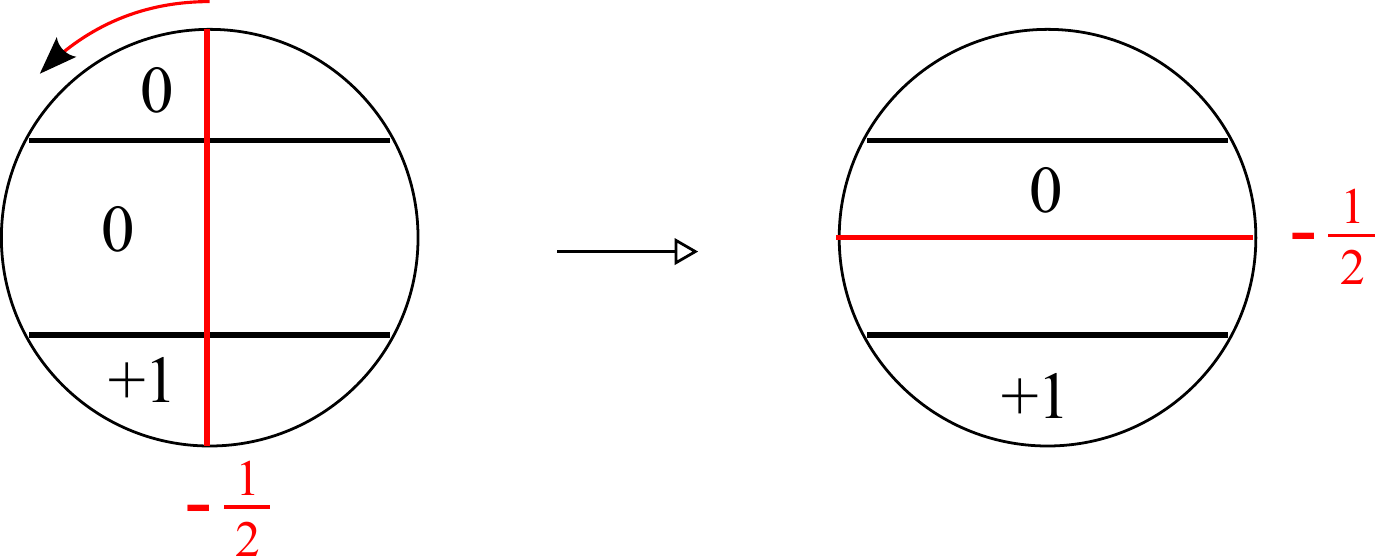}
\nota{Final step for Fig.~\ref{mossa_all:fig}-(6). We draw a projective plane as a disc with opposite boundary points identified. We apply Fig.~\ref{move_all:fig}-(2) and rotate counterclockwise the red curve.}
\label{mossa6_proof2:fig}
\end{center}
\end{figure}

The move in Fig.~\ref{mossa_all:fig}-(5) is constructed in Fig.~\ref{mossa5_proof:fig} as a composition of the move in Fig.~\ref{move_all:fig}-(2) and its inverse. The move in Fig.~\ref{mossa_all:fig}-(6) is constructed similarly: in order to apply Fig.~\ref{mossa5_proof:fig} we first slide away the vertical annulus as shown in Fig.~\ref{mossa6_proof:fig} (only the attaching of the annulus is shown, in red). Finally, note that a $+1$-gleamed disc is attached to the rightmost M\"obius band producing a projective plane: the projective plane and the two incident regions are drawn in Fig.~\ref{mossa6_proof2:fig}-left. We can turn the red segment counterclockwise as in Fig.~\ref{mossa6_proof2:fig} and get a portion as in Fig.~\ref{mossa_all:fig}-(2)-right, as required.
\end{proof}

\section{Shadows without vertices.} \label{without:section}
As shown in Section \ref{simple:subsection}, a simple polyhedron without vertices may be described via a graph. A shadow $X$ without vertices is thus encoded by a graph whose edges are decorated with half-integers. We summarize here briefly the moves introduced in the previous section using such decorated graphs. 

The boundary $\partial N(X)$ of the thickening of $X$ is a closed 3-manifold. As proved by Costantino and Thurston \cite{CoThu}, the graph describes correspondingly a decomposition of $\partial N(X)$ as a graph manifold. Such a decomposition is described at the end of this section.

\subsection{Decorated graph} \label{decorated:subsection}
Let a graph with vertices as in Fig.~\ref{vertices:fig} describe a simple polyhedron without vertices.
Let $e$ be an edge of the graph. If precisely one of its endpoints is incident to a vertex of type (5) as an unmarked edge, then the \emph{parity} of $e$ is odd. Otherwise, it is even.

\begin{defn}
A \emph{decorated graph} is a graph whose vertices are as in Fig.~\ref{vertices:fig}, and whose edges are decorated with half-integers. The half-integer decorating an edge $e$ is an integer or a half-odd, depending on the parity of $e$. 
\end{defn}

A graph determines a simple polyhedron $X$. Note that an edge of the graph determines a region of $X$. (Many edges may determine the same region.) A decorated graph determines a shadow: the gleam of a region is the sum of all the half-integers decorating the edges that determine that region. The parity of the edges was defined above in order to be coherent with the parity of the regions of $X$, so the result is indeed a shadow. 

Every simple shadow $X$ without vertices can be described by some decorated graph in this way. Such a graph is not really unique: some moves modify the graph while leaving the shadow unchanged, see Fig.~\ref{mosse_innocue:fig}.

\begin{figure}
\begin{center}
\includegraphics[width = 12.5 cm]{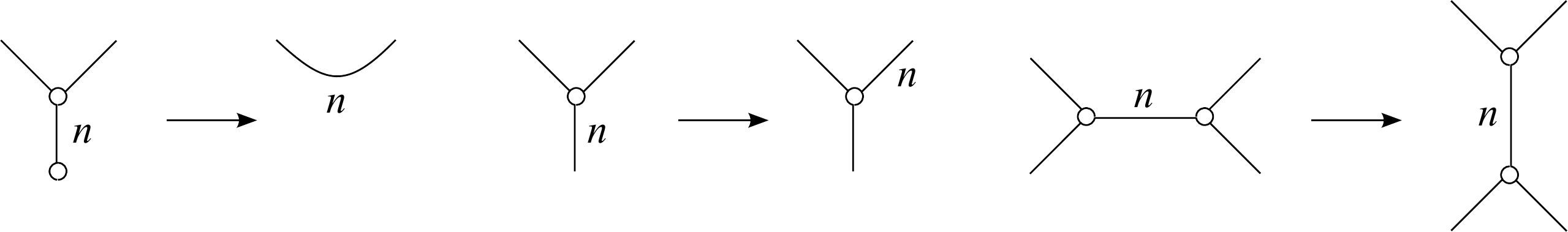}
\nota{These moves do not modify the shadow $X$.}
\label{mosse_innocue:fig}
\end{center}
\end{figure}

\begin{figure}
\begin{center}
\includegraphics[width = 10 cm]{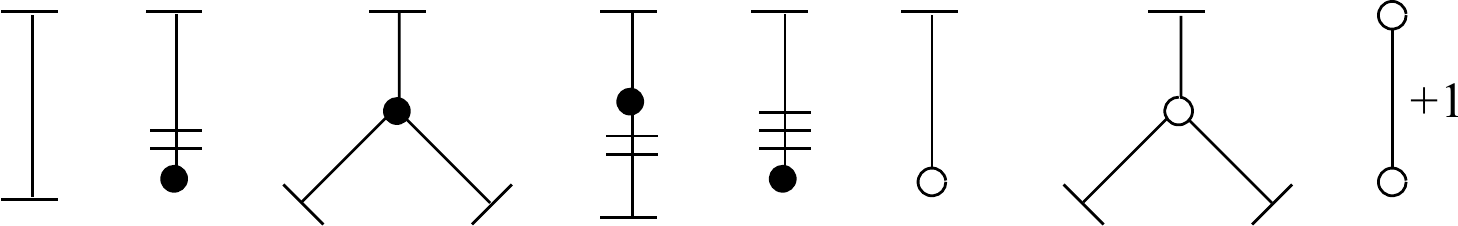}
\nota{The shadows $Y_{11} = A^2$, $Y_2$, $Y_{111}$, $Y_{12}$, $Y_3$, $D^2$, $P^2$, and $S^2$ of the blocks $M_{11} = N_2$, $M_2$, $M_{111}$, $M_{12}$, $M_3$, $N_1$, $N_3$, and $\matCP^2$. (See Propositions \ref{plane:prop} and \ref{blocks:prop}.) Decorations on the edges incident to flat vertices are omitted.
}
\label{blocks:fig}
\end{center}
\end{figure}

There are two types of 1-valent vertices \includegraphics[width = 0.6 cm]{0.pdf} and \includegraphics[width = 0.6 cm]{1.pdf}, and we call them respectively \emph{flat} and \emph{fat}. A flat vertex denotes a component of $\partial X$. When we want to describe a shadow $X$ of some (unframed) block $M$, we may omit decorations on the edges incident to flat vertices, according to Remark \ref{abuse:rem}. As an example, the graphs in Fig.~\ref{blocks:fig} describe the shadows of the blocks in $\calS_0$ and of $\matCP^2$, see Propositions \ref{plane:prop} and \ref{blocks:prop}.

\subsection{Moves}
The moves described in Section \ref{moves:section} can be easily visualized using decorated graphs. 

\begin{figure}
\begin{center}
\includegraphics[width = 12.5 cm]{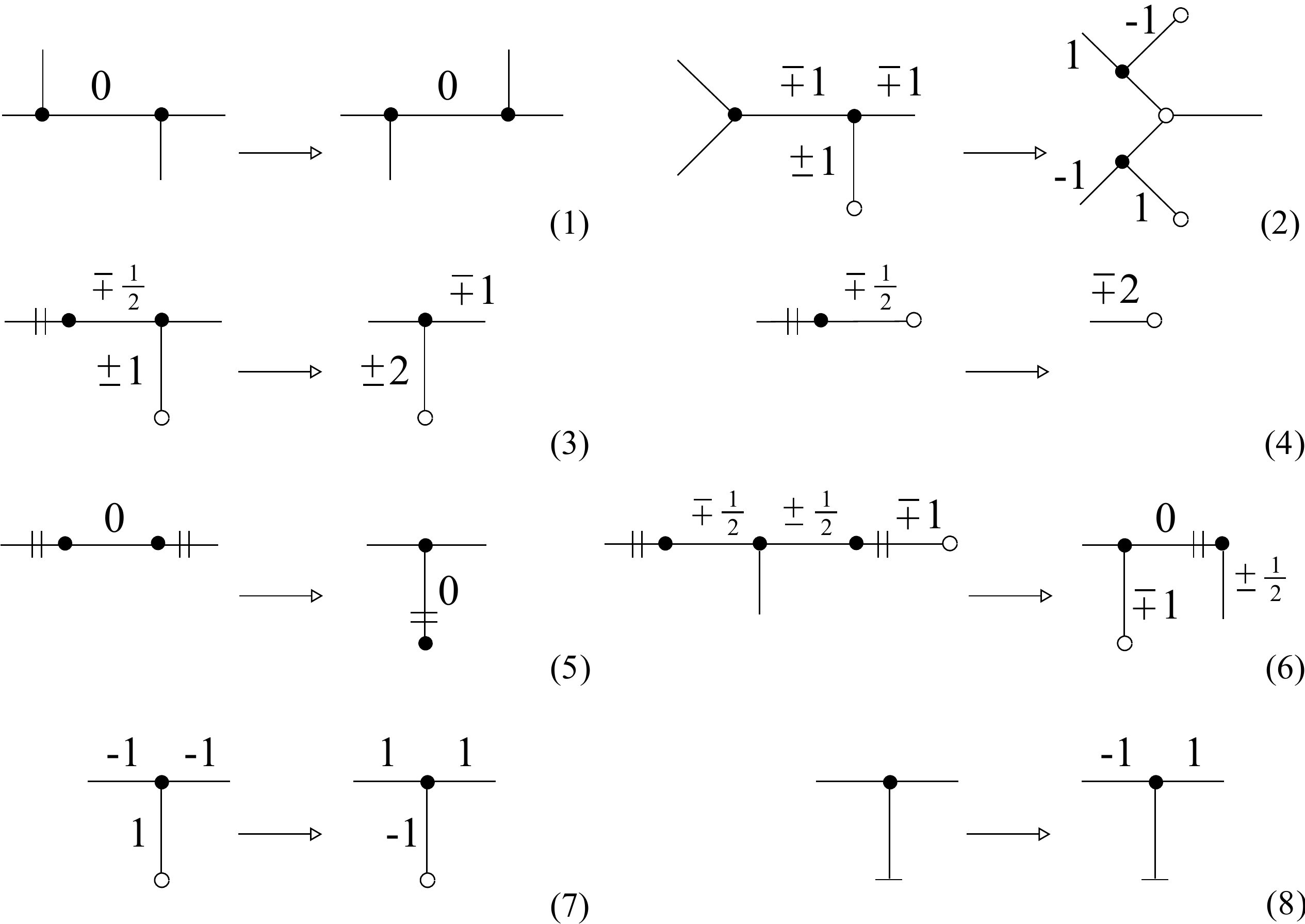}
\nota{These moves relate two shadows of the same block.}
\label{thickening:fig}
\end{center}
\end{figure}

\begin{prop}
The moves in Fig.~\ref{thickening:fig} relate two shadows $X_1, X_2$ of the same block $(M,L)$.
\end{prop}
\begin{proof}
The moves (1-6) are the ones described in Fig.~\ref{mossa_all:fig}. Move (7) corresponds to two different perturbations of a double point, see Fig.~\ref{perturb:fig}. Move (8) follows from Fig.~\ref{sum_ass:fig}-(3) below: both $X_1$ and $X_2$ are shadows of the same block, obtained from another block by drilling along the same curve.
\end{proof}

\begin{figure}
\begin{center}
\includegraphics[width = 12.5 cm]{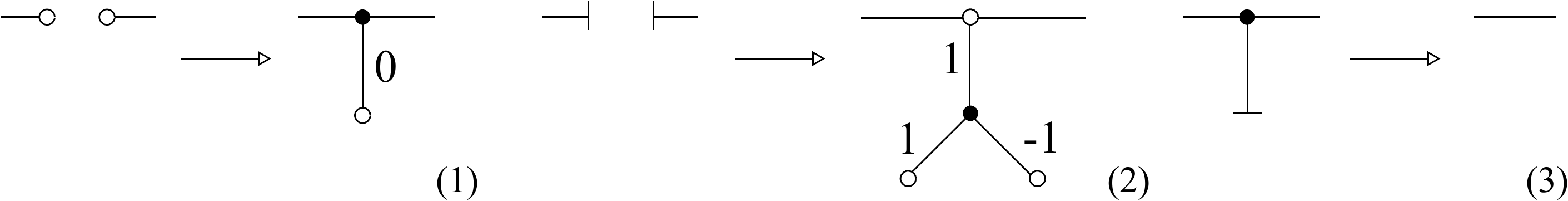}
\nota{These moves transform a shadow of $(M,L)$ into a shadow of some $(M',L')$. The new block $(M',L')$ is a connected sum (1), assembling (2), or filling (3) of the original one $(M,L)$.}
\label{sum_ass:fig}
\end{center}
\end{figure}

\begin{prop}
The moves in Fig.~\ref{sum_ass:fig} transform a shadow $X$ of a block $(M,L)$ into a shadow $X'$ of a block $(M',L')$, and viceversa. The block $(M',L')$ is respectively a connected sum, assembling, or filling of $(M,L)$.
\end{prop}
\begin{proof}
This corresponds to Propositions \ref{sum:prop}, \ref{assembling:prop}, and \ref{filling:prop}.
\end{proof}

\subsection{Decomposition into pieces} \label{decomposition:subsection}
Let $X$ be a shadow without vertices and $N(X)$ its thickening. As shown by Costantino and Thurston \cite{CoThu}, there is a natural map $\pi:\partial N(X) \to X$ which is a circle fibering over the non-singular points of $X$. (Such a map might actually extended to the whole of $N(X)$, but we only need the boundary here.) Let $G$ be a decorated graph describing $X$. Recall that such a graph determines a decomposition into pieces of $X$, and each vertex of $G$ determines a piece of $X$.

\begin{prop}[Costantino-Thurston \cite{CoThu}] \label{deco:prop}
The decorated graph $G$ describes a decomposition of the closed 3-manifold $\partial N(X)$ into pieces bounded by tori, as follows.
\begin{enumerate}
\item Every piece $Q$ of $X$ determines a ``horizontal'' piece $\overline{\pi^{-1}(\interior Q)}$: its homeomorphism type depends on $Q$ and is shown in Table \ref{pieces:table}. 
\item Every component $C$ of $\partial X$ determines a ``vertical'' solid torus $\pi^{-1}(C)$.
\end{enumerate}
\end{prop}

\begin{table} 
  \begin{center}
    \begin{tabular}{r|cccccc}
      \phantom{\Big|}
      Vertex
      & \includegraphics[width = 1 cm]{1.pdf} & \includegraphics[width = 1 cm]{2.pdf} & \includegraphics[width = 1 cm]{3.pdf} & \includegraphics[width = 1 cm]{4.pdf} & \includegraphics[width = 1 cm]{5.pdf} & \includegraphics[width = 1 cm]{6.pdf} \\
      \hline
      \phantom{\bigg|}
      $Q$ (name)
      & $D^2$ & $P^2$ & $Y_2$ & $Y_{111}$ & $Y_{12}$ & $Y_3$ \\
      \phantom{$\begin{array}{c} \Bigg| \\ \bigg| \end{array}$}\!\!\!\!\!\!\!\!\!
      $Q$ (picture)
      & \begin{minipage}{.10\textwidth} \ \ \includegraphics[width = 1 cm]{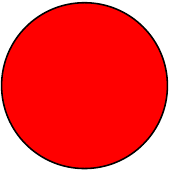} \end{minipage}
      & \begin{minipage}{.10\textwidth} \includegraphics[width = 1.5 cm]{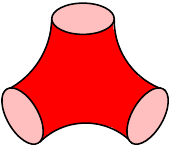} \end{minipage} 
      & \begin{minipage}{.10\textwidth} \includegraphics[width = 1.5 cm]{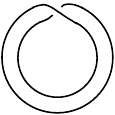} \end{minipage}
      & \begin{minipage}{.10\textwidth} \includegraphics[width = 1.5 cm]{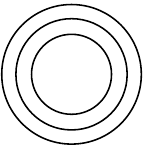} \end{minipage} 
      & \begin{minipage}{.10\textwidth} \includegraphics[width = 1.5 cm]{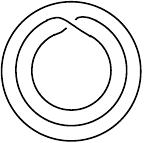} \end{minipage}
     & \begin{minipage}{.10\textwidth} \includegraphics[width = 1.5 cm]{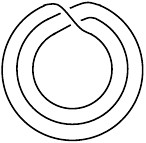} \end{minipage} \\
   \hline
      \phantom{\bigg|}
      $\pi^{-1}(Q)$ (name)
      & $D^2\times S^1$
      & $P^2\times S^1$
      & $Y_2\timtil S^1$
      & $P^2\times S^1$
      & $(A^2,2)$
      & $(D^2,3,3)$ \\
      \phantom{$\begin{array}{c} \Bigg| \\ \bigg| \end{array}$}\!\!\!\!\!\!\!\!\!
      $\pi^{-1}(Q)$ (picture) & & & \begin{minipage}{.10\textwidth} \includegraphics[width = 1.3 cm]{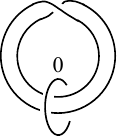}  \end{minipage}
      & \begin{minipage}{.10\textwidth} \includegraphics[width = 1.5 cm]{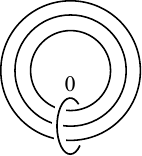} \end{minipage} 
      & \begin{minipage}{.10\textwidth} \includegraphics[width = 1.5 cm]{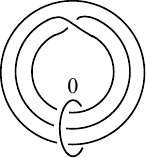}  \end{minipage}
      & \begin{minipage}{.10\textwidth} \includegraphics[width = 1.5 cm]{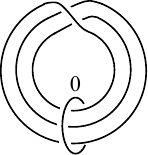} \end{minipage}
    \end{tabular}	
  \end{center} 
  \vspace{.2 cm}
  \nota{Every shadow $X$ without vertices decomposes into
  pieces $Q$. This induces a decomposition of the 3-manifold $\partial N(X)$ into some pieces bounded by tori. We denote by $Y_2\tilde\times S^1$ the orientable fibering over the M\"obius strip $Y_2$. The Seifert manifolds $(A^2,2)$ and $(D^2,3,3)$ fiber respectively over the annulus $A^2$ with one exceptional fiber of order 2, and over the disc $D^2$ with two exceptional fibers of order 3. 
  Such manifolds are pictured as link complements in $S^2\times S^1$.}
  \label{pieces:table}
\end{table}

\begin{proof}
The map $\pi:\partial N(X) \to X$ is a circle bundle on non-singular points. If $Q$ is a surface, the piece $\overline{\pi^{-1}(\interior Q)}$ is the orientable circle bundle over $Q$: this holds in cases
\includegraphics[width = .6 cm]{1.pdf}, \includegraphics[width = .6 cm]{2.pdf}, and \includegraphics[width = .6 cm]{3.pdf}. The pieces corresponding to \includegraphics[width = .6 cm]{4.pdf}, \includegraphics[width = .6 cm]{5.pdf}, and \includegraphics[width = .6 cm]{6.pdf}
are obtained by thickening the singular edge to a product $D^3\times S^1$.
We can think of $Q$ as properly embedded inside $D^3\times S^1$, so that $\overline{\pi^{-1}(\interior Q)}$ consists
of the boundary $S^2\times S^1$ minus an open regular neighborhood of $\partial Q$. The curves $\partial Q$ are the closed braids in $S^2\times S^1$ shown in the table.
\end{proof}

Note that the vertices \includegraphics[width = 0.6 cm]{0.pdf} and \includegraphics[width = .6 cm]{1.pdf} both give rise to solid tori. However, they are positioned differently with respect to the fibration $\pi$: their meridian is respectively vertical (\emph{i.e.}~a fiber of $\pi$) and horizontal (\emph{i.e.}~a section of $\pi$). Analogously, the vertices
\includegraphics[width = .6 cm]{2.pdf} and \includegraphics[width = .6 cm]{4.pdf} both yield a piece homeomorphic to $P^2 \times S^1$, but positioned differently: the fiber $\{pt\} \times S^1$ is respectively vertical and horizontal.

\section{Reduction to very simple polyhedra} \label{reduction:section}
This and the subsequent sections are strictly devoted to the proof of Theorem \ref{main:teo}. We start 
by eliminating some types of vertices. In this section we prove the following.

\begin{figure}
\begin{center}
\includegraphics[width = 11 cm]{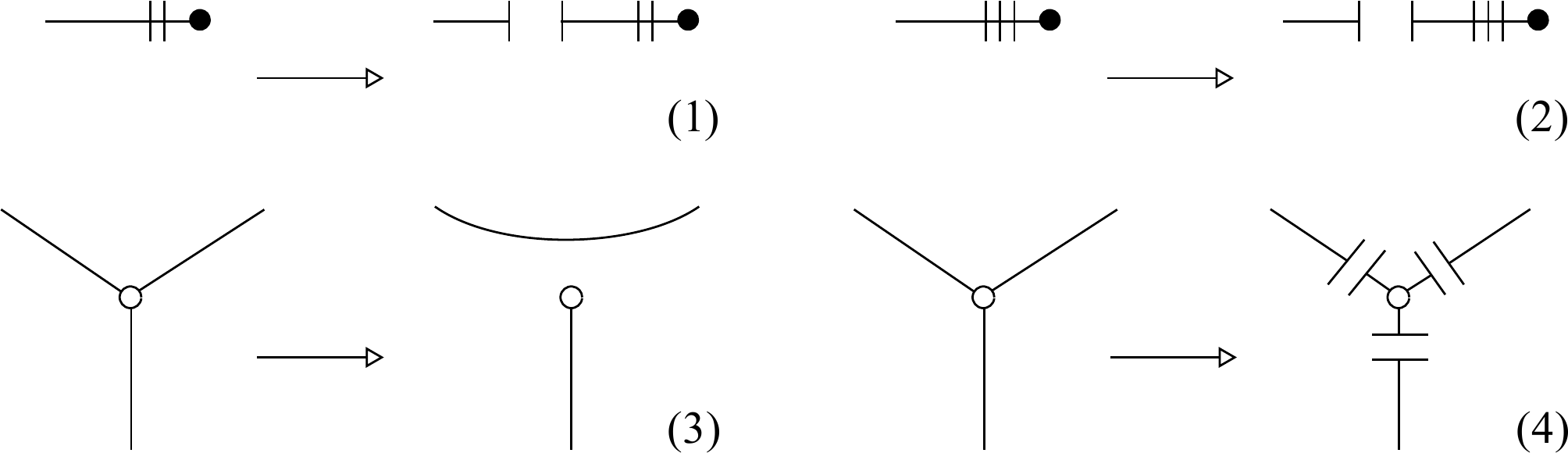}
\nota{The moves (1) and (2) transform a shadow of a block into a shadow of a de-assembled block. If the graph describing the shadow contains a trivalent white vertex (\emph{i.e.}~a pair-of-pants, see Fig.~\ref{pieces:table}-(3)) either the move (3) or (4) applies.}
\label{reduction:fig}
\end{center}
\end{figure}

\begin{teo} \label{very:simple:teo}
Let $X$ be a shadow of a block $(M,L)$, described via a decorated graph.
\begin{itemize}
\item Suppose the graph contains a vertex of type \includegraphics[width = 0.6 cm]{3.pdf}.
The move shown in Fig.~\ref{reduction:fig}-(1) transforms $X$ into a shadow $X'$ of a block $(M',L')$. \item Suppose the graph contains a vertex of type \includegraphics[width = 0.6 cm]{6.pdf}.
The move shown in Fig.~\ref{reduction:fig}-(2) transforms $X$ into a shadow $X'$ of a block $(M',L')$. \item Suppose the graph contains a vertex of type \includegraphics[width = 0.6 cm]{2.pdf}.
One of the two moves shown in Fig.~\ref{reduction:fig}-(3) and (4) transforms $X$ into a shadow $X'$ of a block $(M',L')$.
\end{itemize}
In all cases, the original block $(M,L)$ is obtained from $(M',L')$ by a combination of assemblings or connected sums.
\end{teo}

This result allows to restrict our investigation to a smaller class of shadows, whose underlying polyhedron is as follows. 

\begin{defn} A \emph{very simple polyhedron} is a simple polyhedron which may described via a decorated graph with vertices of types shown in Fig.~\ref{vertices2:fig}.
\end{defn}

\begin{figure}
\begin{center}
\includegraphics[width = 5 cm]{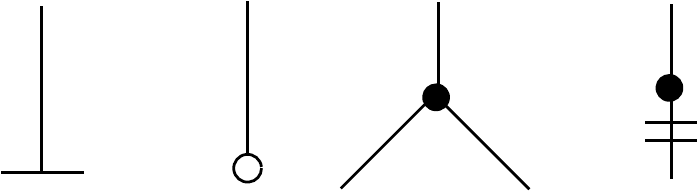}
\nota{A very simple shadow is encoded by a decorated graph with these types of vertices.}
\label{vertices2:fig}
\end{center}
\end{figure}

In other words, there are no pieces of type \includegraphics[width = 0.6 cm]{2.pdf}, \includegraphics[width = 0.6 cm]{3.pdf}, and \includegraphics[width = 0.6 cm]{6.pdf} from Fig.~\ref{vertices:fig}: these pieces can be ruled out thanks to Theorem \ref{very:simple:teo}, as the following corollary shows. (The notion of graph manifold generated by $\calS_0$ extends trivially to manifolds with non-empty boundary.)

\begin{cor} \label{very:simple:cor}
Every block $(M,L)$ having a shadow without vertices is obtained via connected sums and assemblings from $(M_1,L_1)\sqcup (M_2,L_2)$ where $M_1$ is a graph manifold generated by $\calS_0$ and $(M_2,L_2)$ has a very simple shadow.  (Both $M_1$ and $M_2$ may be disconnected.)
\end{cor}
\begin{proof}
Let $X$ be a shadow without vertices of $(M,L)$. It may be described as a decorated graph $G$. If $G$ is as in Fig.~\ref{blocks:fig}, then $M$ is a graph manifold. Otherwise,
suppose it contains a vertex of type \includegraphics[width = 0.6 cm]{2.pdf}, \includegraphics[width = 0.6 cm]{3.pdf}, or \includegraphics[width = 0.6 cm]{6.pdf}. Theorem \ref{very:simple:teo} applies: we can perform one of the moves in Fig.~\ref{reduction:fig} which simplifies the graph, and we conclude by induction. 
\end{proof}

The rest of the section is mainly devoted to the proof of Theorem \ref{very:simple:teo}. 

\subsection{Horizontal and vertical compressing discs}
Let $X$ be a shadow of some block $(M,L)$, encoded via a decorated graph $G$. Each edge of $G$ determines a simple closed curve $\gamma$ in a region of $X$ and a torus $T = \pi^{-1}(\gamma)\subset \partial N(X)$ fibering over $\gamma$ via the natural fibration $\pi:\partial N(X) \to X$, see Section \ref{decomposition:subsection}. Such a torus has a compressing disc $D$ in $\partial N(X\cup \partial M)$ because of the following general fact.

\begin{lemma} \label{compressing:lemma}
Every torus $T$ inside  $\#_k(S^2\times S^1)$ has a compressing disc. 
\end{lemma}
\begin{proof}
The fundamental group of $\#_k(S^2\times S^1)$ is a free group. A free group does not contain $\matZ\times\matZ$, so $T$ has a compressing disc by Dehn's Lemma. 
\end{proof}

Such a compressing disc may be positioned in various ways with respect to the fibration $\pi$. We will be interested only in two special cases.
\begin{defn} \label{horizontal:defn}
A compressing disc $D$ for $T$ is \emph{vertical} (resp. \emph{horizontal}) if $\partial D$ it is isotopic to a fibre (resp. a section) of the fibration $\pi:T\to\gamma$.
\end{defn}

If the compressing disc of $T$ is horizontal or vertical, we may somehow simplify the shadow, as the following shows.

\begin{figure}
\begin{center}
\includegraphics[width =11 cm]{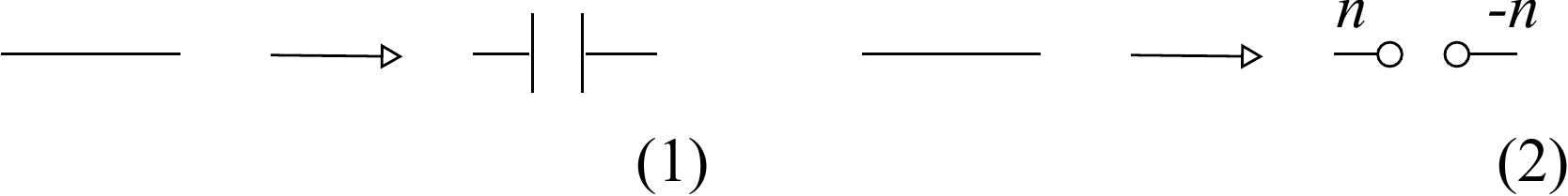}
\nota{If the torus $T$ above $\gamma$ has a vertical or horizontal compressing disc we can perform respectively the move (1) and (2). The integer $n$ depends on how many times the horizontal disc winds around the fiber.}
\label{hv:fig}
\end{center}
\end{figure}

\begin{prop} \label{disc:prop}
Suppose that $T$ has a vertical (resp.~horizontal) compressing disc. The move in Fig.~\ref{hv:fig}-(1) (resp.~(2)) transforms $X$ into a shadow $X'$ of a block $(M',L')$. The original $(M,L)$ is obtained from $(M',L')$ by assembling (resp.~connected sum).
\end{prop}
\begin{proof}
Let $H$ be the 1-handlebody $M\setminus \interior{N(X\cup\partial M)}$. We push the interior of the compressing disc $D$ slightly inside $H$, keeping $\partial D$ fixed. Now $H' = H \setminus \interior{N(D)}$ is homeomorphic to $H\cup($1-handle) and is hence still a 1-handlebody. We enlarge $D$ to a disc $D'\supset D$ with $\partial D' \subset X$, see Fig.~\ref{push:fig}. Set $Y = X \cup D'$. We have $M\setminus \interior{N(Y\cup\partial M)} \isom H'$. 

If $D$ is horizontal, the disc $D'$ is attached along $\alpha$ and $Y$ is thus simple. By construction, the disc $D'$ has gleam zero. Therefore a regular neighborhood of $D'$ looks like the right portion of Fig.~\ref{sum:fig}, with some gleams $n$ and $-n$ added to the two adjacent regions (for some integer $n$, which depends on how many times $\partial D$ winds around the fiber). We can therefore apply the inverse of the move in Fig.~\ref{sum:fig}, and the result is as in Fig.~\ref{hv:fig}-(2)-right. By Proposition \ref{sum:prop}, the result is a shadow of some $(M',L')$ of which $(M,L)$ is a connected sum.

If $D$ is vertical, the curve $\partial D$ projects to a point $x\in \gamma$ and the whole of $\partial D'$ is thus identified with $x$. That is, the disc $D'$ actually closes up to a 2-sphere $\Sigma$ which intersects $X$ transversely in $x$. We can thus apply the converse of the moves shown in Fig.~\ref{perturb2:fig} and Fig.~\ref{assembling:fig}. The result follows from Proposition \ref{assembling:prop}.
\end{proof}

\begin{figure}
\begin{center}
\includegraphics[width = 12 cm]{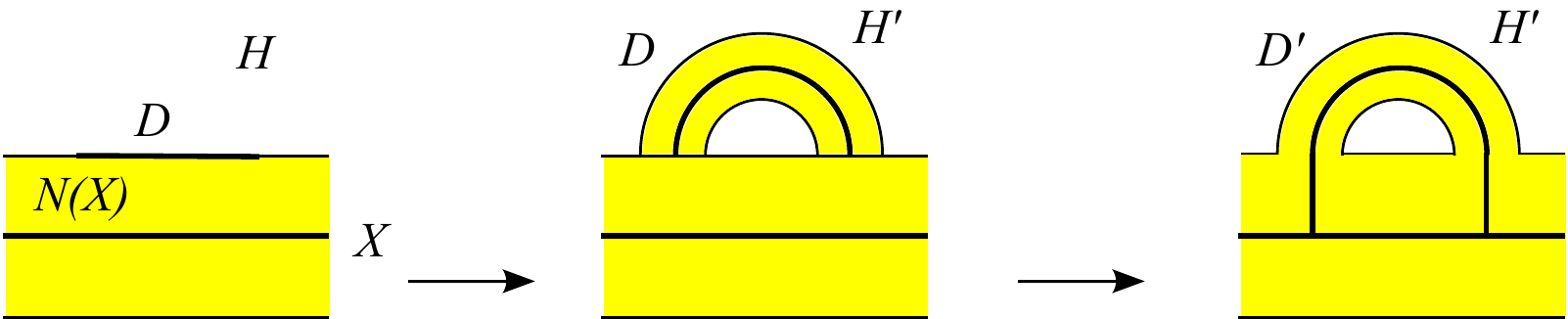}
\nota{We push the compressing disc $D$ inside $H$: the complement of an open regular neighborhood is again a 1-handlebody $H'$. (We have added a canceling pair of 2- and 3-handles.) Then we enlarge $D$ to $D'$, so that $\partial D'\subset X$.}
\label{push:fig}
\end{center}
\end{figure}

Note that in most cases the compressing disc in neither horizontal nor vertical, and no move is possible.
Proposition \ref{disc:prop} is a key tool we will use to prove inductively Theorem \ref{main:teo}. Given a decorated graph, we look for horizontal or vertical compressing discs. If found, the graph may be simplified along one of the moves in Fig.~\ref{hv:fig}, and we are done. Finding such a compressing disc is however hard: it is sometimes necessary to first modify the decorated graph with some of the moves listed in Fig.~\ref{thickening:fig}. The rest of the paper is mostly devoted to fulfill this task.

\subsection{Eliminate some types of vertices}\label{no:236:subsection}
Let $X$ be a shadow of a block $(M,L)$ described by a decorated graph $G$.
We prove here that a vertex of type \includegraphics[width = 0.6 cm]{2.pdf}, \includegraphics[width = 0.6 cm]{3.pdf}, or \includegraphics[width = 0.6 cm]{6.pdf} gives rise to a vertical or horizontal compressing disc. 

\begin{prop} \label{no:2:prop}
Consider a vertex of type \includegraphics[width = 0.6 cm]{2.pdf}. Let $T_1, T_2, T_3$ be the tori lying above the three indident edges. Either there is one $T_i$ which has a horizontal compressing disc, or every $T_i$ has a vertical compressing disc.
\end{prop}
\begin{proof}
The corresponding piece of $\partial N(X\cup\partial M)$ is homeomorphic to $P^2\times S^1$. Some standard arguments in 3-dimensional topology show that at least one boundary torus $T_i$ of $P^2\times S^1$ has a compressing disc $D$ whose boundary $\partial D$ is either isotopic to a fiber (\emph{i.e.}~vertical) or to a section (\emph{i.e.}~horizontal). In the first case, the compressing disc extends fiberwise also to the two other boundary tori.

This is the argument. Every boundary component of $P^2\times S^1$ has a compressing disc. Suppose each of them is neither horizontal nor vertical. If all discs are directed outside of $P^2\times S^1$, then $\partial (N(X\cup\partial M))\isom \#_h (S^2\times S^1)$ has a summand which is a Seifert manifold with 3 singular fibers: a contradiction \cite{Sei}. If one disc is directed inside, after an isotopy it intersects $P^2\times S^1$ into an essential planar surface. However, such a surface in $P^2\times S^1$ must intersect one boundary component either horizontally or vertically \cite{Sei}, against our assumptions.
\end{proof}

\begin{prop} \label{no:36:prop}
Consider a vertex of type 
\includegraphics[width = 0.6 cm]{3.pdf} or \includegraphics[width = 0.6 cm]{6.pdf}.
The torus $T$ lying above the incident edge has a vertical compressing disc.
\end{prop}
\begin{proof}
The corresponding piece in $\partial N(X\cup\partial M)$ is $Y_2\timtil S^1\isom (D^2,2,2)$ or $(D^2,3,3)$, see Table \ref{pieces:table}. Its boundary has a compressing disc $D$, directed outward. By Dehn filling the piece along the slope $\partial D$ we thus get some summands of $\partial N(X\cup \partial M) \isom \#_k (S^2\times S^1)$. 

Standard arguments on Seifert manifolds show that the $p/q$-Dehn filling on the knot shown in Table \ref{pieces:table} is $\#_h (S^2\times S^1)$ if and only if $p/q=\infty $, \emph{i.e.}~when the meridinal disc is vertical (and $h=1$ in this case). Therefore $D$ must be vertical.
\end{proof}

The two propositions just stated imply Theorem \ref{very:simple:teo}.

\dimo{very:simple:teo}
If the decorated graph contains a vertex of type \includegraphics[width = 0.6 cm]{3.pdf} or \includegraphics[width = 0.6 cm]{6.pdf}, the torus lying above the incident edge has a compressing disc and hence we can apply Fig.~\ref{hv:fig}-(1). The result is a move as in Fig.~\ref{reduction:fig}-(1,2).

If it contains a vertex of type \includegraphics[width = 0.6 cm]{2.pdf}, there are three tori above the edges. Either one has a horizontal compressing disc, or all three have vertical compressing discs. The corresponding move in Fig.~\ref{hv:fig} applies and the result is one of the moves in Fig.~\ref{reduction:fig}-(3,4). (Apply Fig.~\ref{mosse_innocue:fig}-left.)
\finedimo

\subsection{Try to eliminate other types of vertices}
Unfortunately, there is no result analogous to Propositions \ref{no:2:prop} and \ref{no:36:prop} for vertices type \includegraphics[width = 0.6 cm]{1.pdf}, \includegraphics[width = 0.6 cm]{4.pdf}, or \includegraphics[width = 0.6 cm]{5.pdf}. A partial result for the 3-valent vertex is the following.

\begin{figure}
\begin{center}
\includegraphics[width = 5 cm]{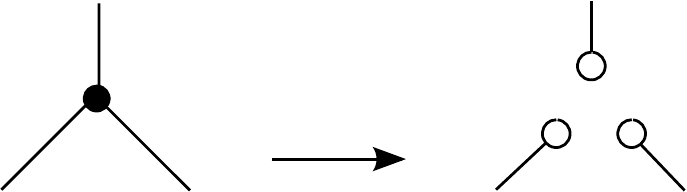}
\nota{This move applies only when the fiber of the $P^2\times S^1$ lying above the vertex bounds a compressing disc.}
\label{scoppia:fig}
\end{center}
\end{figure}

\begin{prop} \label{(4):prop}
Consider a vertex of type 
\includegraphics[width = 0.6 cm]{4.pdf}. It determines a piece in $\partial N(X\cup\partial M)$ homeomorphic to $P^2\times S^1$. Suppose that the fiber $\{pt\}\times S^1$ bounds a disc in $\partial N(X\cup\partial M)$. The move in Fig.~\ref{scoppia:fig} transforms $X$ into a shadow $X'$ of some $(M',L')$ of which $(M,L)$ is a twice connected sum.
\end{prop}
\begin{proof}
The compressing disc is actually a horizontal disc in this case! We can therefore perform the move in Fig.~\ref{hv:fig}-(2) and the inverse of Fig.~\ref{sum_ass:fig}-(1). The sequence of moves is shown in Fig.~\ref{scoppia_proof:fig}.
\begin{figure}
\begin{center}
\includegraphics[width = 9 cm]{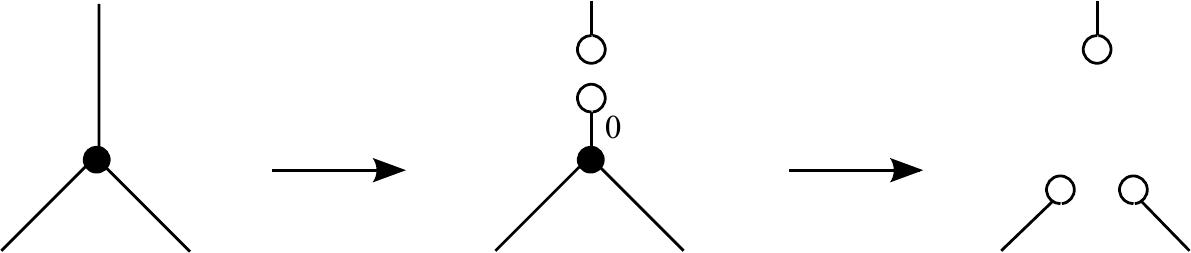}
\nota{Two steps in Fig.~\ref{scoppia:fig}.}
\label{scoppia_proof:fig}
\end{center}
\end{figure}
\end{proof}

A much weaker result concerning the 2-valent vertex is the following.

\begin{figure}
\begin{center}
\includegraphics[width = 6 cm]{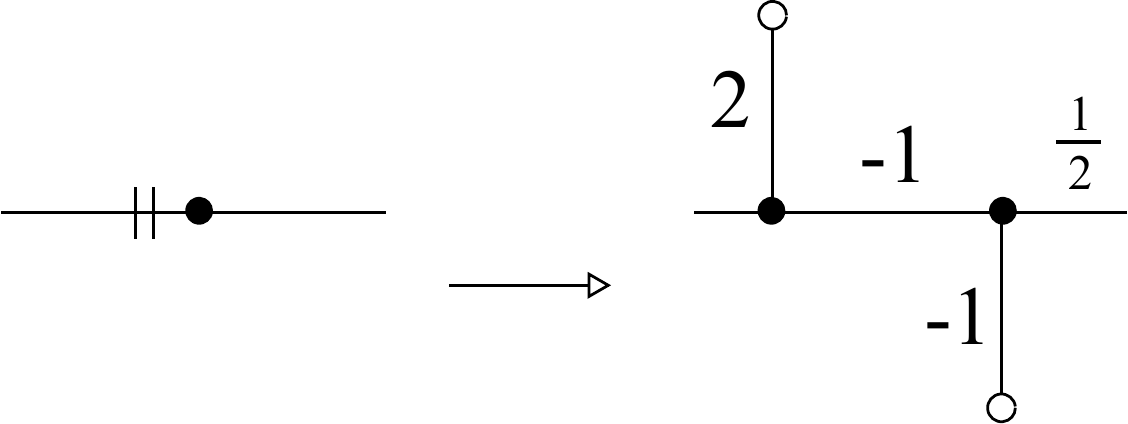}
\nota{This move relates shadows of different blocks.}
\label{boundary:fig}
\end{center}
\end{figure}

\begin{prop} \label{boundary:prop}
Consider a vertex of type \includegraphics[width = 0.6 cm]{5.pdf}. The move in Fig.~\ref{boundary:fig} transform $X$ into the shadow $X'$ of some other block $(M',L')$. 
\end{prop}
\begin{proof}
See Fig.~\ref{boundary_proof:fig}.
\begin{figure}
\begin{center}
\includegraphics[width = 10 cm]{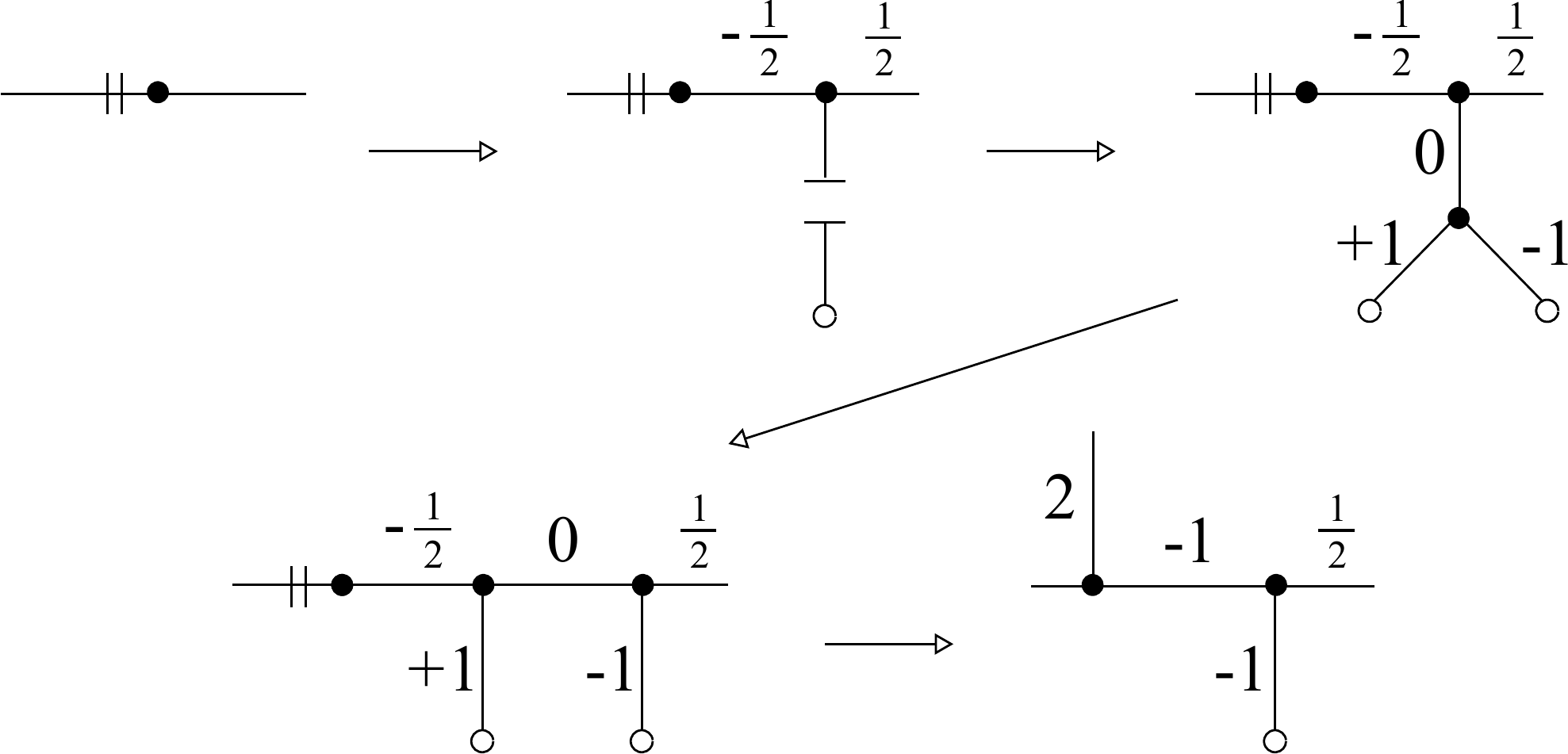}
\nota{Intermediate steps for Fig.~\ref{boundary:fig}. We drill along a curve as in Fig.~\ref{sum_ass:fig} and then assemble with the shadow of $D^2\times S^2$ (taken from Fig.~\ref{blocks:fig}). The manifold $M$ thus changes via surgery. Then we use and Fig.~\ref{thickening:fig}-(1, 3).}
\label{boundary_proof:fig}
\end{center}
\end{figure}
\end{proof}

The move shown in Fig.~\ref{boundary:fig} changes dramatically the block and thus cannot be used to simplify shadows. (The proof  shows that $M'$ is obtained from $M$ by surgery, \emph{i.e.}~by substituting a $S^1\times D^3$ with a $S^2\times D^2$.)

\section{Trees with level functions} \label{trees:section}
We make here another step towards the proof of Theorem \ref{main:teo}. According to Corollary \ref{very:simple:cor}, we may restrict to blocks having very simple shadows. A very simple shadow is described via a decorated graphs with vertices as in Fig.~\ref{vertices2:fig}.

In this section, we show that we may further restrict to decorated graphs that are trees equipped with a \emph{level function}. The level function is a function on vertices which is defined below. A decorated tree equipped with such a function is a \emph{decorated tree with levels}. We prove here the following.

\begin{teo} \label{leveled:teo}
Let $X$ be a very simple shadow of a block $(M,L)$. One of the following holds.
\begin{enumerate}
\item A move as in Fig.~\ref{scoppia:fig} transforms $X$ into a shadow $X'$ of a block $(M',L')$ such that $(M,L)$ is a twice connected sum of $(M',L')$.
\item The shadow $X$ can be encoded via a decorated tree with levels.
\end{enumerate}
\end{teo}

We thus get a refinement of Corollary \ref{very:simple:cor}.
\begin{cor} \label{leveled:cor}
Every block $(M,L)$ having a shadow without vertices is obtained via connected sums and assemblings from $(M_1,L_1)\sqcup (M_2,L_2)$ where $M_1$ is a graph manifold generated by $\calS_0$ and $(M_2,L_2)$ has a shadow encoded by a decorated tree with levels. (Both $M_1$ and $M_2$ may be disconnected.)
\end{cor}
\begin{proof}
By Corollary \ref{very:simple:cor}, we may restrict to very simple shadows. Let $X$ be a very simple shadow. If it may be encoded as a tree with a level function we are done. Otherwise, the move in Fig.~\ref{scoppia:fig} applies: the number of vertices of type \includegraphics[width = 0.6 cm]{4.pdf} decreases and we proceed by induction.
\end{proof}
The rest of this section is devoted to the definition of a level function and to the proof of Theorem \ref{leveled:teo}.

\subsection{The level function} \label{leveled:subsection}
Two vertices in a graph are \emph{adjacent} if they are joined by an edge. A sequence of distinct vertices $v_1,\ldots, v_k$ form a \emph{line} if $v_i$ and $v_{i+1}$ are adjacent for all $i$. A \emph{decorated tree} is a decorated graph $T$ without cycles.

Let $T$ be a decorated tree which encodes a shadow $X$ of a block $(M,L)$. A \emph{level function} on $T$ is a function which associates to each vertex $v$ a non-negative integer $l(v)$ such that the following holds.

\begin{enumerate}
\item there are $k\geqslant 2$ vertices having level zero, and they form a line $v_1,\ldots, v_k$ called \emph{root}; 
\item every vertex $v$ of type \includegraphics[width = 0.6 cm]{4.pdf} or \includegraphics[width = 0.6 cm]{5.pdf}  is adjacent to precisely one vertex $v'$ with $l(v')>l(v)$;
\item on every line $w_1,\ldots, w_h$ we have $w_i \leqslant \max \{w_1,w_h\}$ for all $i$.
\end{enumerate}
The third condition says that whenever the level starts increasing on a line, it keeps being non-decreasing forever. There is also a fourth condition which relates the function $l$ with the induced decomposition of the closed 3-manifold $\partial N(X)$. To state it we first need to introduce first some terminology and prove some easy facts. 

Consider a decorated tree $T$ and a function $l$ fulfilling the three requirements just stated. Let $v$ be a vertex. Define $S_v$ as the set of all vertices $v'$ such that there is a line
$$v=v_1, \ldots, v_k = v'$$
with $l(v_2)>l(v_1)$. 
\begin{prop}
We have $l(v')>l(v)$ for every $v'\in S_v$. The set $S_v$ is non-empty precisely when $v$ is of type \includegraphics[width = 0.6 cm]{4.pdf}  or \includegraphics[width = 0.6 cm]{5.pdf} . When non-empty, it contains precisely one vertex adjacent to $v$, and spans a subtree of $T$; the vertex $v$ is the only one in $T\setminus S_v$ which is adjacent to some vertex in $S_v$.
\end{prop}
\begin{proof}
It follows easily from the assumptions (1), (2), and (3) above.
\end{proof}

Recall from Proposition \ref{deco:prop} that $T$ also encodes a decomposition of the closed 3-manifold $\partial N(X)$. Every vertex $v$ corresponds to a 3-dimensional piece $M_v\subset \partial N(X)$ bounded by tori according to Table \ref{pieces:table}. If $S$ is a set of vertices of $T$, we set $M_S = \cup_{v\in S} M_v$.

\begin{prop} Let $v$ be a vertex of type \includegraphics[width = 0.6 cm]{4.pdf} or \includegraphics[width = 0.6 cm]{5.pdf} . The manifold $M_{S_v}$ is connected and has only one boundary torus, attached to one boundary torus of $M_v$.
\end{prop}
\begin{proof}
The set $S_v$ spans a subtree; thus the corresponding pieces
in $\partial N(X)$ glue to form a connected manifold $M_{S_v}$. Since $v$ is the only vertex adjacent to some vertices of $S_v$, this manifold is bounded by a single torus attached to $M_v$.
\end{proof}
When $v$ is of type \includegraphics[width = 0.6 cm]{4.pdf}  or \includegraphics[width = 0.6 cm]{5.pdf} , the piece $M_v$ is a Seifert manifold, homeomorphic to either $P^2\times S^1$ or $(A,2)$. (In both cases, the Seifert fibration is unique up to isotopy and induces a fibration on the boundary tori \cite{Sei}.) Finally, we can state the fourth and last requirement for our level function $l$.

\begin{enumerate}
\addtocounter{enumi}{3}
\item for every vertex $v$ of type \includegraphics[width = 0.6 cm]{4.pdf} or \includegraphics[width = 0.6 cm]{5.pdf} , the manifold $M_{S_v}$ is a solid torus, whose meridian is attached to a section of the fibration of $M_v$.
\end{enumerate}

\begin{defn} A \emph{level function} on $T$ is a function which fulfills  all the requirements (1)-(4) listed above.
\end{defn}

A \emph{decorated tree with levels} is a decorated tree $T$ which encodes a shadow $X$ of some block $(M,L)$, equipped with a level function.

\subsection{Build a level function}
Let $T$ be a decorated tree with levels, encoding a shadow $X$ of a block $(M,L)$. As the following result shows, the level function puts some serious restrictions on the decomposition of $\partial N(X)$. 

\begin{prop} \label{4:consequence:prop}
For every vertex $v$ of type \includegraphics[width = 0.6 cm]{4.pdf}  or \includegraphics[width = 0.6 cm]{5.pdf} , the manifold $M_v\cup M_{S_v}$ is either homeomorphic to $A\times S^1$ or $D^2\times S^1$. The manifold $\partial N(X)$ is either $S^3$ or $S^2\times S^1$.
\end{prop}
\begin{proof}
The manifold $M_v$ is homeomorphic to either $P^2\times S^1$ or $(A,2)$. The manifold $M_{S_v}$ is a solid torus attached to $M_v$, whose meridian is a section of the fibration of $M_v$. The fibration on $M_v$ thus extends on $M_v\cup M_{S_v}$ without creating new exceptional fibers. Therefore $M_v\cup M_{S_v}$ is either homeomorphic to $A\times S^1$ or to $(D,2)\isom D^2\times S^1$.

We may simplify inductively the decomposition of $\partial N(X)$ as follows: if $v$ is of type \includegraphics[width = 0.6 cm]{4.pdf}, simply delete $M_v\cup M_{S_v}$; if it is of type \includegraphics[width = 0.6 cm]{5.pdf}, substitute it with a single solid torus. After finitely many steps we end up with a decomposition containing only solid tori, and thus $\partial N(X)$ has Heegaard genus at most 1.
Since $X$ is a shadow of some block, we must have $\partial N(X) = \#_h(S^2\times S^1)$. Therefore $h$ equals 0 or 1, as required.
\end{proof}
We can finally prove Theorem \ref{leveled:teo}.

\dimo{leveled:teo}
The very simple shadow $X$ is encoded via a decorated graph $G$, whose vertices are of type \includegraphics[width = 0.6 cm]{0.pdf}, \includegraphics[width = 0.6 cm]{1.pdf}, \includegraphics[width = 0.6 cm]{4.pdf} , or \includegraphics[width = 0.6 cm]{5.pdf}. This also encodes correspondingly a decomposition of $\partial N(X)\isom \#_h(S^2\times S^1)$ into pieces homeomorphic to solid tori, solid tori, $P^2\times S^1$, and $(A,2)$. 

The theorem follows from a slightly more general result about decompositions of $\#_h(S^2\times S^1)$ into pieces homeomorphic to solid tori, $(A,2)$, and $P^2\times S^1$. Any such decomposition yields a graph with vertices of valence 1, 2, or 3, and the notion of level function applies \emph{as is} to this more general context. 

\emph{Claim:
Let a decomposition of $\#_h(S^2\times S^1)$ be given. It induces a graph $G$.
One of the following holds.
\begin{enumerate}
\item There is a piece $P^2\times S^1$ whose fiber $\{pt\}\times S^1$ bounds a compressing disc in its complement.
\item The graph $G$ is actually a tree which may be equipped with a level function.
\end{enumerate}
}
We prove the claim by induction on the number of pieces in the decomposition.
If the decomposition consists of two solid tori then (2) holds and we are done. Otherwise, every solid torus $D^2\times S^1$ is adjacent to a $P^2\times S^1$ or $(A,2)$. If the meridian of one solid torus is attached to $P^2\times S^1$ along the fiber, then (1) holds and we are done. It cannot be attached to a fiber of $(A,2)$, since this would yield a projective plane, but there is no such surface in $\#_h(S^2\times S^1)$.

Therefore we can suppose the solid tori are not attached along fibers. Suppose one solid torus is attached along a section of the fibration of the adjacent $P^2\times S^1$ or $(A,2)$. The two pieces glued together are then homeomorphic to either $A\times S^1$ or $(D,2)\isom D^2\times S^1$. We can thus construct a simpler decomposition by removing these pieces and adding a $D^2\times S^1$ if the second case holds. By our induction hypothesis either (1) or (2) holds. If (1) holds in the new decomposition, it also holds in the old one, and we are done. If (2) holds, the new decomposition has a level function on its graph $G'$. The level function easily lifts from $G'$ to $G$, as follows. The graph $G$ is constructed from $G'$ with one of the moves shown in Fig.~\ref{claim:fig}. With move (1), assign $l(v_3) = \max \{l(v_1), l(v_2)\}$ and $l(v_4) = l(v_3)+1$. With move (2), assign $l(v_2) = l(v_1)$ and $l(v_3) = l(v_2)+1$.

\begin{figure}
\begin{center}
\includegraphics[width = 11 cm]{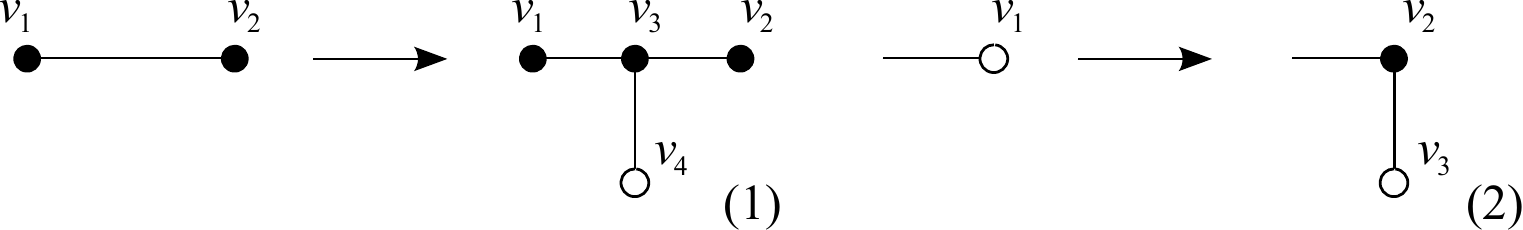}
\nota{The old graph $G$ is obtained by the new graph $G'$ by one of these moves. White vertices represent solid tori. In (1), the vertices $v_1$ and $v_2$ may be of any kind.}
\label{claim:fig}
\end{center}
\end{figure}

We are left with the case every solid torus is attached along a curve which is neither a fiber nor a section of the adjacent $P^2\times S^1$ or $(A,2)$. This produces a new singular fiber. We thus get a decomposition into blocks that are either $P^2\times S^1$, an annulus with one singular fiber, a
disc with two singular fibers, or $S^2$ with $3$ singular fibers. 
By assembling blocks with matching
fibers, we get either a Seifert manifold fibering over an orbifold with $\chi\leqslant 0$, and hence not homeomorphic to $S^2\times S^1$
and $S^3$, or a prime manifold with nontrivial JSJ: a contradiction in all cases. The claim is proved.

Finally, we show how the claim implies Theorem \ref{leveled:teo}. Our shadow $X$ may be represented as a decorated graph $G$, which also encodes a decomposition of $\partial N(X\cup\partial M)\isom \#_h(S^2\times S^1)$. The claim applies to $G$. If (1) holds, there is a piece $P^2\times S^1$ whose fiber bounds a compressing disc. It corresponds to a vertex of type \includegraphics[width = 0.6 cm]{4.pdf} and Proposition \ref{(4):prop} applies. The move in Fig.~\ref{scoppia:fig} thus transforms $X$ into a shadow $X'$ of some $(M',L')$ of which $(M,L)$ is a twice connected sum. If (2) holds, the graph $G$ is a tree which may be equipped with a level function, and we are done again.

\finedimo

\section{Leaves, fruits, and branches} \label{leaves:section}
We investigate here the decorated trees with levels defined in the previous section. We introduce some terminology -- \emph{leaves}, \emph{fruits}, and \emph{branches} -- and we study some moves that transform a tree into another.

\subsection{Drawing and cutting}
Let $T$ be a decorated tree with levels. We always draw $T$ with this convention: higher vertices in the picture have lower levels. As an example, see Fig.~\ref{leveled:fig}.

\begin{figure}
\begin{center}
\includegraphics[width = 6 cm]{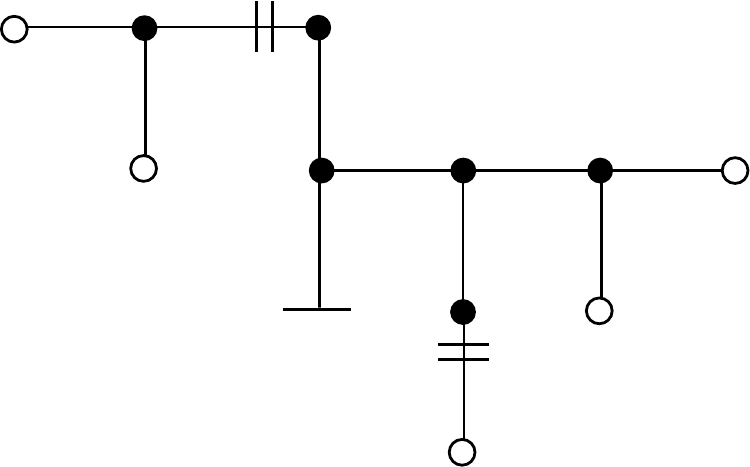}
\nota{A tree with levels. The level function may be deduced from the picture. There are 3 vertices at level 0 (the root), 5 vertices at level 1, 3 at level 2, and one at level 3.}
\label{leveled:fig}
\end{center}
\end{figure}

Consider a vertex $v$ on $T$. Recall that $S_v$ generates a subtree, which is non-empty precisely when $v$ is of type \includegraphics[width = 0.6 cm]{4.pdf} or \includegraphics[width = 0.6 cm]{5.pdf}. Since the vertex \includegraphics[width = 0.6 cm]{5.pdf} may be oriented in two different ways, there are three possibilities, which we may picture as in Fig.~\ref{Sv:fig}. 

We start by investigating the first one. Fig.~\ref{stacca:fig} shows a way of \emph{cutting} the subtree spanned by $S_v$. The result is a new decorated tree $T'$ with levels. The new level function $l'$ should be clear from the figure. (More precisely: let $v_2\in S_v$ be the vertex adjacent to $v$. We set $l'(w) = 0$ and $l'(v_*) = l(v_*) - l(v_2)$ for each vertex $v_*\in S_v$. In particular, the vertices $w$ and $v_2$ belong to the root of $T'$.)

\begin{figure}
\begin{center}
\includegraphics[width = 7 cm]{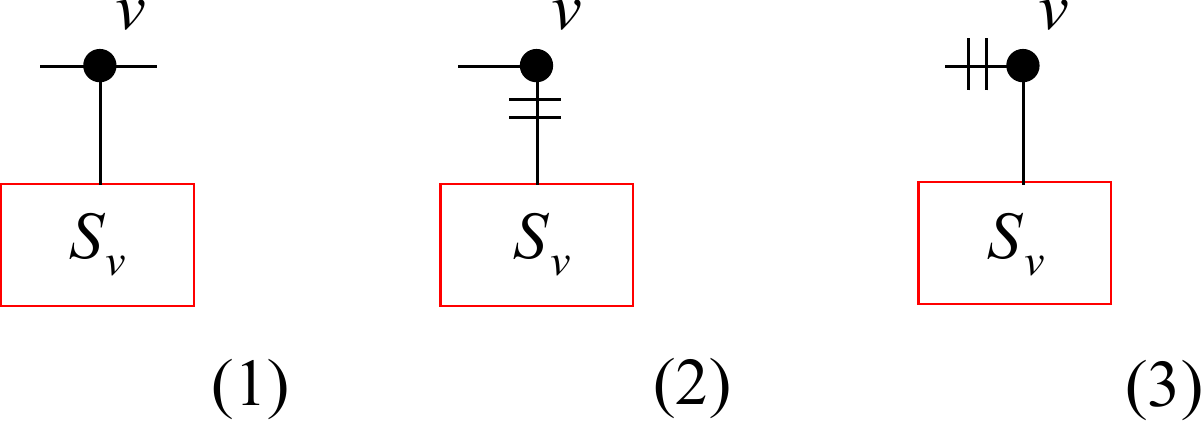}
\nota{When $S_v$ is non-empty, \emph{i.e.}~when the vertex $v$ has valence 2 or 3, there are three possible configurations.}
\label{Sv:fig}
\end{center}
\end{figure}

\begin{figure}
\begin{center}
\includegraphics[width = 5 cm]{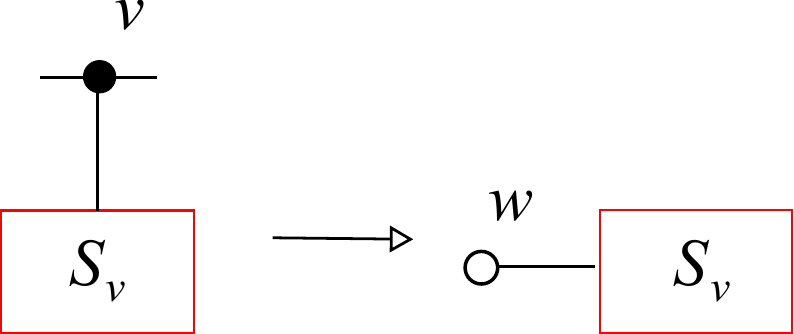}
\nota{This cut produces a new tree $T'$ with levels. It encodes a shadow $X'$ with $\partial N(X')= S^3$.}
\label{stacca:fig}
\end{center}
\end{figure}

\begin{prop} \label{cut:prop}
Let $T$ be a decorated tree with levels and $v$ a vertex of type \includegraphics[width = 0.6 cm]{4.pdf}. The cut in Fig.~\ref{stacca:fig} produces a new decorated tree $T'$ with levels. The new tree $T'$ encodes a shadow $X'$ with $\partial N(X') = S^3$.
\end{prop}
\begin{proof}
The axioms (1)-(4) descend easily from $T$ to $T'$, hence $T'$ is indeed a decorated tree with levels. The only non-trivial fact to prove is that $\partial N(X')=S^3$.

The manifold $M_{S_v}$ is a solid torus, whose meridian is attached to a section of $M_v$. We have $\partial N(X') = M_w \cup M_{S_v}$. The meridian of the solid torus $M_w$ is in fact isotopic to the fiber of $M_v$. Therefore the meridians of the two solid tori $M_w$ and $M_{S_v}$ have intersection 1, and thus $\partial N(X')=S^3$.  
\end{proof}

\subsection{Leaves}
Let $v$ be a vertex of type \includegraphics[width = 0.6 cm]{4.pdf}. If $S_v$ consists of a single vertex, this vertex is a \emph{leaf}. A leaf is a vertex of valence 1 and is either flat or fat, see Fig.~\ref{leaf:fig}. The edge connecting the base $v$ with its leaf is decorated with some integer $n$. When the vertex is flat the integer is not very important since it only determines the framing on the corresponding component of $\partial X$. On a fat vertex, we must have $n=\pm 1$, as the following shows.

\begin{figure}
\begin{center}
\includegraphics[width = 3 cm]{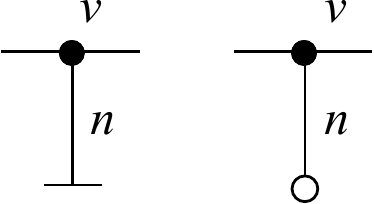}
\nota{A flat and fat leaf based at some vertex $v$. On a fat leaf, we must have $n= \pm 1$.}
\label{leaf:fig}
\end{center}
\end{figure}

\begin{prop} \label{leaf:prop}
Let $T$ be a decorated tree with levels. The edge joining a fat leaf and its base is decorated by $\pm 1$.
\end{prop}
\begin{proof}
If we cut the leaf as in Fig.~\ref{stacca:fig} we find a shadow $X' = S^2$ with gleam $n$. Therefore $\partial N(X')$ is the lens space $L(n,1)$. We must have $\partial N(X') = S^3$ by Proposition \ref{cut:prop}: therefore $n= \pm 1$.
\end{proof}

The sign of $\pm 1$ can in fact be changed easily.
\begin{prop}
Let $T$ be a decorated tree with levels, encoding a shadow $X$ of some block $(M,L)$. The moves in Fig.~\ref{move_leaf:fig} transform $T$ into a decorated tree with levels $T'$ encoding a shadow $X'$ of the same block $(M,L)$.
\end{prop}
\begin{proof}
Move (1) is Fig.~\ref{thickening:fig}-(7). To get (2), first use Fig.~\ref{thickening:fig}-(8) to move the gleam $n$ to the right, and then use Fig.~\ref{thickening:fig}-(1). 
\end{proof}

\begin{figure}
\begin{center}
\includegraphics[width = 10 cm]{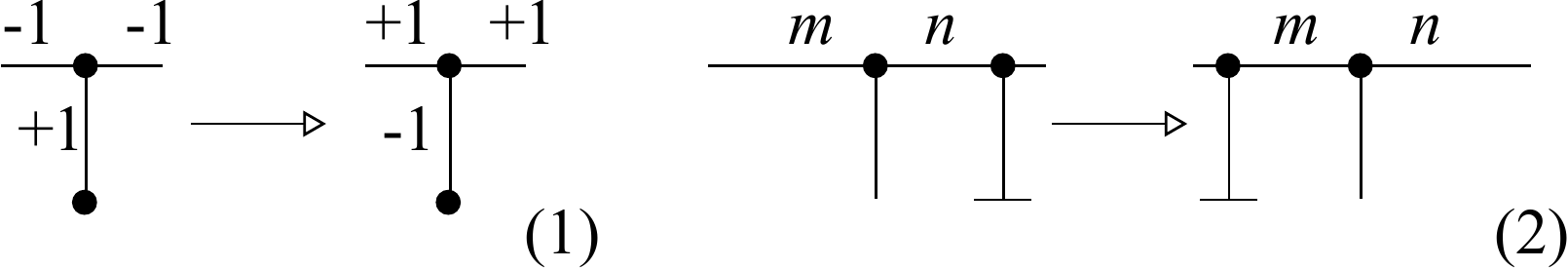}
\nota{These moves relate two decorated trees with levels determining the same block $(M,L)$.}
\label{move_leaf:fig}
\end{center}
\end{figure}

\subsection{Vertices of valence 2.}
Vertices of type \includegraphics[width = 0.6 cm]{5.pdf} are more difficult to treat than 3-valent vertices.
We may eliminate them with a move which changes however dramatically the topology of the block.  

\begin{prop}
Each of the moves in Fig.~\ref{elimina5:fig} transforms a decorated tree $T$ with levels into another decorated tree $T'$ with levels.
\end{prop}
\begin{proof}
The move is taken from Fig.~\ref{boundary:fig}. In each move of Fig.~\ref{elimina5:fig} the levels of the new vertices in $T'$ can be deduced from the picture. They are arranged so that $T'$ is indeed equipped with a level function. (Note that every leaf is decorated with a gleam $\pm 1$, in accordance with Proposition \ref{leaf:prop}.)
\end{proof}

\begin{figure}
\begin{center}
\includegraphics[width = 12 cm]{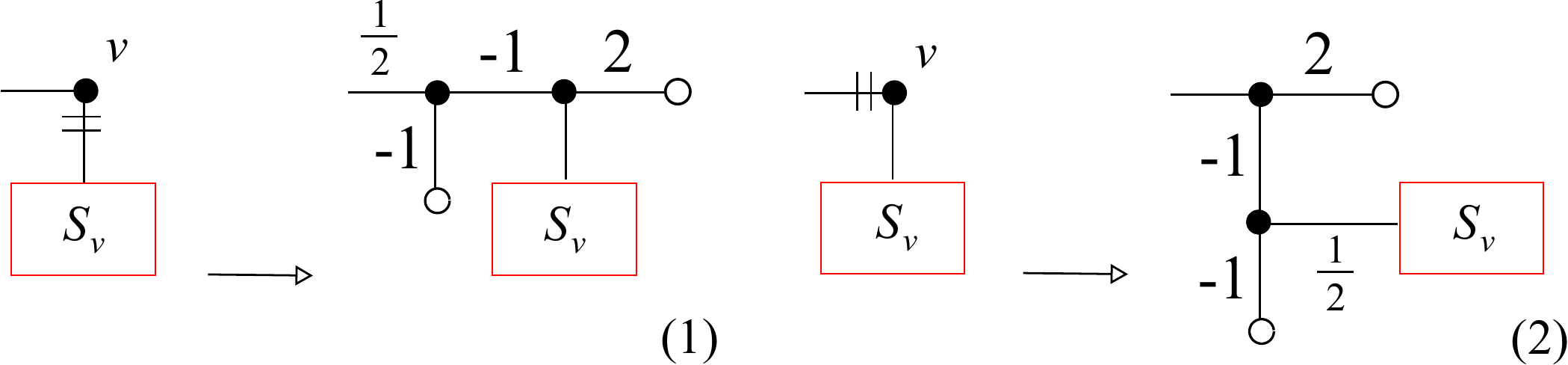}
\nota{These moves transform a decorated tree with levels into another.}
\label{elimina5:fig}
\end{center}
\end{figure}

Note that $T$ and $T'$ determine non-homeomorphic blocks $(M,L)$ and $(M',L')$ in general.
We can now state a version of Proposition \ref{cut:prop} for 2-valent vertices.
\begin{prop} Let $T$ be a decorated tree with levels and $v$ a vertex of type \includegraphics[width = 0.6 cm]{5.pdf}. Each cut in Fig.~\ref{stacca_generalized:fig} produces a new decorated tree with levels $T'$. The new tree $T'$ encodes a shadow $X'$ such that $\partial N(X') = S^3$.
\end{prop}
\begin{proof}
First apply the corresponding move in Fig.~\ref{elimina5:fig} and then Proposition \ref{cut:prop}.
\end{proof}

An example is shown in Fig.~\ref{leveled_example:fig}. 
\begin{figure}
\begin{center}
\includegraphics[width = 12.5 cm]{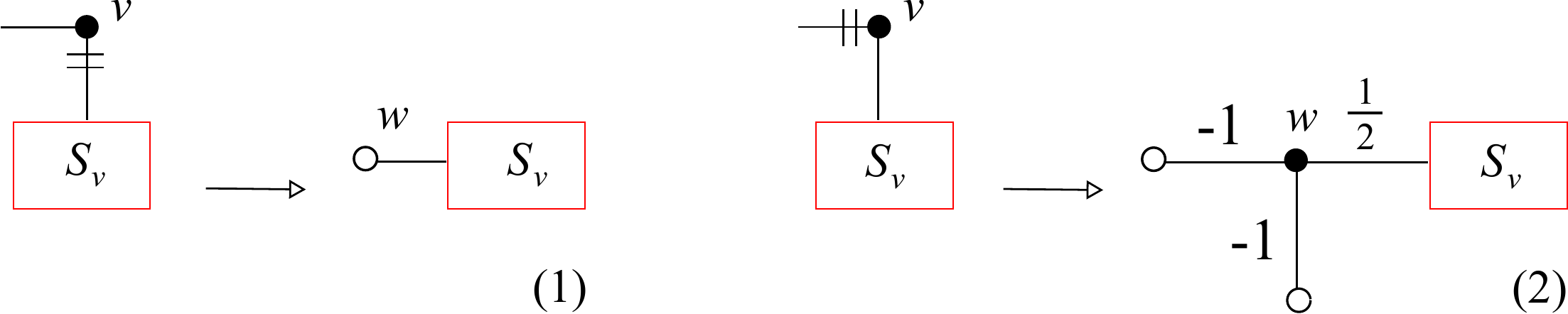}
\nota{Both these cuts produce a new tree $T'$ with levels which encodes a shadow $X'$ with $\partial N(X')=S^3$.}
\label{stacca_generalized:fig}
\end{center}
\end{figure}

\begin{figure}
\begin{center}
\includegraphics[width = 11 cm]{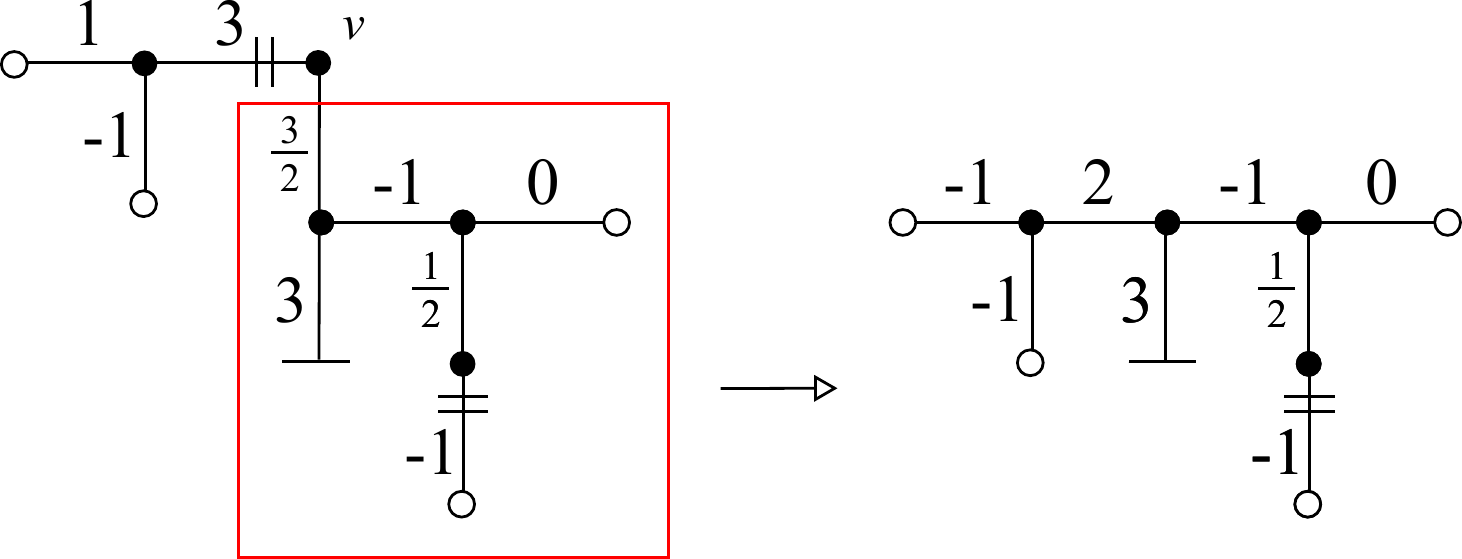}
\nota{How to apply the move in Fig.~\ref{stacca_generalized:fig}-(2). Note that $1/2+3/2 = 2$.}
\label{leveled_example:fig}
\end{center}
\end{figure}

\subsection{Nice flat vertices}
We may suppose that flat vertices only occur in some ``nice'' position, which we now explain. 

\begin{figure}
\begin{center}
\includegraphics[width =12.5 cm]{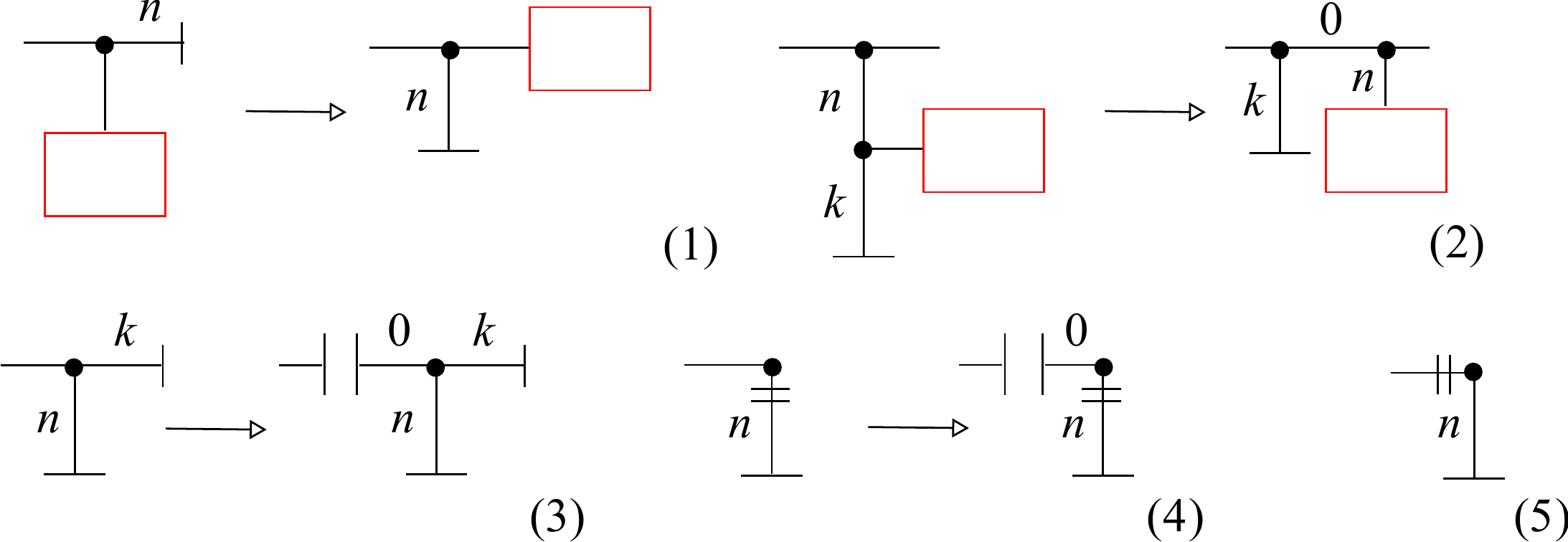}
\nota{The moves (1)-(4) transform a tree $T$ with levels encoding a shadow $X$ of some block $(M,L)$ into another tree $T'$ with levels encoding a shadow $X'$ of some block $(M',L')$. The former $(M,L)$ is homeomorphic to $(M',L')$, possibly after one assembling. A portion as in (5) cannot occur.}
\label{flat:fig}
\end{center}
\end{figure}

\begin{prop} \label{flat:prop}
Let $T$ be a decorated tree with levels, encoding a shadow $X$ of some block $(M,L)$. 
\begin{itemize}
\item Each of the moves in Fig.~\ref{flat:fig}-(1,2,3,4) transforms $T$ into a decorated tree $T'$ with levels encoding a shadow $X'$ of some block $(M',L')$. The block $(M,L)$ is homeomorphic to $(M',L')$ or obtained from it via an assembling. 
\item The tree $T$ cannot contain a portion as in Fig.~\ref{flat:fig}-(5).
\end{itemize}
\end{prop}
\begin{proof}
The move (1) is simply a changing of level function. Move (2) is similar to Fig.~\ref{move_leaf:fig}-(2). Move (3) and (4) follow from Proposition \ref{disc:prop}: the non-flat vertex gives a block $P^2\times S^1$ or $(A,2)$, which we see as a link complement from Table \ref{pieces:table}. The flat vertices produce an $\infty$ Dehn filling. The result is a solid torus which yields a vertical compressing disc, so that Fig.~\ref{hv:fig}-(1) applies.

On the other hand, an $\infty$ filling on the knot winding once in the picture representing $(A,2)$ does not give a solid torus. Proposition \ref{4:consequence:prop} thus forbids Fig.~\ref{flat:fig}-(5).
\end{proof}

\begin{figure}
\begin{center}
\includegraphics[width = 2 cm]{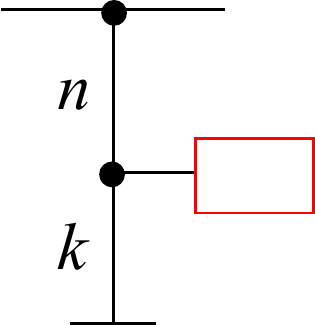}
\nota{Such a flat vertex is not nice.}
\label{not_nice:fig}
\end{center}
\end{figure}

A flat vertex $v$ is \emph{nice} if it is a leaf and is not contained in a portion as in Fig.~\ref{not_nice:fig}. (In other words, $v$ is nice if it is adjacent to a 3-valent vertex $v'$ with $l(v')< l(v)$, which is not itself adjacent to another 3-valent vertex $v''$ with $l(v'')<l(v)$.)  By the following result, we may suppose that every vertex is nice.
\begin{cor} \label{flat:cor}
Let $T$ be a decorated tree with levels encoding a shadow $X$ of some block $(M,L)$. The block is obtained via assemblings from $(M_1,L_1)\sqcup (M_2,L_2)$ where $(M_1,L_1)$ is a graph manifold generated by $\calS_0$ and $(M_2,L_2)$ has a shadow $X'$ encoded via a decorated tree $T'$ with levels such that
\begin{enumerate}
\item every flat vertex of $T'$ is nice;
\item the tree $T'$ has no more vertices than $T$.
\end{enumerate}
\end{cor}
\begin{proof}
We may suppose that every flat vertex is a leaf by using the moves in Fig.~\ref{flat:fig}-(1,3,4). Each such move de-assembles a 4-dimensional graph manifold. We then eliminate the configurations as in Fig.~\ref{not_nice:fig} using Fig.~\ref{flat:fig}-(2).
\end{proof}

\subsection{Fruits}
Take a decorated tree $T$ with levels and a 3-valent vertex $v$. If $S_v$ is as in Fig.~\ref{fruit:fig}-(1), we call it a \emph{fruit}. The vertex $v$ is the \emph{base} of the fruit. Note that a fruit encodes a projective plane in the shadow.

\begin{figure}
\begin{center}
\includegraphics[width = 4 cm]{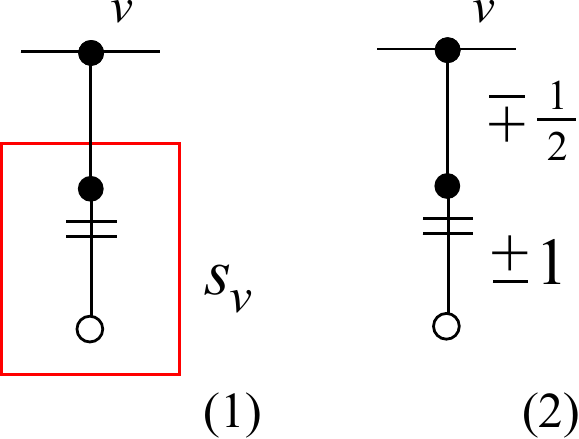}
\nota{A fruit based at some vertex $v$ (1). It must be decorated as in (2): there are two possibilities (signs do not match).}
\label{fruit:fig}
\end{center}
\end{figure}

\begin{figure}
\begin{center}
\includegraphics[width = 5.5 cm]{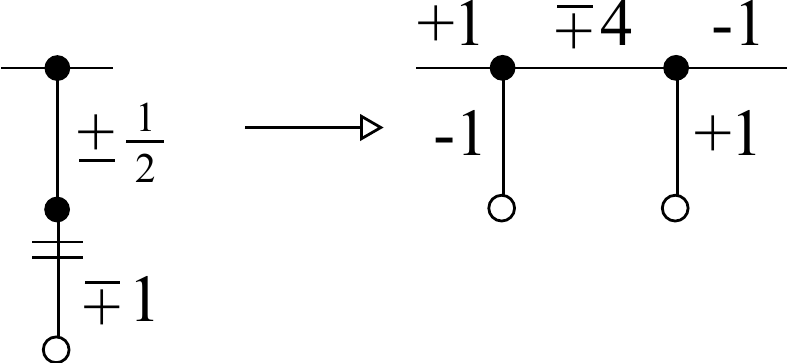}
\nota{These moves transform a decorated tree with levels into another.}
\label{move_fruit:fig}
\end{center}
\end{figure}

\begin{prop} \label{fruit:prop}
Let $T$ be a decorated tree with levels. A fruit is decorated as in Fig.~\ref{fruit:fig}-(2). The move in Fig.~\ref{move_fruit:fig} transforms $T$ into another tree $T'$ with levels.
\end{prop}
\begin{proof}
Take a fruit, decorated with some gleams $a$ and $b-1/2$ as in Fig.~\ref{fruit_proof:fig}-(1), with $a,b$ both integers. As in the proof of Proposition \ref{leaf:prop}, by cutting the lowest vertex as in Fig.~\ref{stacca_generalized:fig}-(1) we find that we must have $a=\pm 1$. Up to switching both gleams we may set $a =-1$. The moves in Fig.~\ref{fruit_proof:fig}-(2) produce a tree $T'$ with levels encoding some shadow $X'$. 

We can cut $X'$ as in Fig.~\ref{fruit_proof:fig}-(3). The result is a shadow $X''$ with $\partial N(X'')=S^3$ by Proposition \ref{cut:prop}. It is made of three discs with gleams $b$, $-1$, $-3$. This may be further transformed into two spheres intersecting transversely in a point, with Euler numbers $b-1$ and $-4$ using Fig.~\ref{perturb:fig}. Since $\partial N(X'')=S^3$, one such sphere must have Euler number zero, and hence $b=1$ as required. (See Lemma \ref{plumbing:lemma}.)

Finally, the move in Fig.~\ref{move_fruit:fig} is constructed in Fig.~\ref{fruit_proof:fig}-(4).
\end{proof}
As above, recall that $T$ and $T'$ represent non-homeomorphic blocks in general.

\begin{figure}
\begin{center}
\includegraphics[width = 12.5 cm]{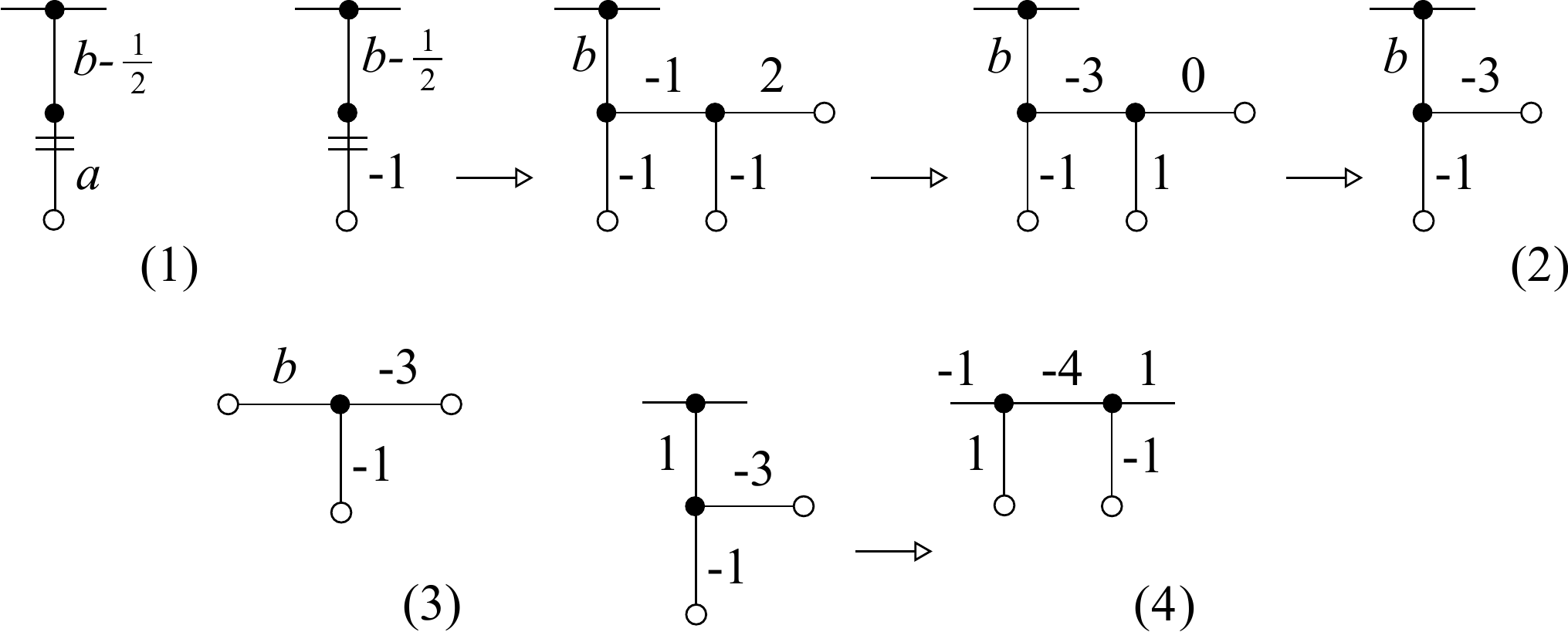}
\nota{Proof of Proposition \ref{fruit:prop}. A fruit (1). We use the moves in Figg.~\ref{elimina5:fig}, \ref{move_leaf:fig}, \ref{sum_ass:fig}-(1) in (2). We cut it. The resulting shadow $X''$ has $\partial N(X'')=S^3$: therefore $b=1$ (3). We use the move in Fig.~\ref{thickening:fig}-(2) to conclude (4).}
\label{fruit_proof:fig}
\end{center}
\end{figure}

\subsection{Branches}
Take a decorated tree $T$ with levels and a vertex $v$ of type \includegraphics[width = 0.6 cm]{4.pdf} or \includegraphics[width = 0.6 cm]{5.pdf}. If $S_v$ is not contained in a leaf or in a fruit we call it a \emph{branch}. See an example in Fig.~\ref{leveled_generalized:fig}.

\begin{figure}
\begin{center}
\includegraphics[width = 8 cm]{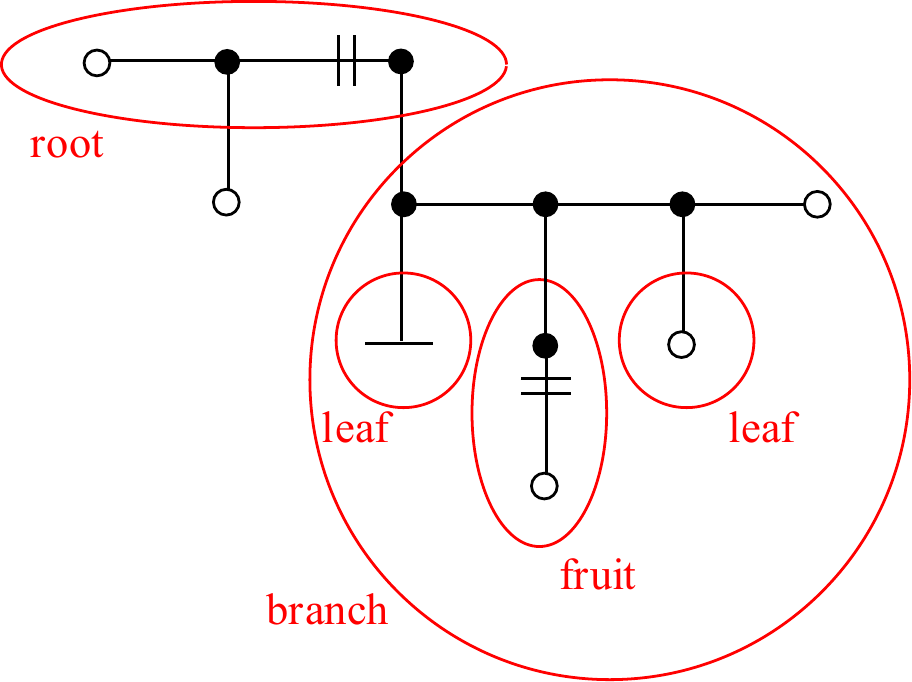}
\nota{Branches, leaves, and fruits on a tree with levels.}
\label{leveled_generalized:fig}
\end{center}
\end{figure}

Depending on its base $v$, there are three types of branches, shown in Fig.~\ref{Sv:fig}. We will prove Theorem \ref{main:teo} inductively by simplifying branches, starting from the ones of highest level. It is relatively easy to simplify a branch of type (1) or (3) from Fig.~\ref{Sv:fig}. Unfortunately, more work needs to be done to simplify branches of type (2). Therefore we call a branch as in Fig.~\ref{Sv:fig}-(2) a \emph{bad branch}. We now analyze bad branches.

Let $T$ be a decorated tree with levels defining a shadow $X$.
Let a vertex $v$ be the base of a bad branch. It defines a block $M_v \isom (A,2)$ in the decomposition of $\partial N(X)$. The branch $S_v$ in turn defines a solid torus $M_{S_v}$ whose meridian is attached to a boundary component $T$ of $(A,2)$. The torus $T$ has a preferred homology basis: the meridian $\mu$ is the fiber $\pi^{-1} (x)$ of a point in $X$ along the natural projection $\pi:\partial N(X) \to X$. The longitude $\lambda$ is the fiber of the Seifert fibration $(A,2)$.

The meridian of the solid torus $M_{S_v}$ is attached along a curve $\mu + q\lambda$. We call the integer $q$ the \emph{torsion} of the bad branch. We show some examples (omitting the proof).

\begin{example}
Two bad branches with torsion $q$ are shown in Fig.~\ref{bad_branch_examples:fig}-(1,2).
\end{example}

\begin{figure}
\begin{center}
\includegraphics[width = 9 cm]{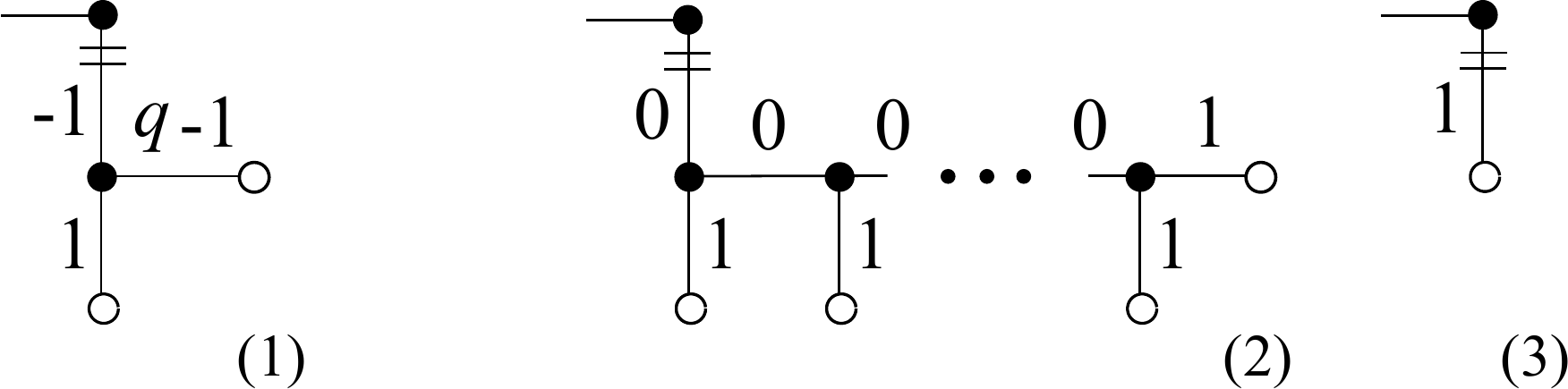}
\nota{Two bad branches with torsion $q$. The branch (2) contains $q\geqslant 1$ leaves: when $q$ is negative, simply reverse the signs of all the gleams (the branch is not defined when $q=0$). 
Such a branch describes a particular ``tower'', using the terminology of \cite{CoThu}. 
When $q=1$ the branch (2) is as in (3).}
\label{bad_branch_examples:fig}
\end{center}
\end{figure}

\begin{figure}
\begin{center}
\includegraphics[width = 7 cm]{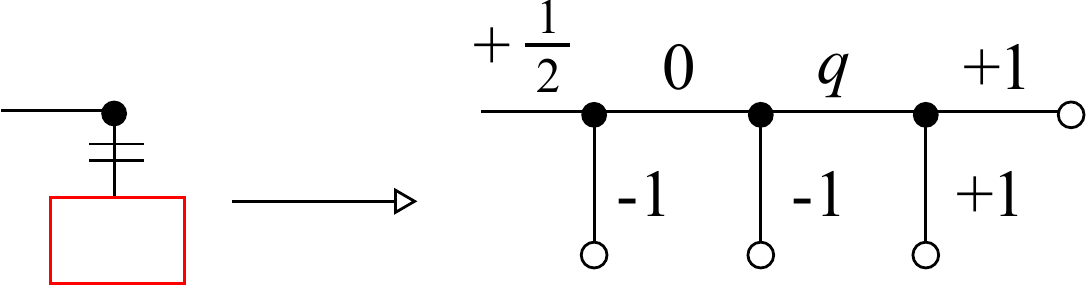}
\nota{This move kills a bad branch and transforms a decorated tree with levels into another decorated tree with levels. }
\label{bad_branch_move1:fig}
\end{center}
\end{figure}

\begin{prop} \label{bad:branch:prop}
Let $T$ be a decorated tree with levels containing a bad branch with torsion $q$. The move in Fig.~\ref{bad_branch_move1:fig} transforms $T$ into another decorated tree $T'$ with levels.
\end{prop}
\begin{proof}
If we substitute a bad branch with torsion $q$ with another bad branch having the same torsion $q$ we get a new decorated tree with levels. Here, we substitute the bad branch with the one in Fig.~\ref{bad_branch_examples:fig}-(1). Then we modify as in Fig.~\ref{bad_branch_proof:fig}.
\end{proof}

\begin{figure}
\begin{center}
\includegraphics[width = 10 cm]{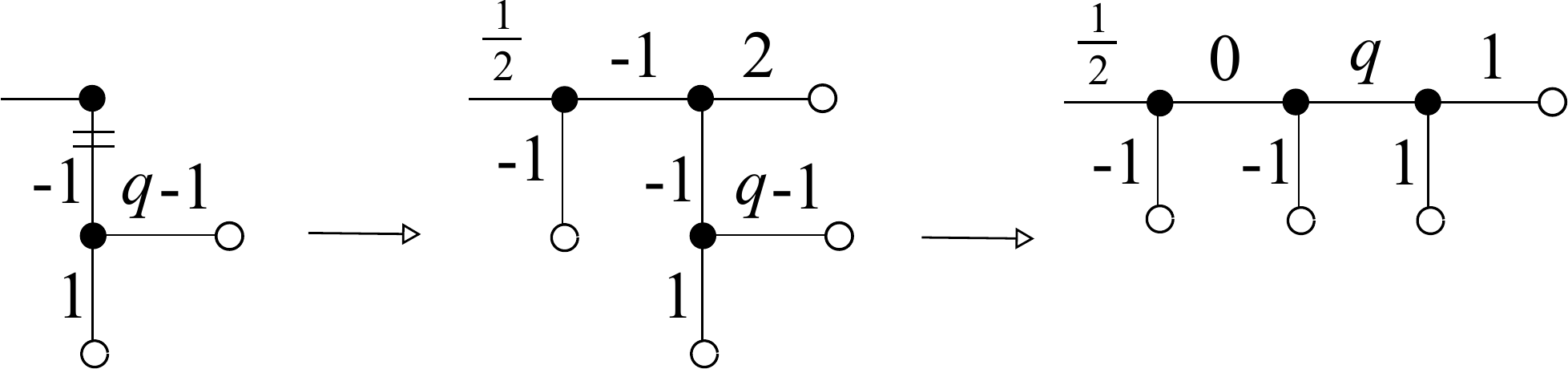}
\nota{We may substitute any bad branch with another branch having the same torsion $q$. Take the portion on the left. Then modify it by using Fig.~\ref{elimina5:fig}-(1) and Fig.~\ref{thickening:fig}-(2).}
\label{bad_branch_proof:fig}
\end{center}
\end{figure}

\begin{figure}
\begin{center}
\includegraphics[width = 8 cm]{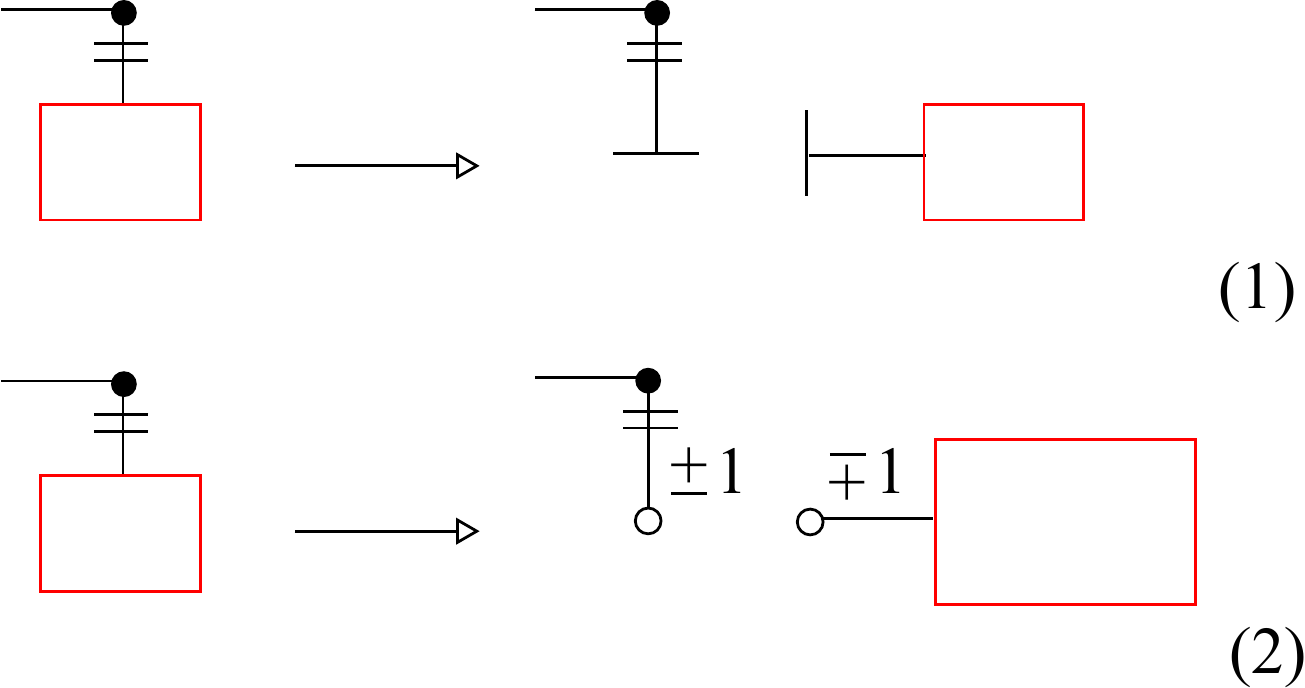}
\nota{If $q=0$, the move (1) applies. If $q=\pm 1$, the move (2) applies. The former block is an assembling (1) or a connected sum (2) of the new one.}
\label{bad_branch_move2:fig}
\end{center}
\end{figure}

More can be done if $|q|\leqslant 1$.

\begin{prop}  \label{reducible:prop}
Let $T$ be a decorated tree with levels encoding a shadow $X$ of a block $(M,L)$. Let $T$ contain a bad branch with torsion $q$. If $q=0$ (resp. $\pm 1$), the move in Fig.~\ref{bad_branch_move2:fig}-(1) (resp. (2)) produces a decorated tree $T'$ with levels encoding a shadow $X'$ of a block $(M',L')$. The block $(M,L)$ is an assembling (resp. connected sum) of $(M',L')$.
\end{prop}
\begin{proof}
If $q=0$, the meridian of the solid torus is vertical. If $q= \pm 1$, it is horizontal. This gives a vertical or horizontal compressing disc and Lemma \ref{compressing:lemma} applies. The moves in Fig.~\ref{hv:fig} may be represented here as in Fig.~\ref{bad_branch_move2:fig}.
\end{proof}

Let us say that a bad branch is \emph{reducible} if one of the following holds:
\begin{itemize}
\item $q=0$ and the branch does not consist of a single flat vertex;
\item $q=\pm 1$ and the branch does not consist of a single fat vertex.
\end{itemize}
A reducible bad branch can indeed be simplified thanks to Proposition \ref{reducible:prop}. We will thus focus on non-reducible bad branches.

\section{Plumbing lines} \label{plumbing:section}
We will use some techniques that were inspired by a paper of Neumann and Weintraub~\cite{Neu}. In that paper, the authors classified the closed 4-manifolds that may be obtained by adding a 4-handle to a plumbing of spheres. What we do here is in fact a generalization of that result, since a plumbing of sphere becomes a simple polyhedron without vertices after perturbing the double points as in Fig.~\ref{perturb:fig}. Our generalization is twofold: we consider any kind of simple polyhedron without vertices, and we also admit 3-handles.

Recall that a \emph{plumbing} of spheres in a $4$-manifold is a subspace
consisting of some embedded oriented (locally flat) $2$-spheres with 
transverse intersections. Its regular neighborhood is encoded by the
\emph{plumbing graph}, having a vertex for each sphere, decorated with
the Euler number of its normal bundle (\emph{i.e.}~its algebraic self-intersection), and an edge for each intersection, decorated with
its sign. In particular, in a \emph{plumbing line} as in
Fig.~\ref{plumbing_line:fig} we can orient all spheres in order to get
positive intersections, so that the plumbing is determined by the sequence of Euler
numbers $(e_1,\ldots,e_n)$. The \emph{boundary} of the plumbing is the
$3$-dimensional boundary of its regular neighborhood. 

\begin{figure}
\begin{center}
\resizebox{6 cm}{!}{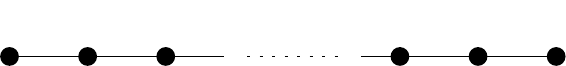}
\nota{A plumbing line.}
\label{plumbing_line:fig}
\end{center}
\end{figure}

\begin{lemma}[Neumann-Weintraub, \cite{Neu}]\label{easy:plumbing:lemma}
Let $(e_1,\ldots,e_n)$ be a plumbing line, whose boundary is homeomorphic to $S^3$. We have $|e_i|\leqslant 1$ for at least one value of $i$.
\end{lemma}

We need here the following stronger version of Lemma \ref{easy:plumbing:lemma}. If $(e_1,\ldots,e_n)$ is a plumbing line, note that $(e_n,\ldots,e_1)$ and $(-e_1,\ldots,-e_n)$ are plumbing lines defining the same unoriented 4-manifold.

\begin{lemma}\label{plumbing:lemma}
Let $(e_1,\ldots,e_n)$ be a plumbing line, whose boundary is homeomorphic to either $S^3$ or $S^2\times S^1$. Up to reversing the sequence and/or changing all signs, one of the following holds.
\begin{itemize}
\item $e_1 = 0$,
\item $e_1=1$ and $n=1$,
\item $e_1=1$ and $e_2\in \{0,1,2,3\}$,
\item $e_i=0$ for some $i\not\in \{1,n\}$ and $e_{i-1}e_{i+1}\leqslant 0$,
\item $e_i= 1$ for some $i\not\in \{1,n\}$ and $e_{i-1} \in\{0,1,2,3\}, e_{i+1}\geqslant 0$.
\end{itemize}
\end{lemma}

\begin{proof}
We prove the assertion by contradiction. Therefore we suppose that
\begin{itemize}
\item if $e_i=0$, then $i\not\in\{1,n\}$ and $e_{i-1}e_{i+1}>0$;
\item if $e_i=\pm 1$, one of the following holds:
	\begin{enumerate}
	\item there is a $j\in\{i-1, i+1\}\cap\{1,\ldots,n\}$ such that $e_ie_j<0$, or 
	\item we have $e_ie_j\geqslant 4$ for all $j\in\{i-1,i+1\}\cap \{1,\ldots,n\}$. 
	\end{enumerate}
\end{itemize}
and we conclude that the boundary of a regular neighborhood of the plumbing is neither homeomorphic to $S^3$ nor to $S^2\times S^1$.

The fundamental group of the boundary is a cyclic group, whose order
is the absolute value of the determinant of the bilinear form on
$H_2$ (the order is infinite when this value is zero). This determinant is
$$f(e_1,\ldots,e_n) = \det \left(\begin{array}{ccccc} 
e_1 & 1 & 0 & \ldots & 0 \\
1 & e_2 & 1 & \ddots & \vdots \\
0 & 1 & e_3 & \ddots & 0 \\
\vdots & \ddots & \ddots & \ddots & 1 \\
\phantom{\Big|}\!\! 0 & \ldots & 0 & 1 & e_n
\end{array}\right).$$
We have the following equalities
\begin{eqnarray}
f(\emptyset) & = & 1 \\
f(e_1) & = & e_1 \\
f(e_1,\ldots,e_n) & = & e_1f(e_2,\ldots,e_n)-f(e_3,\ldots,e_n)
\label{first:eqn} \\
f(\ldots,e_{i-1},0,e_{i+1},\ldots) & = & -f(\ldots,e_{i-1}+e_{i+1},\ldots), \\
f(0,e_2,e_3,\ldots ) & = & -f(e_3,\ldots), \\
f(\ldots,e_{i-1},1,e_{i+1},\ldots) & = & f(\ldots,e_{i-1}-1,e_{i+1}-1,\ldots), \\
f(1,e_2,\ldots) & = & f(e_2-1,\ldots), \\
f(\ldots,e_{i-1},1,1,e_{i+2},\ldots) & = & -f(\ldots,
e_{i-1}+e_{i+2}-1, \ldots). \label{last:eqn}
\end{eqnarray}
We prove now by induction on $n$ that $|f(e_1,\ldots,e_n)|\geqslant 2$. This implies that the boundary of the plumbing is neither $S^3$ nor $S^2\times S^1$.

If $n=1$, we have
$|f(e_1)|=|e_1|\geqslant 2$ by our hypothesis above. Suppose now $n>1$. If $|e_i|\geqslant 2$
for all $i$, equation (\ref{first:eqn}) gives
$|f(e_1,\ldots,e_{i+1})|>|f(e_1,\ldots,e_i)|$ for all $i$, and we are
done.

If there is a $e_i=0$, then $f(\ldots,
e_{i-1},0,e_{i+1},\ldots) = -f(\ldots,e_{i-1}+e_{i+1},\ldots)$. By hypothesis $e_{i-1}e_{i+1}>0$, hence $|e_{i-1}+e_{i+1}|\geqslant 2$ and the shorter sequence
$(\ldots,e_{i-1}+e_{i+1},\ldots)$ is easily seen to still satisfy our induction hypothesis
(note that $e_{i-1}, e_{i+1}$, and $e_{i-1}+e_{i+1}$ all
have the same sign and thus $|e_{i-1}+e_{i+1}|\geqslant |e_{i-1}|+|e_{i+1}|$). Therefore we conclude.

We may now suppose that $e_i\neq 0$ for all $i$. Hence $e_i=\pm 1$ for
some $i$, say $e_i=1$. We consider the case $i=1$. We have $f(1,e_2,\ldots) =
f(e_2-1,\ldots)$.  By our hypothesis we have either $e_2<0$ or $e_2\geqslant 4$. In the first case, the shorter sequence $(e_2-1,\ldots)$ still satisfies the induction hypothesis. In the second case, it also does, except if $(e_1,e_2,e_3, e_4,\ldots )= (1,4,1, e_4,\ldots)$ and $e_4\geqslant 4$. The new sequence is $(3,1,e_4,\ldots)$, which may in turn be shortened to $(2,e_4-1,\ldots)$. Again, we are done except when $(\ldots, e_4, e_5, e_6,\ldots) = (\ldots, 4, 1, e_6,\ldots)$ with $e_6\geqslant 4$. By repeating this argument we eventually end up with a sequence $(2,\ldots,2,e_{2k}-1,\ldots)$ with $e_{2k-1}>4$, or $(2,\ldots,2)$, or $(2,\ldots,2,3)$. Each of these satisfies our hypothesis, so we are done.

Consider the case $i\not\in\{1,n\}$. One of the following holds.
\begin{enumerate}
\item we have $e_{i-1}<0$ (up to reversing the sequence), or 
\item we have $e_{i-1}, e_{i+1}\geqslant 4$.
\end{enumerate}
We have $f(\ldots,e_{i-1},1,e_{i+1},\ldots) =
f(\ldots,e_{i-1}-1,e_{i+1}-1,\ldots)$. Suppose (1) holds. The new sequence satisfies our hypothesis except if one of the following holds:
\begin{itemize}
\item $e_{i+1} = 1$,
\item $(\ldots, e_{i+1}, e_{i+2}, e_{i+3},\ldots) = (\ldots,4,1,e_{i+3},\ldots )$ with $e_{i+3}\geqslant 4$.
\end{itemize}

If $e_{i+1}=1$, we
have $(\ldots, e_{i-1},1,1,e_{i+2},\ldots )$ with $e_{i+2}<0$. Equation (\ref{last:eqn}) gives
$f(\ldots,e_{i-1},1,1,e_{i+2},\ldots) = -f(\ldots, e_{i-1}+e_{i+2}-1,
\ldots)$ and the new sequence fullfills
the hypothesis. 

If the second case holds, we repeat our argument as above and end up with a shorter sequence of type $(\ldots, e_{i-1}-1,2,\ldots,2,e_h -1, \ldots)$ with $e_h>4$, or $(\ldots, e_{i-1}-1,2,\ldots,2)$, or $(\ldots,e_{i-1}-1,2,\ldots,2,3)$. Each such satisfies the hypothesis.

Suppose (2) holds. The new sequence fulfills the hypothesis, except if one of the following holds:
\begin{itemize}
\item $(\ldots, e_{i+1}, e_{i+2}, e_{i+3},\ldots ) = (\ldots, 4,1,e_{i+3},\ldots)$ with $e_{i+3}\geqslant 4$, or
\item $(\ldots, e_{i-3}, e_{i-2}, e_{i-1},\ldots) = (\ldots, e_{i-3},1,4,\ldots)$ with $e_{i-3}\geqslant 4$.
\end{itemize}
Both cases may hold. For each such we proceed as above.
\end{proof}

\section{Proof of the theorem} \label{proof:section}
Finally, we prove here Theorem \ref{main:teo}. We start with a lemma.

\begin{figure}
\begin{center}
\includegraphics[width = 12.5 cm]{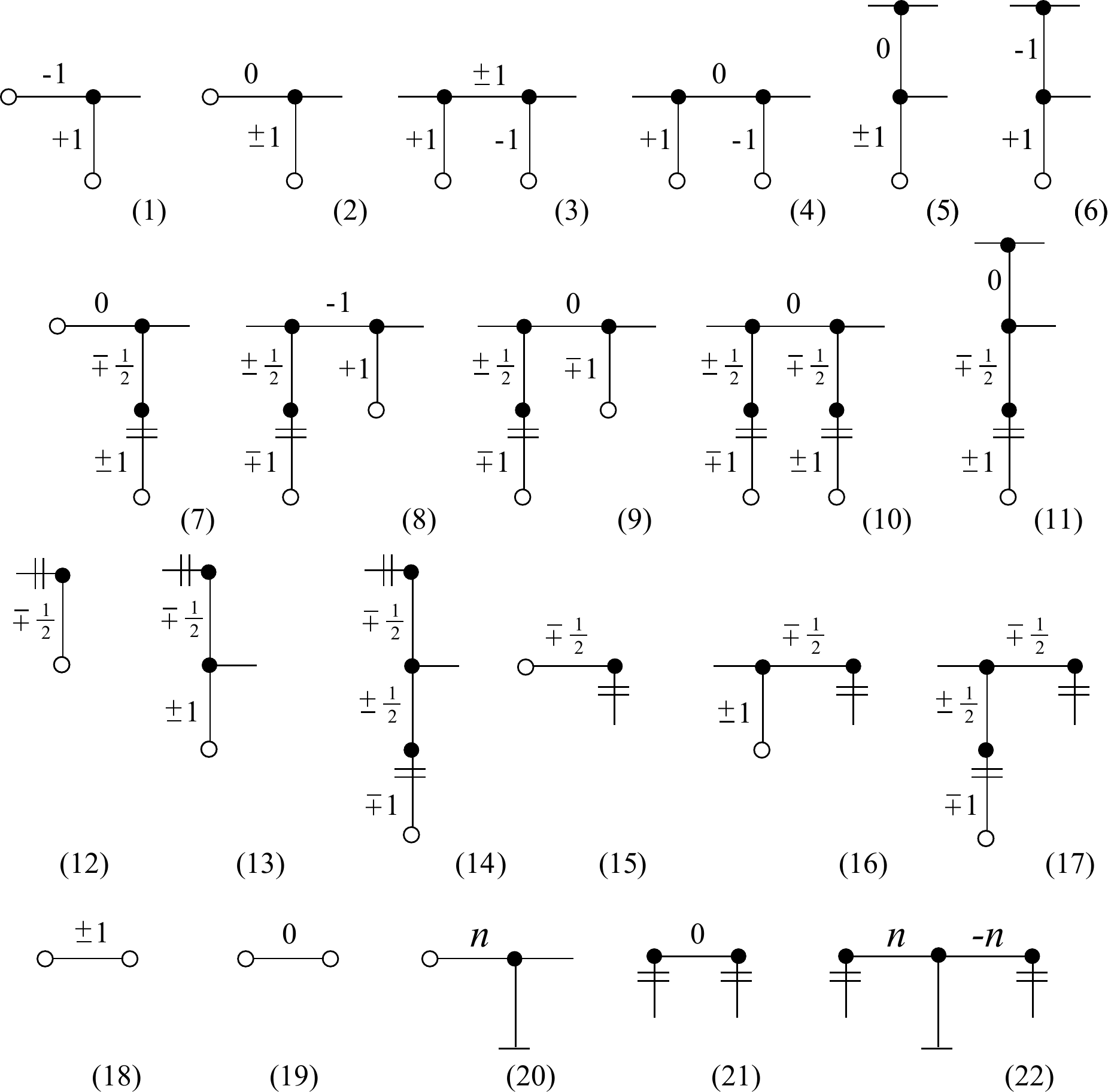}
\nota{A decorated tree with levels contains either a non-nice flat vertex, a reducible bad branch, or one one of these portions.}
\label{local:fig}
\end{center}
\end{figure}

\begin{lemma} \label{contains:lemma}
Let $T$ be a decorated tree with levels. One of the following holds:
\begin{itemize}
\item the tree contains a flat vertex which is not nice;
\item the tree contains a reducible bad branch;
\item the tree contains a portion as in Fig.~\ref{local:fig}, possibly after applying some moves as in Fig.~\ref{move_leaf:fig}.
\end{itemize}
\end{lemma}
\begin{proof}
We suppose that every flat vertex is nice and that there are no reducible branches in $T$. We deduce that $T$ contains a portion as in Fig.~\ref{local:fig}.
We start by claiming that $T$ contains a portion $Z$ as in Fig.~\ref{branch:fig}-(1), such that:
\begin{figure}
\begin{center}
\includegraphics[width =10 cm]{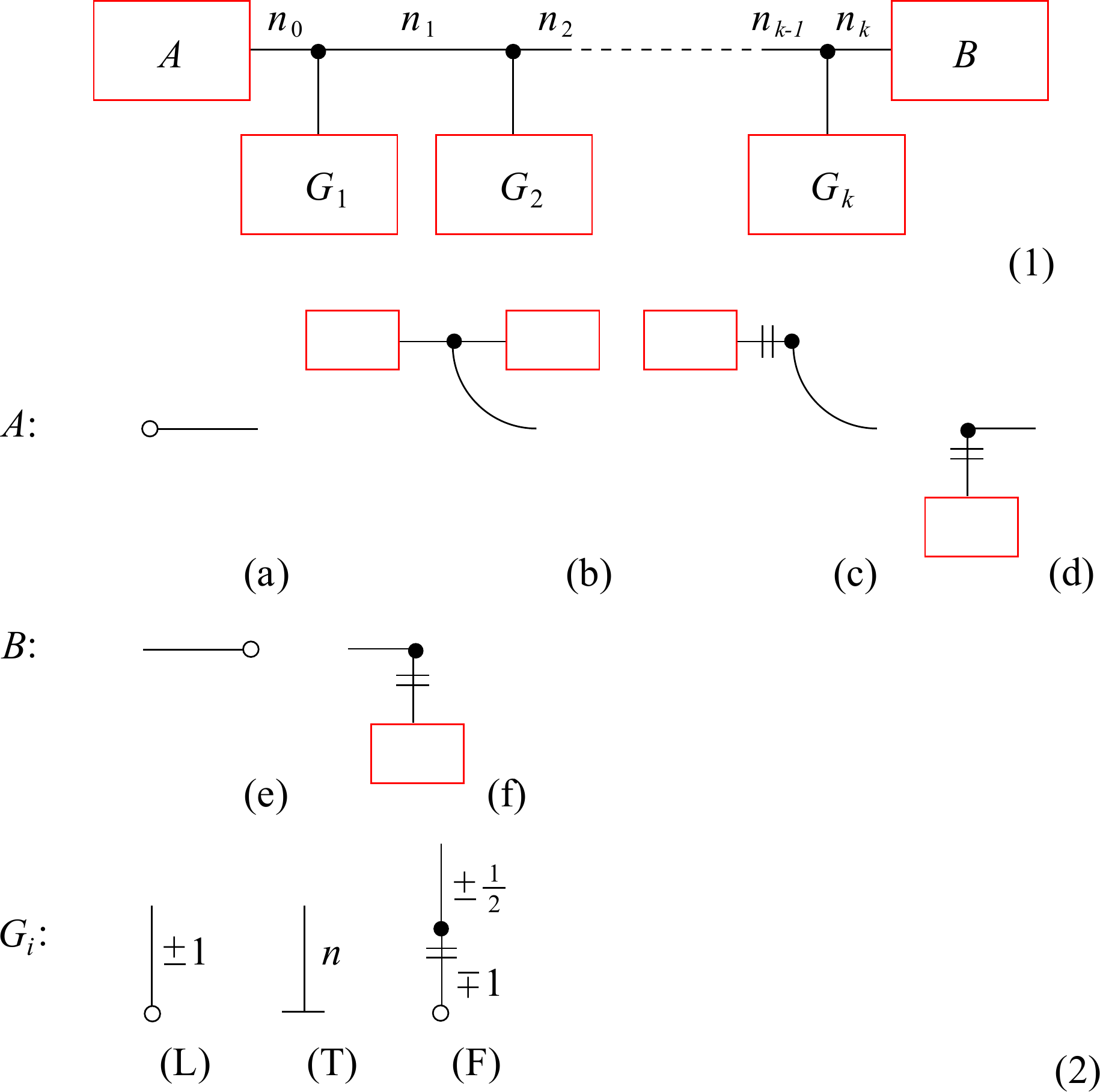}
\nota{The tree $T$ contains a portion $Z$ as in (1), where each $G_i$ is either a fruit or a leaf, and $A, B$ is of one
  of the types shown in (2). (When A is of type (a) or (d) the portion $Z$ is actually the whole tree $Z=T$.)}
\label{branch:fig}
\end{center}
\end{figure}
\begin{enumerate}
\item every $G_i$ is either a leaf or a fruit, see Fig.~\ref{branch:fig}-(2);
\item the portion $A$ is one of those (a), (b), (c), (d) shown in Fig.~\ref{branch:fig}-(2);
\item the portion $B$ is one of those (e), (f) shown in Fig.~\ref{branch:fig}-(2). 
\end{enumerate}
Let the \emph{level} of a branch $S_v$ be the level $l(v)$ of its base vertex $v$.
If there is no branch at all in $T$, then the whole tree $T$ is as in Fig.~\ref{branch:fig}-(1) (with $A$ of type (a) and B of type (e)) and we may take $Z=T$. Otherwise, consider a branch having the highest level among branches. The branch is as in Fig.~\ref{branch:fig}-(1) with B of type (e) and $A$ either of type (b), (c), or \includegraphics[width = 1 cm]{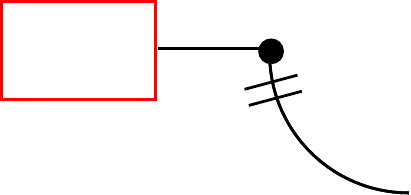}. We are done, except when the latter case holds, \emph{i.e.}~when the branch is bad. 

To avoid bad branches, we take $Z$ as a branch having the highest level among good branches. (Again, if there are no good branches, take $Z=T$.) The portion $Z$ is as required. 

\begin{figure}
\begin{center}
\includegraphics[width = 7 cm]{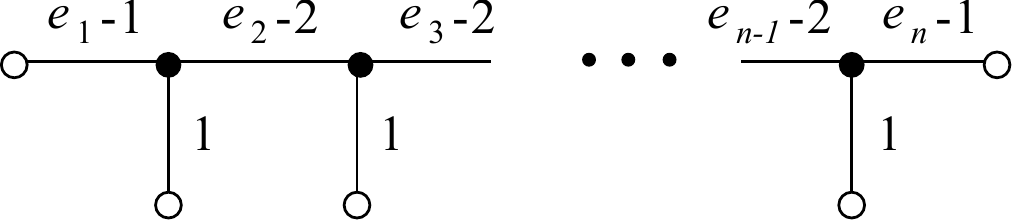}
\nota{A tree with one level is the perturbation of the plumbing as in Fig.~\ref{plumbing_line:fig}.}
\label{plumbing_line2:fig}
\end{center}
\end{figure}

We now construct a plumbing line from $Z$. Actually, we construct a tree with levels as in Fig.~\ref{plumbing_line2:fig}, which in turn may be transformed into a plumbing line as in Fig.~\ref{plumbing_line:fig} via the (inverse of the) move that perturbs double points, see Fig.~\ref{perturb:fig}. 

The tree with levels is constructed by substituting the pieces A and B as prescribed by Fig.~\ref{ripeto:fig}-(b,c,d,f), and each flat leaf and fruit as in Fig.~\ref{ripeto:fig}-(T,F).

\begin{figure}
\begin{center}
\includegraphics[width = 12.5 cm]{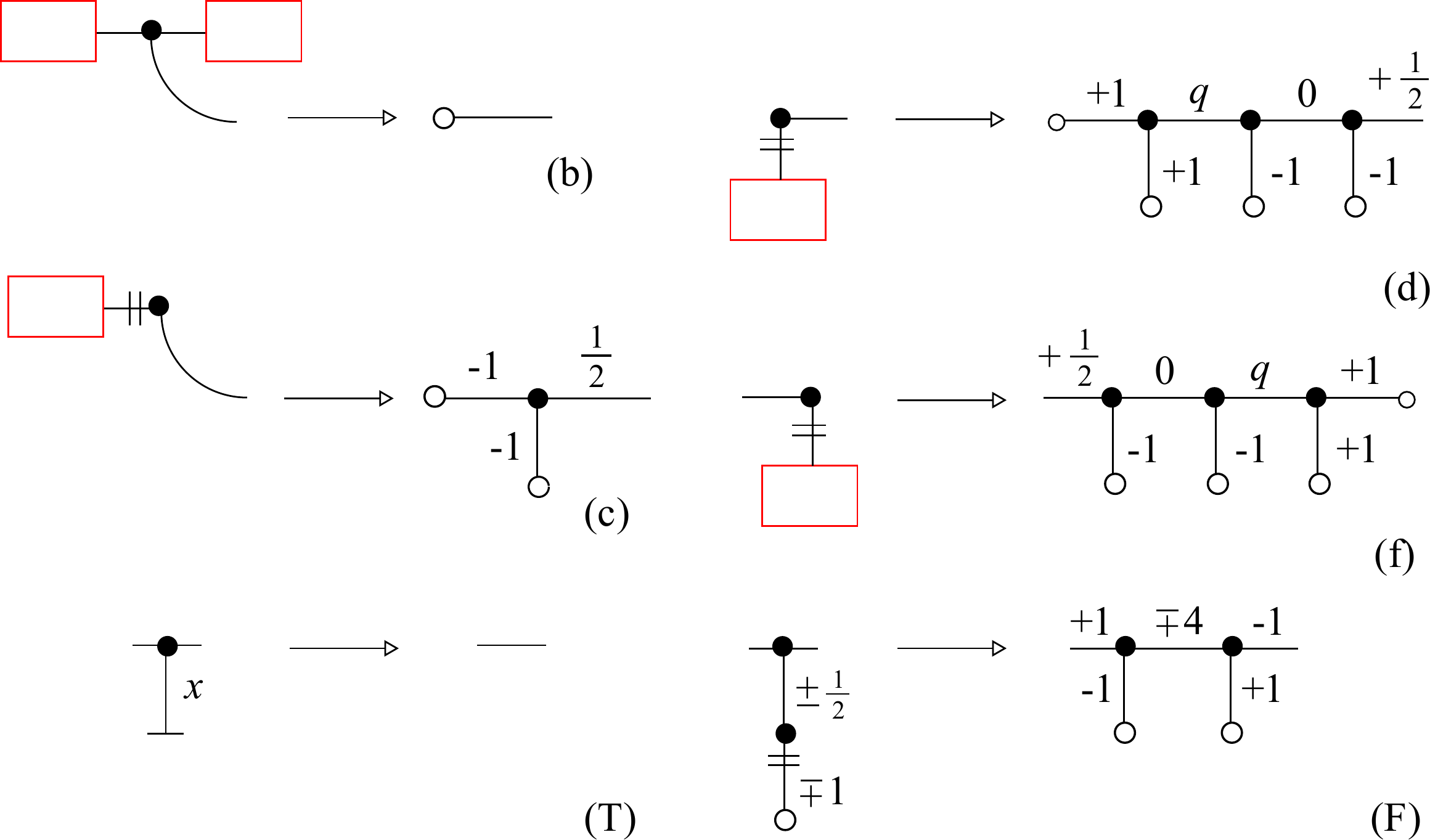}
\nota{We use Fig.~\ref{stacca:fig} (b), Fig.~\ref{stacca_generalized:fig}-(2) (c), Fig.~\ref{bad_branch_move1:fig} (d, f), Fig.~\ref{sum_ass:fig}-(3) (T), and Fig.~\ref{move_fruit:fig} (F). 
}
\label{ripeto:fig}
\end{center}
\end{figure}

If $A$ is of type (d) or $B$ is of type (f), it is a bad branch with some torsion $q$. By hypothesis, it is not reducible. In other words:
\begin{itemize}
\item if $q=0$, the bad branch consists of a single flat vertex;
\item if $q=\pm 1$, the bad branch consists of a single fat vertex.
\end{itemize}
Since every flat vertex is nice, the first case is excluded. Therefore $q\neq 0$. 

We end up with a decorated tree with levels as in Fig.~\ref{plumbing_line2:fig}, which determines a plumbing line as in Fig.~\ref{plumbing_line:fig}, with some integers $e_1,\ldots, e_n$. Now we apply Lemma \ref{plumbing:lemma}. The sequence $(e_1,\ldots, e_n)$ contains one of the following subsequences:
\begin{enumerate}
\item[(i)] $(0)$;
\item[(ii)] $(\pm 1)$;
\item[(iii)] $(0,e_2,\ldots)$; 
\item[(iv)] $(\ldots,e_{n-1},0)$;
\item[(v)] $(\ldots, e_{i-1},0,e_{i+1},\ldots)$ with $e_{i-1}e_{i+1}\leqslant 0$;
\item[(vi)] $(1,e_2,\ldots)$ with $e_2\geqslant 0$ not equal to 4; 
\item[(vii)] $(\ldots,e_{n-1}, 1)$ with $e_{n-1}\geqslant 0$ not equal to 4;
\item[(viii)] $(-1,e_2,\ldots)$ with $e_2\leqslant 0$ not equal to $-4$;
\item[(ix)] $(\ldots,e_{n-1}, -1)$ with $e_{n-1}\leqslant 0$ not equal to $-4$; 
\item[(x)] $(\ldots,e_{i-1},1,e_{i+1},\ldots)$ with $e_{i-1}\geqslant 0$, $e_{i+1}\geqslant 0$, not both equal to 4;
\item[(xi)] $(\ldots,e_{i-1},-1,e_{i+1},\ldots)$ with $e_{i-1}\leqslant 0$, $e_{i+1}\leqslant 0$, not both equal to -4.
\end{enumerate}
In the first two cases (i) and (ii) the sequence has only one element. The subsequence identifies a portion of $Z$. We now show that this portion is one of those listed in Fig.~\ref{local:fig}. To preserve clarity, we first suppose that $Z$ does not contain flat leaves. The portions $A$, $B$, $G_i$ of $Z$ contribute to the plumbing line $(e_1,\ldots, e_n)$ as follows, see Fig.~\ref{ripeto:fig}:
\begin{itemize}
\item portions of type $A$-(a) and $A$-(b) contribute in the same way;
\item a portion of type $A$-(c) contributes with $(-2, \ldots)$;
\item a portion of type $A$-(d) contributes with $(2,q,-2,\ldots)$;
\item a portion of type $B$-(f) contributes with $(\ldots,-2,q,2)$;
\item a fruit contibutes with an integer $\pm 4$.
\end{itemize}
Recall that $q$ is always non-zero. We consider first the case $k=0$, \emph{i.e.}~there is no $G_i$. The portion $Z$ thus consists of the pieces $A$ and $B$ glued together. The various possibilities are shown in Fig.~\ref{AB:fig}. We analyse each separately:

\begin{figure}
\begin{center}
\includegraphics[width = 12 cm]{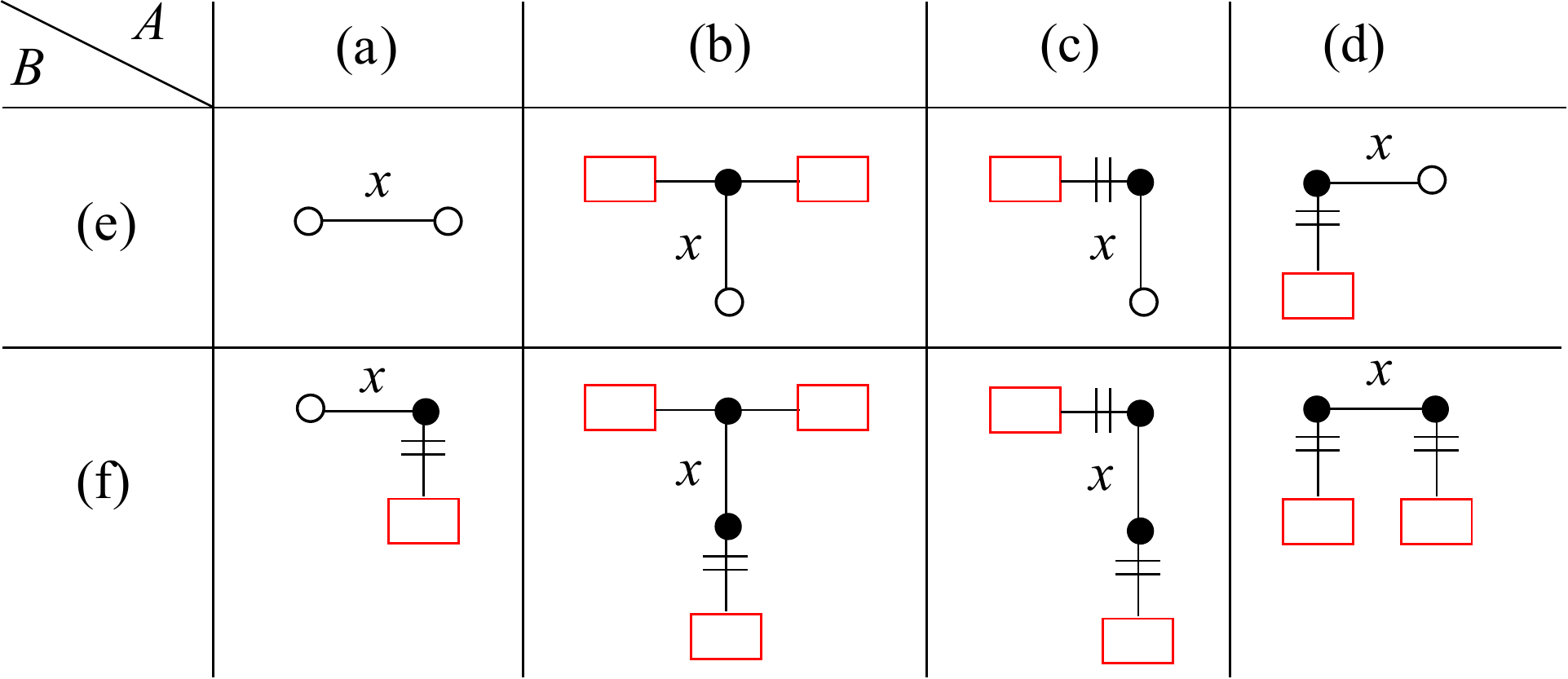}
\nota{When $k=0$, the portion $Z$ consists of $A$ and $B$ glued together and thus looks like one of the pictures listed here.}
\label{AB:fig}
\end{center}
\end{figure}

\begin{itemize}
\item[(ae)] the sequence consists of a single number $(x)$. By hypothesis, $|x|\leqslant 1$ which leads to (18) or (19);
\item[(be)] the portion $Z$ would be a leaf and not a branch: excluded;
\item[(ce)] the sequence is $(-2,x-1/2)$. Therefore $x=\pm 1/2$ which leads to (12);
\item[(de)] the sequence is $(2,q,-2,x-1/2)$ with $q\neq 0$. Therefore $x=\pm 1/2$ which leads to (15);
\item[(af)] like (de);
\item[(bf)] the sequence is $(x-1/2, -2, q, 2)$ with $q\neq 0$. As above, we get $x=\pm 1/2$. If $q= \pm 1$, the bad branch consists of a single vertex, and hence $Z$ is a fruit and not a branch: excluded. Therefore $|q|\geqslant 2$. However, the moves contained in the proof of Lemma \ref{plumbing:lemma} show that this sequence does not give $S^3$ or $S^2\times S^1$: excluded;
\item[(cf)] the sequence is $(-2,x-1,-2,q,2)$ with $q\neq 0$. Therefore $x=0$ and again this sequence does not give $S^3$ or $S^2\times S^1$;
\item[(df)] the sequence is $(2,q,-2,x-1,-2,q',2)$ with $q,q'\neq 0$. Therefore $x=0$ which leads to (21).
\end{itemize}

We turn to the case $k>0$. We consider first the portion formed by $A$ and $G_1$. It is as in Fig.~\ref{AG:fig}. We use implicitly Fig.~\ref{move_leaf:fig}-(1) at various points. We analyze each case separately:

\begin{figure}
\begin{center}
\includegraphics[width = 12 cm]{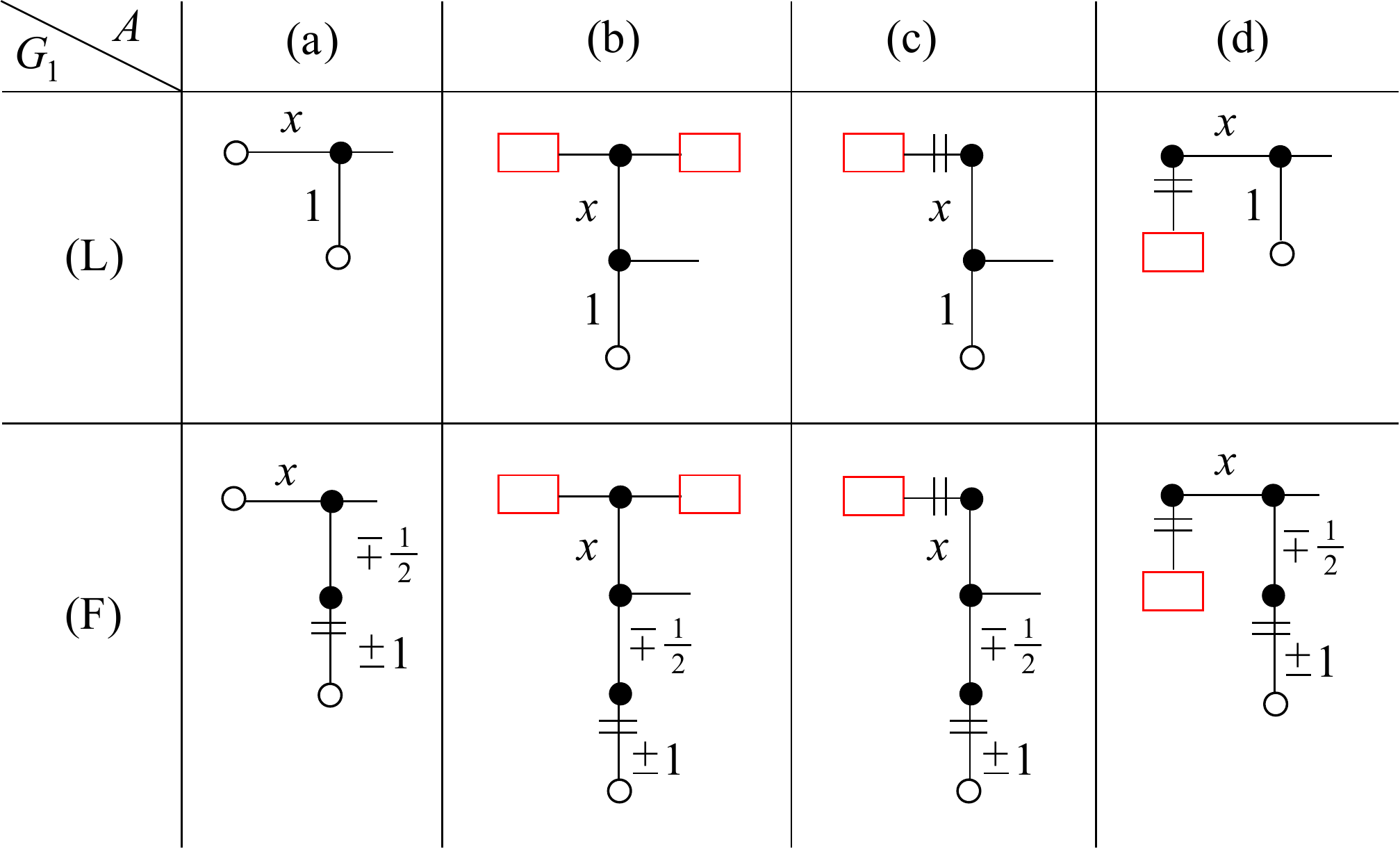}
\nota{Portions obtained as the union of $A$ and $G_1$.}
\label{AG:fig}
\end{center}
\end{figure}

\begin{itemize}
\item[(La)] the sequence starts as $(x+1,\ldots )$. If $|x+1| \leqslant 1$ we get either (1) or (2);
\item[(Lb)] the sequence starts as $(x+1,\ldots )$. If $|x+1|\leqslant 1$ we get either (5) or (6);
\item[(Lc)] the sequence starts as $(-2,x+1/2, \ldots)$. If $x+1/2 \in \{-1,0\}$ we get (13);
\item[(Ld)] the sequence starts as $(2,q,-2,x+1/2,\ldots)$. If $x+1/2 \in \{-1,0\}$ we get (16);
\item[(Fa)] the sequence starts as $(x, \pm 4, \ldots)$. If $x=0$ we get (7);
\item[(Fb)] the sequence starts as $(x, \pm 4, \ldots)$. If $x=0$ we get (11);
\item[(Fc)] the sequence starts as $(-2,x-1/2, \pm 4, \ldots)$. Two configurations both lead to (14): they are $(-2,x-1/2,-4,\ldots)$ with $x-1/2 = -1$ and $(-2,x-1/2,4)$ with $x-1/2 = 0$;
\item[(Fd)] the sequence starts as $(2,q,-2,x-1/2, \pm 4)$; we get two configurations exactly as before, which lead to (17).
\end{itemize}

The portion formed by $G_k$ and $B$ is treated analogously. We turn to a portion involving $G_i$ and $G_{i+1}$ as in Fig.~\ref{GG:fig}. We analyze each case:

\begin{figure}
\begin{center}
\includegraphics[width = 10 cm]{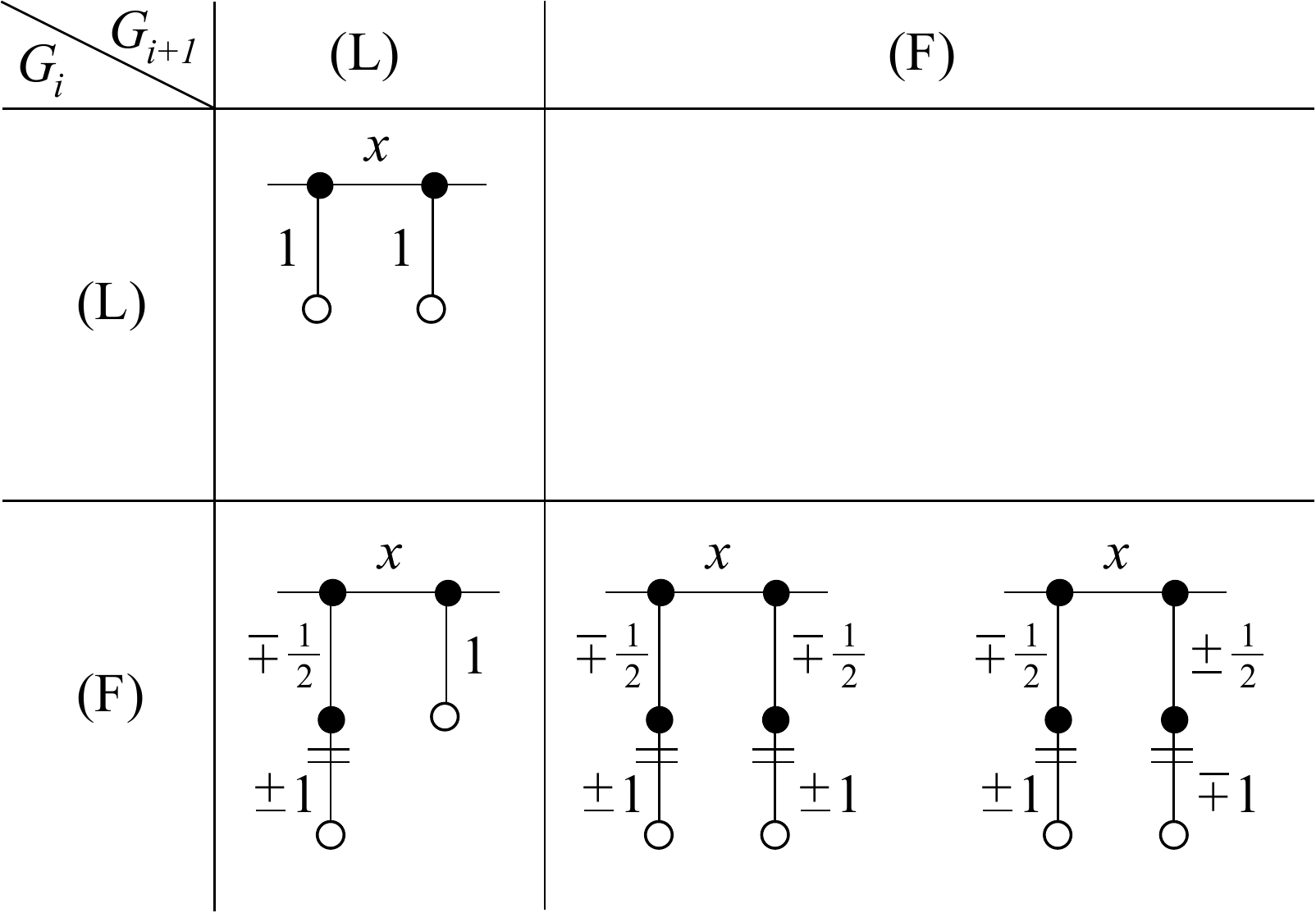}
\nota{Portions obtained as the union of $G_i$ and $G_{i+1}$.}
\label{GG:fig}
\end{center}
\end{figure}

\begin{itemize}
\item[(LL)] the sequence contains $(\ldots, x+2, \ldots)$. If $|x+2|\leqslant 1$ we get (3) or (4);
\item[(FL)] the sequence contains $(\ldots,\pm 4, x+1, \ldots)$. If $x+1 = 0$ we get (8); if $(\ldots, \pm 4, x+1, \ldots)$ equals $(\ldots, 4, 1, \ldots)$ or $(\ldots, -4, -1, \ldots)$ we get (9);
\item[(FF)] the sequence contains $(\ldots, \pm 4, x , \pm 4,\ldots)$ or $(\ldots, \pm 4, x, \mp 4,\ldots)$. In the second case, if $x=0$ we get (10).
\end{itemize}

We are left to consider the presence of flat leaves. These do not contribute to the plumbing line $(e_1,\dots, e_n)$: we therefore conclude that the branch contains a portion of those already listed, plus maybe some additional flat leaves. 

In all the portions found, such leaves may be slid away by using the move in Fig.~\ref{move_leaf:fig}-(2), except when the branch is very small: this happens in cases (12), (15), (18), (19), and (21). In all but the last case, the branch contains a portion of type (20). In the last case, it contains a portion of type (22). 
\end{proof}

Neumann and Weintraub \cite{Neu} used Lemma \ref{easy:plumbing:lemma} to simplify the plumbing line, via a move that eliminates the sphere with small Euler number. Here we do the same. As the following shows, all the portions listed in Fig.~\ref{local:fig} may be simplified.

\begin{figure}
\begin{center}
\includegraphics[width = 12.5 cm]{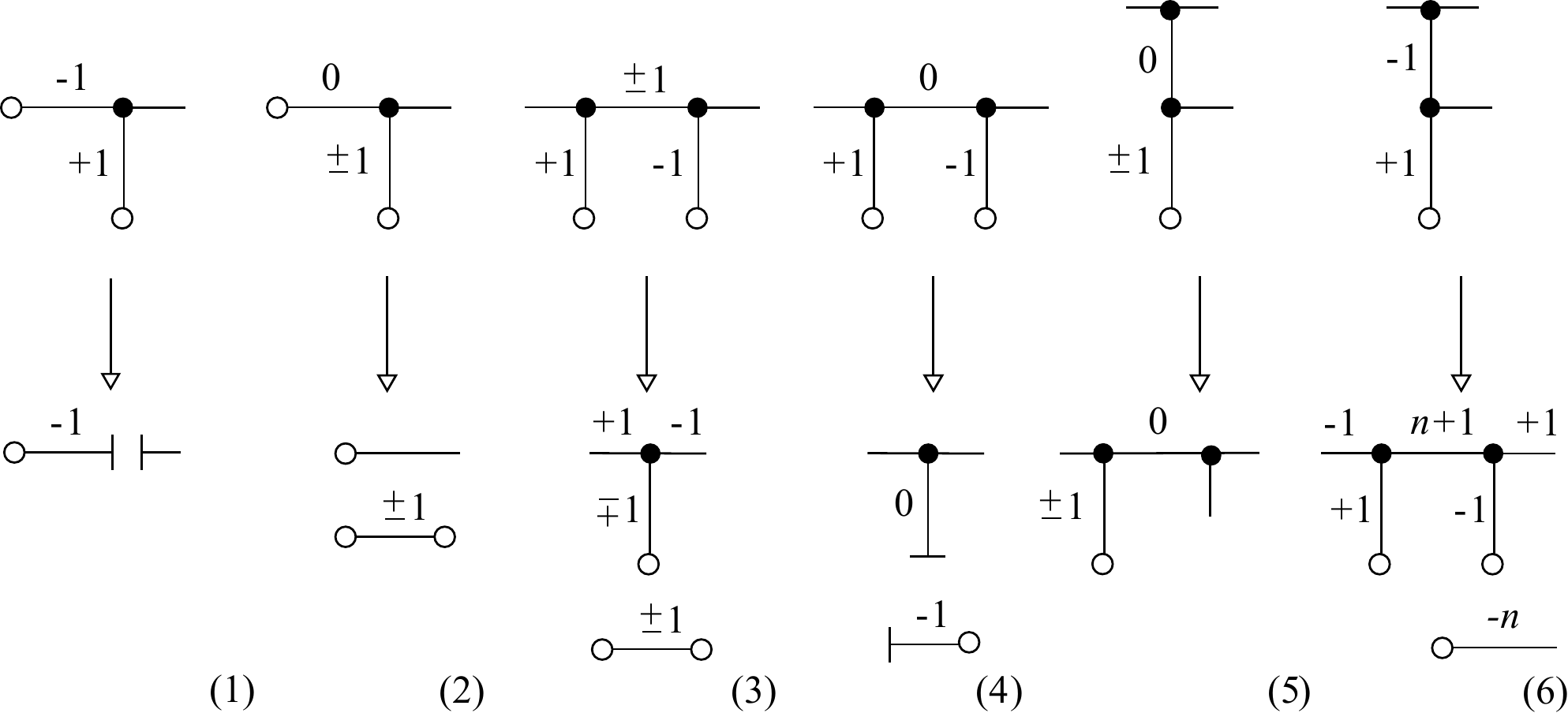}
\nota{Each of these moves transforms a shadow (described via a decorated tree with levels) of a block $(M,L)$ into a shadow of a block $(M',L')$. In (5) we have $(M,L)\isom(M',L')$. In (1) and (4) the block $(M,L)$ is an assembling of $(M',L')$. In (2), (3), (6) the block $(M,L)$ is a connected sum of $(M',L')$.}
\label{simplify:fig}
\end{center}
\end{figure}

\begin{figure}
\begin{center}
\includegraphics[width = 11.5 cm]{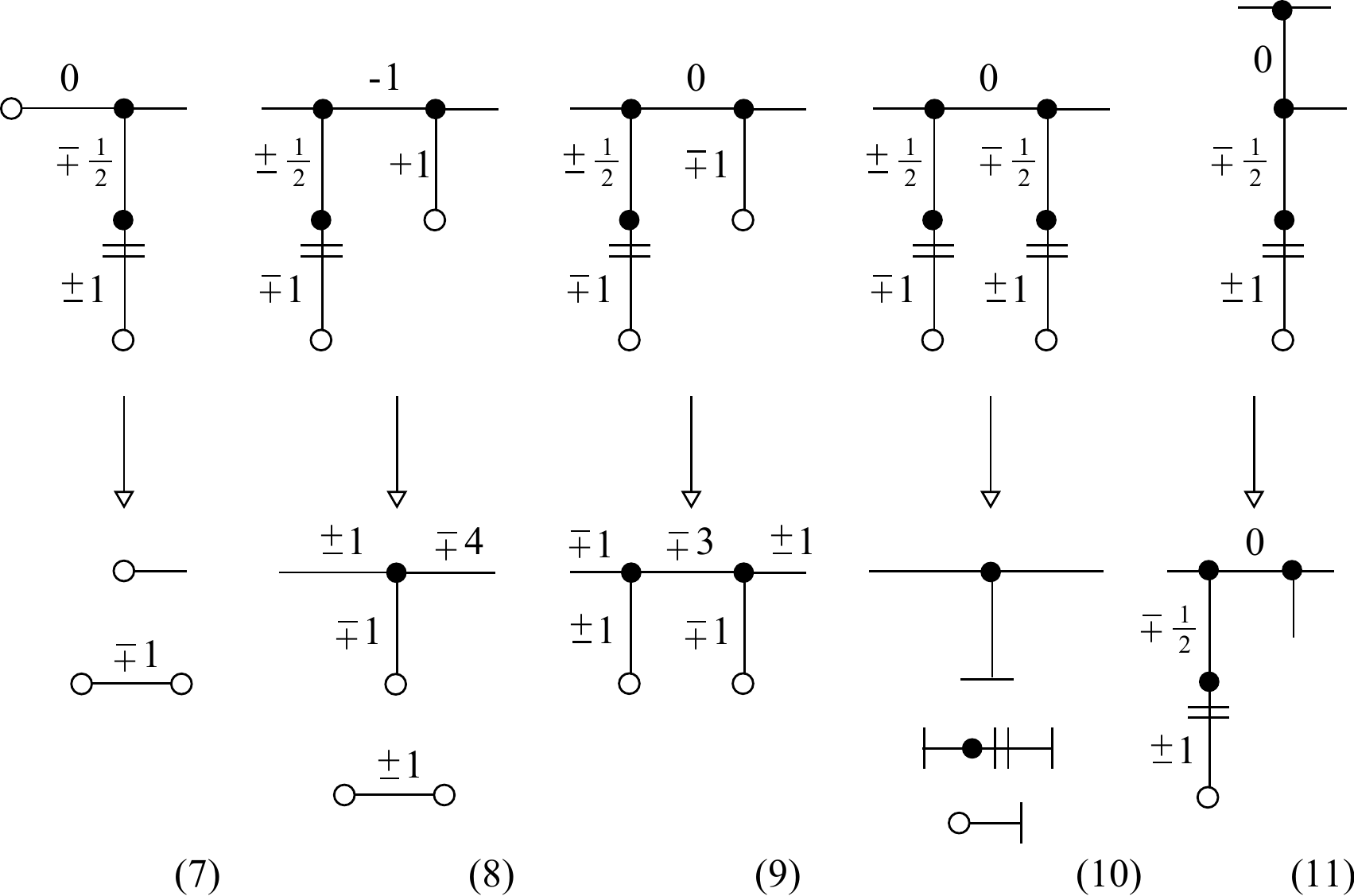}
\nota{Each of these moves transforms a shadow (described via a decorated tree with levels) of a block $(M,L)$ into a shadow of a block $(M',L')$. In (9) and (11) we have $(M,L)\isom (M',L')$. In (7) and (8)the block $(M,L)$ is a connected sum of $(M',L')$. In (10) it is an assembling of $(M,L)$.}
\label{simplify2:fig}
\end{center}
\end{figure}

\begin{figure}
\begin{center}
\includegraphics[width = 12.5 cm]{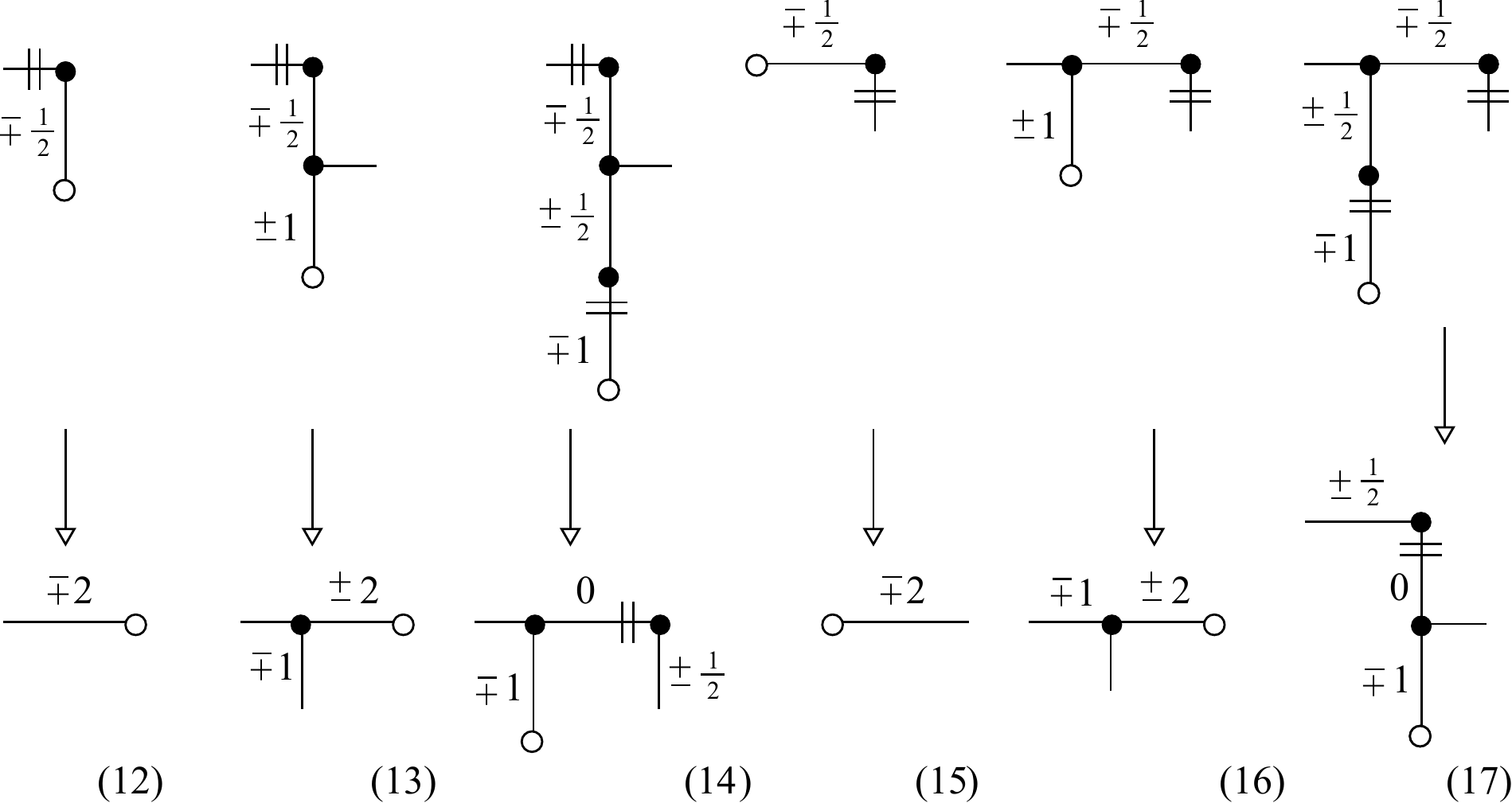}
\nota{Each of these moves transforms a shadow (described via a decorated tree with levels) of a block $(M,L)$ into a shadow of the same block.}
\label{simplify3:fig}
\end{center}
\end{figure}

\begin{figure}
\begin{center}
\includegraphics[width = 9 cm]{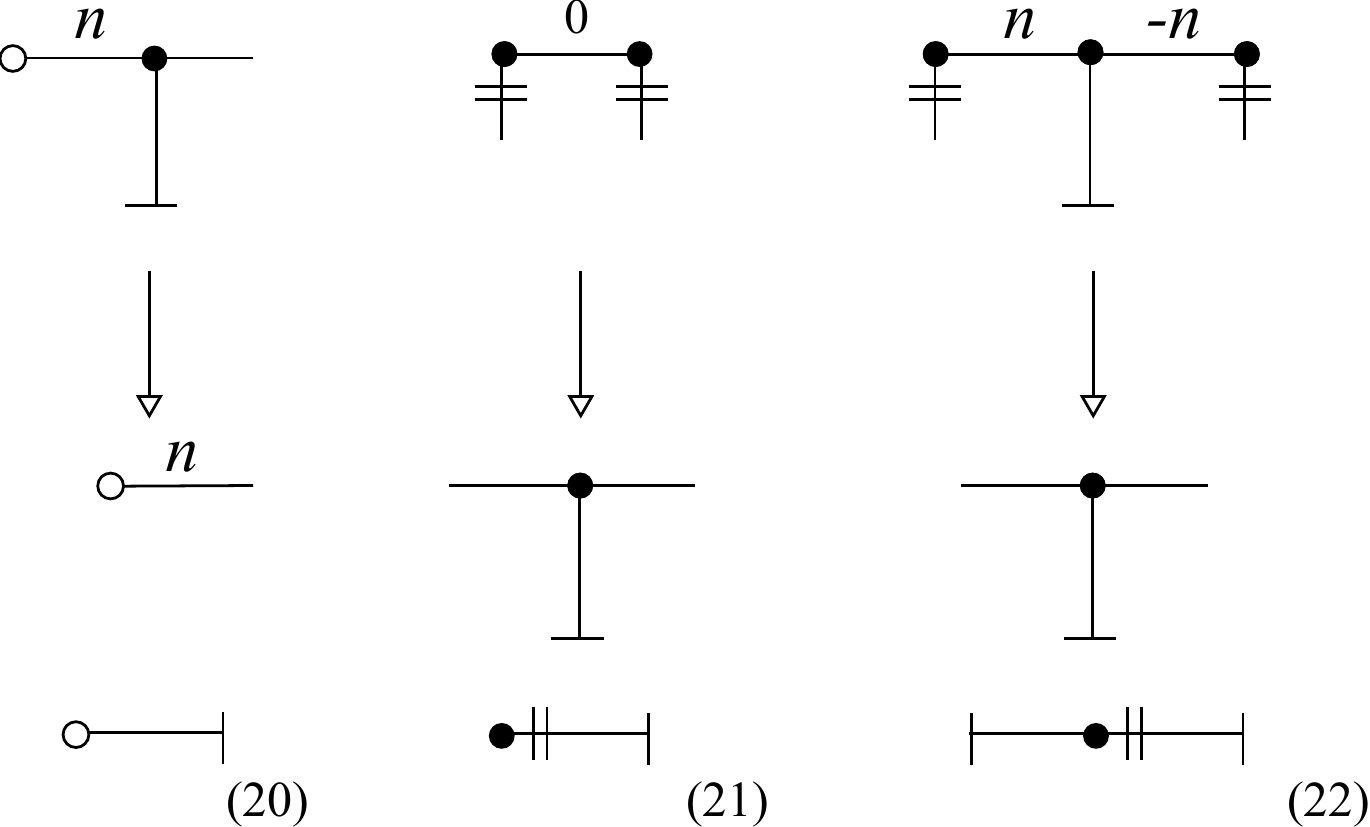}
\nota{Each of these moves transforms a shadow (described via a decorated tree with levels) of a block $(M,L)$ into a shadow of a block $(M',L')$. In (21) and (22) the block $(M,L)$ is an assembling of $(M',L')$. In (20) it is a connected sum of $(M',L')$.}
\label{simplify4:fig}
\end{center}
\end{figure}

\begin{prop} \label{simplify:prop}
Let $T$ be a decorated tree with levels encoding a shadow $X$ of a block $(M,L)$. Each of the moves in Figg.~\ref{simplify:fig}, \ref{simplify2:fig}, \ref{simplify3:fig}, and \ref{simplify4:fig} transforms $T$ into a new tree $T'$ with levels encoding a shadow $X'$ of some block $(M',L')$. The block $(M,L)$ is homeomorphic to $(M',L')$, or obtained from it via one assembling or connected sum.
\end{prop}
\begin{proof}
Move (1) is the inverse of Fig.~\ref{sum_ass:fig}-(2), with one ($-1$)-gleamed \includegraphics[width = 0.6 cm]{1.pdf} attached on the left and some moves from Fig.~\ref{mosse_innocue:fig}. Move (2) is the inverse of Fig.~\ref{sum_ass:fig}-(1). Move (3) is Fig.~\ref{thickening:fig}-(2) followed by the inverse of Fig.~\ref{sum_ass:fig}-(1). Move (4) is Fig.~\ref{thickening:fig}-(1) followed by (1). Move (5) is Fig.~\ref{thickening:fig}-(1). 

\begin{figure}
\begin{center}
\includegraphics[width = 12 cm]{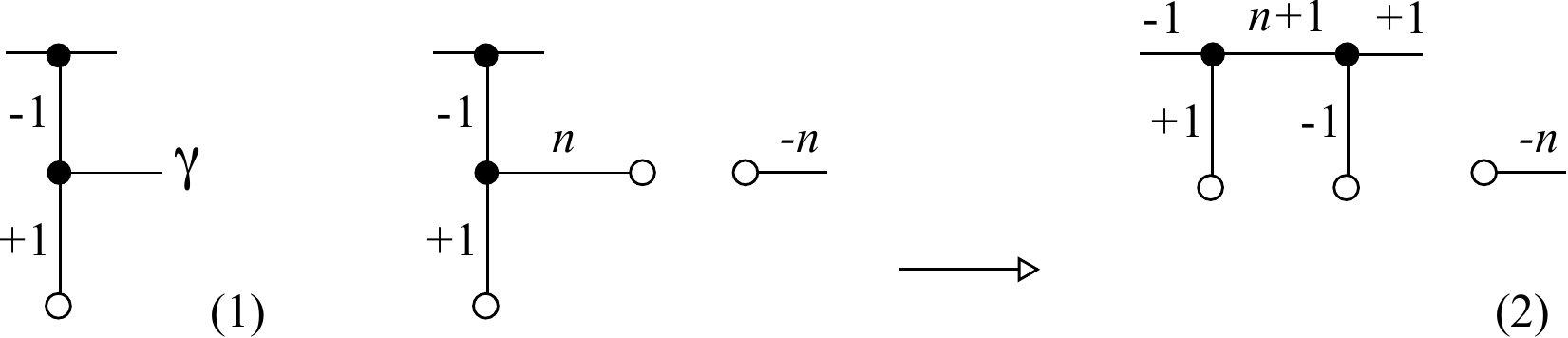}
\nota{An intermediate step for Fig.~\ref{simplify:fig}-(6)}
\label{simplify_proof:fig}
\end{center}
\end{figure}

In move (6), consider the simple closed curve $\gamma$ determined by the edge in Fig.~\ref{simplify_proof:fig}-(1). If we cut the branch as in Fig.~\ref{stacca_generalized:fig} we get a tree with levels of a shadow $X'$ with $\partial N(X') = S^3$. The curve $\gamma$ bounds on the left of this tree a portion equal to the one in Fig.~\ref{simplify:fig}-(1)-top. It is easy to see that the torus over $\gamma$ has a vertical disc over that portion (on the left). Since we are in $S^3$, the torus over $\gamma$ bounds (on the right) another disc which intersects this vertical disc in a point: that is, it is horizontal. Therefore $\gamma$ bounds a horizontal disc on the right. It does so also in the original tree $T$. Since $\gamma$ bounds a horizontal disc we can perform the move in Fig.~\ref{hv:fig}-(2).  The result is as in Fig.~\ref{simplify_proof:fig}-(2)-left. It now suffices to apply Fig.~\ref{thickening:fig}-(2) and we are done. 

Move (7) is the inverse of Fig.~\ref{sum_ass:fig}-(1) and Fig.~\ref{thickening:fig}-(4). Move (8) is the composition of Fig.~\ref{thickening:fig}-(2), Fig.~\ref{thickening:fig}-(3), Fig.~\ref{move_leaf:fig}-(1), and (2). Move (9) is Fig.~\ref{thickening:fig}-(1-3-2). Concerning (10), apply Fig.~\ref{thickening:fig}-(1-6), then (1) and Fig.~\ref{flat:fig}-(4). Move (11) is again Fig.~\ref{thickening:fig}-(1).

Move (12) is Fig.~\ref{thickening:fig}-(4). Move (13) is Fig.~\ref{thickening:fig}-(3). Move (14) is Fig.~\ref{thickening:fig}-(6). Moves (15), (16), and (17) are similar.

Concerning move (20), note that removing a flat vertex corresponds to filling by Fig.~\ref{sum_ass:fig}-(3). The inverse operation is drilling along the curve $\gamma$ determined by the flat vertex. The curve $\gamma$ is null-homotopic since it is contained in a disc. Therefore drilling corresponds to making a connected sum with $S^2\times D^2$, whence move (20). Move (21) is Fig.~\ref{thickening:fig}-(5). The resulting M\"obius strip determines a vertical disc and thus can be deassembled by Proposition \ref{no:36:prop}. Move (22) is a mixure of (20) and (21): we first remove the flat vertex and fill, then perform (21) and drill back the curve, which is now homotopic to the core of the M\"obius strip.
\end{proof}

We finally prove the difficult part of Theorem \ref{main:teo}. We actually prove a more general version, which includes blocks with boundary.
\begin{teo}
Let $X$ be a shadow without vertices of some block $(M,L)$. We have $M = M'\#_h \matCP^2$ for some integer $h$ and some graph manifold $M'$ generated by $\calS_0$. 
\end{teo}
\begin{proof}
By Corollary \ref{leveled:cor}, we may suppose that $(M,L)$ has a shadow encoded by a decorated tree $T$ with levels. 

We prove our theorem by induction on the number of vertices of $T$. By Corollary \ref{flat:cor} we may suppose that every flat vertex in $T$ is nice. We may also suppose that every bad branch is non-reducible (otherwise we may simplify it by Proposition \ref{reducible:prop} and decrease the number of vertices). We can now apply Proposition \ref{contains:lemma} to ensure that the tree contains one of the 22 portions listed in Fig.~\ref{local:fig}. If the portion is (18) or (19), the shadow $X$ is a sphere with gleam $\pm 1$ or $0$, and $M$ is respectively $\pm \matCP^2$ or $S^4$. Otherwise, the portion may be simplified by Proposition \ref{simplify:prop} and we conclude by our induction hypothesis.

More precisely, in all cases except (5), (6), and (11) the number of non-flat vertices decreases. In case (5), (6), (11) the number of non-flat vertices may remain unchanged: however, there 
can be only finitely many such moves, since they strictly decrease the levels of some vertices (and leave the levels of the other vertices unchanged).
\end{proof}

\end{document}